\documentclass[11pt, a4paper, reqno]{article}
\usepackage{graphicx,bm} 
\usepackage{authblk}
\usepackage{times,cite}
\usepackage[utf8]{inputenc}
\usepackage{hyperref}
\hypersetup{
  colorlinks,
  linkcolor={violet},
  citecolor={brown},
  urlcolor={gray}
}
\usepackage{pgfplots}
\pgfplotsset{compat=newest}

\usetikzlibrary{positioning, shapes, arrows.meta}

\usepackage{amsmath}
\usepackage{amsfonts,mathrsfs}
\usepackage{amsthm}
\usepackage{amsfonts}
\usepackage{amssymb}
\usepackage{stmaryrd}
\usepackage{epic}
\usepackage{eepic,epsfig}
\usepackage{graphics}
\usepackage{epsfig,multicol,bbm}
\usepackage{braket}
\usepackage{subcaption}
\usepackage{verbatim}
\usepackage{lscape,longtable}
\usepackage{subfiles}
\usepackage{pdflscape}

\usepackage{enumitem}
\usepackage{mathtools}
\usepackage{stmaryrd}

\usepackage{tikz-cd}
\usepackage{tikz}
\usetikzlibrary{decorations.pathmorphing}
\usetikzlibrary{fit,shapes.geometric}
\usepackage{xcolor}
\usepackage{environ}

\usepackage[left=2.5cm,right=2.5cm,top=3cm,bottom=4cm]{geometry}

\renewcommand{\=}{\; = \;}

\newcommand\CC{\mathbb{C}}
\newcommand\HH{\mathbb H}
\newcommand\RR{\mathbb R}
\newcommand\QQ{\mathbb{Q}}
\newcommand\ZZ{\mathbb{Z}}

\newcommand\calL{\mathcal{L}}

\newcommand{\re}{{\rm e}}
\newcommand{\ri}{{\rm i}}

\def\IZ{{\mathbb Z}}
\def\IR{{\mathbb R}}
\def\CC{{\mathbb C}}
\def\IH{{\mathbb H}}

\def\IP{{\mathbb P}}

\def\frakg{\mathfrak{g}}
\def\frakf{\mathfrak{f}}
\def\tfrakg{\tilde{\mathfrak{g}}}
\def\tfrakf{\tilde{\mathfrak{f}}}

\newcommand{\CA}{{\cal A}}
\newcommand{\CB}{{\cal B}}

\newcommand{\CL}{{\cal L}}
\newcommand{\CCL}{{\mathscr{L}}}

\newcommand{\CS}{{\cal S}}

\newcommand{\be}{\begin{equation}}
\newcommand{\ee}{\end{equation}}
\newcommand{\ba}{\begin{aligned}}
\newcommand{\ea}{\end{aligned}}

\newcommand{\borel}{\mathcal{B}}

\newcommand{\Jodd}{J^{{\rm odd}}}
\newcommand{\Jeven}{J^{{\rm even}}}

\newcommand{\Om}{\boldsymbol{\Omega}}
\newcommand{\OmBor}{\Omega^{\mathrm{Borel}}}

\newcommand{\tEI}{\widetilde{\mathsf{EI}}}
\newcommand{\bEI}{\boldsymbol{\mathsf{EI}}}
\newcommand{\tbEI}{\widetilde{\boldsymbol{\mathsf{EI}}}}

\newtheorem{thm}{Theorem}[section]
\newtheorem{prop}[thm]{Proposition}
\newtheorem{lemma}[thm]{Lemma}

\newtheorem{dfn}{Definition}[section]
\newtheorem{con}{Conjecture}
\newtheorem*{con*}{Conjecture}
\newtheorem{rmk}{Remark}[section]

\title{\huge{Modular resurgent structures for vectors}}

\author[$a$, $b$]{Miranda C. N. Cheng}
\affil[$a$]{\normalsize\textit{Institute of Physics and Korteweg-de Vries Institute for Mathematics, University of Amsterdam, Science Park 904, 1098 XH Amsterdam, Netherlands}}
\affil[$b$]{\normalsize\textit{Institute for Mathematics, Academia Sinica, Taipei, Taiwan}}
\author[$c$]{Ioana Coman}
\affil[$c$]{\normalsize\textit{School of Mathematics and Maxwell Institute for Mathematical Sciences, University of Edinburgh, UK}}
\author[$d$]{Veronica Fantini}
\affil[$d$]{\normalsize\textit{LMO, Université Paris-Saclay, 91405 Orsay, France}}
\author[$e$]{Claudia Rella}
\affil[$e$]{\normalsize\textit{Institut des Hautes \'Etudes Scientifiques, 91440 Bures-sur-Yvette, France}}

\date{}

\begin{document}

\maketitle

\begin{abstract}
Building on prior results~\cite{FR24-2}, we introduce vector-valued modular resurgent series, whose components exhibit a single infinite tower of singularities in the Borel plane, trivial secondary resurgent series, and Stokes constants given by linear combinations of the coefficients of a vector of $L$-functions. 
We extend the paradigm of modular resurgence to this setting, emphasizing the role of the Stokes constants and the interplay between the associated vectors of $q$-series and Dirichlet series, and describing the resulting symmetry relating canonical pairs of vector-valued modular resurgent series.
Moreover, we conjecture that certain vectors of $q$-series with modular resurgent asymptotics are vector-valued quantum modular forms and can be reconstructed via median resummation. 
Finally, we show that vectors of $q$-Pochhammer symbols, previously considered in~\cite{FR25}, and Eichler integrals of vector-valued modular forms of weight $1/2$ and $3/2$ can be studied within the framework of vector-valued modular resurgence. While the first case amounts to a convenient repackaging of scalar modular resurgence, the second involves a non-trivial representation of the modular group and therefore illustrates the necessity of the vector-valued framework; we establish its modular resurgent structure in general, and work it out in full detail for the unary theta series, whose Eichler integrals are the false theta functions of quantum topology.
\end{abstract}

\begingroup
\renewcommand{\thefootnote}{}
\footnotetext{\textit{E-mails:} 
\href{mailto:mcheng@as.edu.tw}{\texttt{mcheng@as.edu.tw}},  
\href{mailto:icomanl@ed.ac.uk}{\texttt{icomanl@ed.ac.uk}},  
\href{mailto:veronica.fantini@universite-paris-saclay.fr}{\texttt{veronica.fantini@universite-paris-saclay.fr}},  
\href{mailto:rella@ihes.fr}{\texttt{rella@ihes.fr}}}
\endgroup

\tableofcontents

\section{Introduction}

The asymptotic expansion of a function at a singular point is often divergent, and its divergence is informative. The factorial growth of the coefficients reflects the location and strength of the singularities of the Borel transform, which in turn encode the exponentially small, non-perturbative contributions that the asymptotic series alone cannot capture. The framework of \emph{resurgence}, due to \'Ecalle~\cite{EcalleI, diver-book}, makes this relationship precise: it organises the singular structure of the Borel transform into a collection of analytic data--the \emph{Stokes constants} and the resurgent series at each singularity--which control the discontinuities of the Borel--Laplace sums across Stokes rays and constrain the global analytic structure of the function reconstructed from the asymptotics.

A completely different structural principle, that of modularity, organises many of the $q$-series that arise as characters of (non-rational) vertex operator algebras, quantum topological invariants, and partition functions of quantum field theories, among others. 
While resurgence extracts the analytic structure of a function $f$, modularity captures its hidden symmetries. Sometimes, $f$ is invariant when transformed under a subgroup $\Gamma\subseteq\mathsf{SL}_2(\IZ)$ only up to a holomorphic correction: classical modularity \emph{fails}, but it does so in a controlled way. Zagier's notion of a (holomorphic) quantum modular form~\cite{ZagierI, ZagierII} formalises this controlled failure as a cocycle condition on a difference function $h_\gamma[f]$, $\gamma \in \Gamma$.

Interesting examples exhibiting both non-trivial quantum modular and resurgent structures prompt the question: exactly when and how are the two bound together? 
The framework of \emph{modular resurgence}, introduced by the last two authors in~\cite{FR24-2} (see also~\cite{FR24-1, FR25}), gives an answer in a class of examples. A Gevrey-1 asymptotic series $\tfrakg$ is said to be a modular resurgent structure if (i) its Borel transform has a single tower of simple singularities at the integer multiples of a fixed constant $\CA\in\CC$, (ii) the resurgent series at each singularity is the constant function given by a Stokes constant, and (iii) the resulting Dirichlet series of Stokes constants admits meromorphic continuation and obeys a functional equation.
Under these hypotheses, the generating $q$-series $\frakf$ of the Stokes constants is paired with a companion $q$-series $\frakg$, obtained as the inverse Mellin transform of the partner Dirichlet series provided by the functional equation, in a closed diagram of resurgent and Mellin operations: each member of the pair encodes the non-perturbative content of the other. Conjecturally, the $q$-series $\frakg$ is reconstructed by the median resummation of $\tfrakg$ and is a quantum modular form~\cite[Conjs.~1 and~2]{FR24-2}. See the diagram in Eq.~\eqref{diag:resurgence-L funct}. 
As such, a modular resurgent series (MRS) is characterized by special analytic, number-theoretic, and quantum modular properties that make it a particularly enticing object of study.

At the same time, modular forms, together with their generalizations such as mock or quantum modular forms, 
often acquire their most natural formulation in a vector-valued setting. Theta functions attached to even lattices, characters of rational vertex operator algebras, and mock theta functions and their completions, all transform as a vector under $\mathsf{SL}_2(\IZ)$ or one of its finite-index subgroups. Moreover, there is often a physical or geometrical reason why this vector-valued structure must be present. 
The corresponding multiplier system encodes representation-theoretic data of substantial interest: the modular $S$-matrix of a rational CFT, the underlying modular tensor category, the Weil representation attached to a discriminant form, the representation of the metaplectic cover of $\mathsf{SL}_2(\IZ)$ associated to theta series of half-integer weight. In each of these settings, scalar components are projections of an inherently vector-valued object, and structural features such as the $S$-matrix are opaque until the full vector is revealed. 

It is therefore natural to seek a vector-valued enhancement of the modular resurgence paradigm. Indeed, a compelling motivation arises from examples that escape the scalar definition but get accommodated naturally in a vector-valued framework. 
\emph{Eichler integrals of unary theta series} ~\cite{eichler1957verallgemeinerung, shimura1959integrales, LawrenceZagier99} are a family of $q$-series featured frequently in the study of vertex algebras, quantum topology, and quantum field theories. They are building blocks of the Witten--Reshetikhin--Turaev invariants of certain Seifert-fibred three-manifolds~\cite{LawrenceZagier99, HIKAMI_2005, Gukov:2017kmk, Cheng:2018vpl, Cheng:2024vou}, they appear in the colored Jones polynomials of alternating knots~\cite{Garoufalidis:2011np, BeirneOsburn}, and they enter the characters of modules of the $(1,N)$- and $(N,N')$-singlet vertex algebras~\cite{Adamovic:2007er, BringmannMilas15}. Clarifying their modular resurgent structure has therefore potential consequences for non-perturbative Chern--Simons theory with complex gauge group, for resurgent invariants of three- and four-manifolds, and for the interplay between mock modularity, defects, and quantum topology. As $q$-series with $q=\re^{2\pi\ri\tau}$, they are given by 
\be \label{eq:EI-intro}
{\sf EI}[\theta^{(\nu)}_{N,k}](\tau)=\sum_{n\in\ZZ}\text{sgn}(n)\,n^{1-\nu}\,\delta_{k,n\,(2N)}\,q^{\frac{n^2}{4N}}\,,\quad N \in \ZZ, \; k\in\IZ/2N\,,
\ee
where $\delta_{k,n\,(2N)}$ is the congruence delta function and $\nu\in\{0,1\}$. They transform as vector-valued holomorphic quantum modular forms of half-integer weight under $\mathsf{SL}_2(\IZ)$~\cite{BR, GO}, and their resurgent structure inherits this vector-valued nature: the perturbative coefficients of ${\sf EI}[\theta^{(\nu)}_{N,k}]$ can be expressed as linear combinations across the index $k$, and no individual component admits an isolated modular resurgent description in the sense of~\cite{FR24-2}. 
They, however, fit into the vector-valued enhancement of modular resurgence introduced in this paper. Moreover, as we show, so does the Eichler integral of \emph{any} vector-valued modular form of weight $1/2$ or $3/2$, the unary theta series providing the fully worked example. The resurgence of Eichler integrals of unary theta series has been recently investigated from several complementary directions, mostly focused on how the false theta function is extended into the lower-half plane~\cite{Costin:2023kla, Costin:Mock26, adams2025orientationreversalchernsimonsnatural, GO, BR, CCR-ongoing, CCR-ongoing2}; we point out the relation to our results in Section~\ref{sec:eichler-setting}.

A second family of natural examples is built from vectors of $q$-Pochhammer symbols. The $q$-series
\be \label{eq:fk-gk-intro}
f_{N,k}(\tau):=\log(\zeta_N^k;q)_\infty\,,\quad g_{N,k}(\tau):=\log(q^k;q^N)_\infty\,,\quad N\in\IZ_{\geq2},\; k\in G_N\,,
\ee
with $\zeta_N=\re^{2\pi\ri/N}$ and $q=\re^{2\pi\ri\tau}$, were studied in~\cite{FR25} from the perspective of (scalar) modular resurgence. Their asymptotic expansions as $\tau\to0^+$ exhibit a tower of Borel-plane simple poles, but their Stokes constants are not coefficients of a single $L$-function: their Dirichlet series decompose as linear combinations of products of Dirichlet $L$-functions and the Riemann zeta function. While it is possible for the scalar framework to accommodate them after a change of basis, the vector-valued framework accommodates them, albeit in a somewhat trivial way, in its native natural basis.  

\paragraph{Results.} In this paper, we introduce the notion of a \emph{vector-valued modular resurgent series} (Definition~\ref{def:vv-MRS}): a vector $(\tfrakg_1,\ldots,\tfrakg_r)$ of Gevrey-1 asymptotic series whose components share a single tower of Borel-plane singularities, have trivial secondary resurgent series, and whose Stokes constants generate a vector of Dirichlet series $(\CL_1,\ldots,\CL_r)$ with polynomial-growth coefficients admitting meromorphic continuation through a linear decomposition in a common set of $L$-functions.
We extend the modular resurgence paradigm to this setting, where companion vectors of $q$-series with modular resurgent asymptotics are paired by the cocycle action of the modular $S$-matrix and with the functional equation of a single $L$-function now replaced by the meromorphic continuation of a vector of Dirichlet series.
We conjecture that a vector-valued modular resurgent expansion encodes two non-trivial properties of the corresponding original vector of $q$-series from the paradigm. The first is \emph{summability}: under a mild hypothesis on the Mellin transforms of the components, each $q$-series is reconstructed by the median resummation of its asymptotic expansion (Conjecture~\ref{conj:vvMR-conj1}). The second is \emph{quantum modularity}: the vector of $q$-series is a holomorphic vector-valued quantum modular form for a subgroup $\Gamma\subseteq\mathsf{SL}_2(\IZ)$ with multiplier $\Om:\Gamma\to\mathsf{GL}_r(\CC)$ (Conjecture~\ref{conj:vvMR-conj2}).

We then establish the paradigm for the two families of examples described above: for the vectors of $q$-Pochhammer symbols of Section~\ref{sec: qPochh-vv}, both conjectures are verified building on the scalar results of~\cite{FR25}; for the Eichler integrals, we prove both conjectures not only for unary theta series (Section~\ref{sec:eichler}) but for the Eichler integral of an arbitrary vector-valued cusp form of weight $1/2$ or $3/2$ for $\mathsf{Mp}_2(\IZ)$ (Section~\ref{sec:beyond-unary}), of which the unary case is the fully worked example. A central structural observation regarding our vector-valued framework is the following. The behaviour of a vector-valued MRS is governed by the modular group representation $\Om: \Gamma \to \mathsf{GL}_r(\CC)$ in Conjecture~\ref{conj:vvMR-conj2}. When $\Om$ is trivial, the vector with the modular resurgent structure of Definition~\ref{def:vv-MRS} can be diagonalised, with each component being an independent (scalar) MRS satisfying the original scalar paradigm of~\cite{FR24-2}, and the vector formalism becomes merely a convenient re-packaging of scalar results. When $\Om$ is non-trivial, no rotation diagonalises the modular action, and the vector-valued framework is genuinely essential. Our two examples illustrate these two scenarios: $q$-Pochhammer symbols (Section~\ref{sec: qPochh-vv}) have trivial multiplier and recast the scalar results of~\cite{FR25} in the new language; Eichler integrals of unary theta series (Section~\ref{sec:eichler}) are quintessential vector-valued quantum modular forms transforming with non-trivial multipliers and constitute the detailed worked example of the general weight-$1/2$ and $3/2$ results of Section~\ref{sec:beyond-unary}.
In the latter case, the paradigm closes in an interesting way: the matrix intertwining the pairing of asymptotic and resurgent data coincides with the modular $S$-matrix of the underlying unary theta series, and the meromorphic continuation of the corresponding Dirichlet vector is induced by classical modularity. In this sense, vector-valued modular resurgence makes manifest the role of modular transformations as the organising principle behind the resurgent structure.

In a slightly different direction, we show in Section~\ref{sec:bimodular-new} that the bimodular completions of Eichler integrals in the sense of Bringmann--Nazaroglu~\cite{Bringmann2019} exhibit a fundamentally different resurgent behaviour; the poles disappear upon the introduction of the second modular variable due to cancellations of the residues, reflecting the fact that these functions, defined on $\IH\times\IH$, display modular symmetry under the diagonal modular transformations. 

\paragraph{Organisation.} Section~\ref{sec:mod-res-vect} reviews the theories of resurgence and quantum modularity, recalls the scalar modular resurgence framework of~\cite{FR24-2}, and introduces its vector-valued enhancement, including the precise statements of the paradigm and conjectures above. Section~\ref{sec: qPochh-vv} revisits and augments the example of $q$-Pochhammer symbols, exhibiting its vector-valued modular resurgent behaviour. 
Section~\ref{sec:eichler} develops the example of Eichler integrals of unary theta series in detail: we compute their resurgent and quantum modular features, verify the paradigm and conjectures of Section~\ref{sec:mod-res-vect}, discuss bimodular completions, and outline the extension to Eichler integrals of general vector-valued modular forms of weight $1/2$ and $3/2$ (Section~\ref{sec:beyond-unary}). 
Conclusions and open problems are collected in Section~\ref{sec:conclusions}.

\section{Modular resurgence for vectors}\label{sec:mod-res-vect}

In this section, we review and extend the framework of modular resurgence initially presented by the last two authors in~\cite{FR24-2}. 
First, we recall the key notions of the theory of resurgence of \'Ecalle~\cite{EcalleI, diver-book} and of the theory of quantum modular forms of Zagier~\cite{ZagierI, ZagierII} and outline how they enter the original program of modular resurgence of~\cite{FR24-2}. 
Then, we introduce and develop a generalization of modular resurgence to the vector-valued setting. In particular, vector-valued modular resurgent series are shown to satisfy a
variation of the modular resurgence paradigm and conjectures.

\subsection{Resurgence of asymptotic series}

Let $\tau$ be a formal variable and  
\be \label{eq: phi}
\phi(\tau) = \sum_{n=0}^{\infty} a_n \tau^{n+1} \in \CC[\![\tau]\!] 
\ee
a Gevrey-1 asymptotic series, namely such that its coefficients are bounded by
\be
|a_n|\le \CA^{-n} n! \, , \quad n \geq 0 \, ,
\ee
for some constant $\CA \in \IR_{>0}$. The Borel transform of $\phi$, which is defined as\footnote{Alternative conventions for the Borel transform of a Gevrey-1 asymptotic series appear in the literature, \emph{e.g.}, it is sometimes useful to adopt the definition $\mathcal{B}[\phi](\zeta) = \sum_{n=0}^\infty \frac{a_n}{\Gamma(n+2)} \zeta^{n+1}$. }
\be \label{eq: phihat}
\borel[\phi](\zeta) := \sum_{n=0}^{\infty} \frac{a_n}{\Gamma(n+1)} \zeta^{n} \in  \CC\{\zeta\} \, ,
\ee
is a convergent power series in a disc centred at $\zeta=0$ of radius $\CA$ in the complex plane of the new formal variable $\zeta$, known as Borel plane. When $\borel[\phi](\zeta)$ can be analytically continued to the whole Borel plane to a (possibly multi-valued) function, we say that the asymptotic series $\phi$ is resurgent~\cite{EcalleI}. 
\begin{dfn}
    A Gevrey-1 asymptotic series $\phi$ is \emph{resurgent} if its Borel transform can be endlessly analytically continued. Namely, for every $L>0$, there is a finite set of points $\OmBor_L \subset \CC$ such that $\borel[\phi](\zeta)$ can be analytically continued along any path of length at most $L$ that starts from a point $\zeta=\eta_0$ with $|\eta_0|<\CA$ and avoids $\OmBor_L$. The set of singularities of $\borel[\phi](\zeta)$ is then $\OmBor:=\bigcup_{L>0}\OmBor_L \subset \CC$. If, additionally, the Borel transform has only simple poles and logarithmic branch points, then $\phi$ is \emph{simple resurgent}.
\end{dfn}

Assume that the asymptotic series in Eq.~\eqref{eq: phi} is simple resurgent and fix an angle $0 \le \theta < 2 \pi$. If the analytically continued Borel transform of $\phi$ grows at most exponentially in an open sector containing the ray $\mathcal{C}_\theta = \re^{\ri \theta} \, \IR_{\ge 0}$, then its Laplace transform at angle $\theta$ 
\be \label{eq: Laplace}
s_{\theta}[\phi](\tau) := \int_{\mathcal{C}_\theta} \re^{-\zeta/\tau} \borel[\phi](\zeta) \, d \zeta 
\ee
is a well-defined analytic function in the half-plane $\Re(\re^{-\ri\theta} \tau)>0$ whose asymptotic expansion at $\tau=0$ reproduces the original formal power series $\phi(\tau)$. We call $s_{\theta}[\phi]$ the Borel--Laplace sum of $\phi$ at angle $\theta$. 
Note that it is discontinuous across the rays corresponding to
\be
\arg(\tau)=\arg(\rho) \, , \quad \rho \in \OmBor \, ,
\ee
where the points $\rho \in \OmBor$ are the locations of the singularities of its Borel transform in the Borel plane.
In particular, the discontinuity across a given ray $\mathcal{C}_{\theta}$ is defined as
\be \label{eq: disc}
\mathrm{disc}_{\theta}[\phi](\tau) := s_{\theta_+}[\phi](\tau) - s_{\theta_-}[\phi](\tau) = \int_{\mathcal{C}_{\theta_+} - \, \mathcal{C}_{\theta_-}} \re^{-\zeta/\tau} \borel[\phi](\zeta) \,  d \zeta \, , 
\ee
where $\theta_{\pm}= \theta \pm \epsilon$ for a fixed $0 < \epsilon \ll 1$.

By Cauchy's integral theorem, the above contour integral reduces to a sum of exponentially small contributions, one from each singularity of the Borel transform in Eq.~\eqref{eq: phihat} with argument $\theta$. Specifically, when the singularity at $\zeta=\rho$ is a simple pole, the local expansion of $\borel[\phi](\zeta)$ around it has the form
\be \label{eq: Stokes0}
\borel[\phi](\zeta) = - \frac{S_{\rho}}{2 \pi \ri (\zeta - \rho)} + \text{regular in $(\zeta-\rho)$} \, ,
\ee
which defines the corresponding \emph{Stokes constant} $S_\rho \in \CC$ and contributes to the integral in Eq.~\eqref{eq: disc} by the exponentially suppressed term $S_\rho \, \re^{-\rho/\tau}$. Summing these contributions, when $\borel[\phi]$ has only simple poles along $\mathcal{C}_\theta$, we obtain
\be \label{eq: Stokes1-poles}
\mathrm{disc}_{\theta}[\phi](\tau) = \sum_{\rho  \in \OmBor_{\theta}} S_{\rho} \, \re^{-\rho/\tau}  \, ,
\ee
where $\OmBor_{\theta} = \{\rho \in \OmBor : \arg(\rho) = \theta\}$ collects the singularities along the given ray.
More generally, when the singularity at $\zeta=\rho$ is not a simple pole, the local expansion of the Borel transform around it produces more than just a constant $S_\rho$---typically, a whole convergent power series in $(\zeta-\rho)$. The Borel--Laplace sum of the inverse Borel transform of the latter then multiplies the factor $\re^{-\rho/\tau}$ in the discontinuity, which leads to the so-called trans-series of $\phi$~\cite{diver-book,M-notes}.

\paragraph{Median resummation.} Sometimes, a Gevrey-1 asymptotic series $\phi$ as in Eq.~\eqref{eq: phi} is produced by perturbatively expanding a known holomorphic function in a given limit. One might then inquire whether the Borel--Laplace sum of $\phi$ for a fixed $0 \le \theta < 2 \pi$ effectively reconstructs the original function and,  if not, which summability technique does. In this work, we will often consider the median resummation of $\phi$ at angle $\theta$, that is, 
\be \label{eq: median}
\mathcal{S}^{\mathrm{med}}_{\theta}[\phi](\tau) := \frac{s_{\theta_+}[\phi](\tau) + s_{\theta_-}[\phi](\tau)}{2} \, ,
\ee
which is a well-defined analytic function in the half-plane $\Re(\re^{-\ri\theta} \tau)>0$ whose asymptotic expansion at $\tau=0$ reproduces $\phi(\tau)$. 
Note that this can be equivalently expressed as
\be \label{eq: median2}
\mathcal{S}^{\mathrm{med}}_{\theta}[\phi](\tau) = s_{\theta_-}[\phi](\tau)+\tfrac{1}{2}\,\mathrm{disc}_{\theta}[\phi](\tau) = s_{\theta_+}[\phi](\tau)-\tfrac{1}{2}\,\mathrm{disc}_{\theta}[\phi](\tau) \, , 
\ee
highlighting the dependence on the non-perturbative contributions given by the discontinuities.

\subsection{Quantum modular forms}\label{sec:QMF-review}

A (classical) modular form is a holomorphic function defined on the complex upper half-plane that is invariant under the action of $\mathsf{SL}_2(\IZ)$ up to a simple automorphy factor. More precisely, consider a holomorphic function $f\colon\IH\to\CC$ and an integer $\omega$. For every $\gamma=\left( \begin{smallmatrix}
    a & b\\
    c & d
\end{smallmatrix} \right)\in\mathsf{SL}_2(\IZ)$, we define the difference function $h_\gamma[{f}]\colon\IH\to\CC$ to be 
\begin{equation} \label{eq: diff-funct}
    h_\gamma[{f}](\tau):= (c \tau+d)^{-\omega}\,f\left(\gamma \tau \right) - f(\tau) \, ,
\end{equation}
where $\gamma\tau:=\tfrac{a\tau+b}{c\tau+d}$ and $(c \tau+d)^{-\omega}$ is the inverse of the automorphy factor.\footnote{The difference function $h_\gamma[{f}]$ in Eq.~\eqref{eq: diff-funct} is the rank-$1$, trivial-multiplier specialization of the slash operator in Eq.~\eqref{eq: cocycle} and thus inherits the additive cocycle identity in Eq.~\eqref{eq: slash-cocycle}.} 
If $h_\gamma[{f}](\tau)=0$ for all $\tau \in \IH$ and all $\gamma \in \mathsf{SL}_2(\IZ)$, then the function $f$ is a  modular form\footnote{Note that $-I \in \mathsf{SL}_2(\IZ)$ forces $f \equiv 0$ for odd $\omega$.} of weight $\omega$. 

Numerous extensions and variations of the notion of modular form have been studied.
For instance, one might assume the holomorphic function $f$ to be vector-valued, consider a finite-index subgroup $\Gamma \subseteq \mathsf{SL}_2(\IZ)$ and a half-integral weight, and appropriately generalize the original automorphy factor to include a multiplier system.
Specifically, we define an automorphy factor of rank $r \in \IZ_{>0}$ and weight $\omega \in\frac{1}{2}\IZ$ for a subgroup $\Gamma\subseteq\mathsf{SL}_2(\IZ)$ to be a function ${\bf j}\colon\IH \times \Gamma\to\mathsf{GL}_r(\CC)$ of the form 
\be \label{eq: j}
{\bf j}(\tau; \gamma):=(c \tau+d)^{\omega} \, \Om(\gamma) \, , \quad \forall \gamma=\left( \begin{smallmatrix}
    a & b\\
    c & d
\end{smallmatrix} \right) \in\Gamma \, , 
\ee
where we choose the principal branch of $(c\tau+d)^{\omega}$ and $\Om$ is an $r$-dimensional multiplier system for $\Gamma$, that is, a map $\Om \colon \Gamma \to \mathsf{GL}_r(\CC)$ for which the function ${\bf j}$ satisfies the cocycle identity\footnote{Equivalently, the map $\Om$ is such that the slash operator $|_\omega$ defined in Eq.~\eqref{eq: slash} is a right action.}
\be \label{eq: cocycle-id}
{\bf j}(\tau; \gamma \gamma') = {\bf j}(\gamma' \tau; \gamma)\, {\bf j}(\tau; \gamma') \, , \quad \forall \gamma, \gamma' \in \Gamma \, .
\ee
For integral $\omega$, the scalar factor $(c\tau+d)^\omega$ is itself a cocycle, and the map $\Om$ is a group homomorphism; for half-integral $\omega$, $(c\tau+d)^\omega$ is a cocycle only up to sign, and $\Om$ is in general a projective representation of $\Gamma$, lifting to a genuine representation $\widetilde{\Om}$ of the preimage $\widetilde{\Gamma}$ of $\Gamma$ in the metaplectic double cover $\mathsf{Mp}_2(\IZ)$.

It is often convenient to repackage the automorphic factor into a rank-$r$ weight-$\omega$ \emph{slash operator}, that is, a right action of $\Gamma$ on vector-valued holomorphic functions ${\bf f}\colon\IH\to\CC^r$ defined by
\be \label{eq: slash}
({\bf f}|_\omega\gamma)(\tau):=(c\tau+d)^{-\omega}\,\Om(\gamma)^{-1}\,{\bf f}(\gamma\tau)\,, \quad \forall \gamma=\left( \begin{smallmatrix}
    a & b\\
    c & d
\end{smallmatrix} \right) \in\Gamma \, ,
\ee
where $\Om$ is the same multiplier system in Eq.~\eqref{eq: j} and we again choose the principal branch of $(c\tau+d)^{-\omega}$. 
For every $\gamma=\left( \begin{smallmatrix}
    a & b\\
    c & d
\end{smallmatrix} \right)\in\Gamma$, we define the \emph{cochain}
$h_\gamma[{\bf f}] \colon \IH \to \CC^r$ via the slash operator $|_\omega$ in Eq.~\eqref{eq: slash} as
\begin{equation} \label{eq: cocycle}
    h_\gamma[{\bf f}](\tau) := ({\bf f}|_\omega\gamma)(\tau) - {\bf f}(\tau) \,,
\end{equation}
which satisfies the additive cocycle identity
\begin{equation} \label{eq: slash-cocycle}
    h_{\gamma\gamma'} = h_\gamma|_\omega\gamma' + h_{\gamma'}\,, \quad \forall \gamma, \gamma' \in \Gamma \, ,
\end{equation}
where, by slight abuse of notation, we denote $h_\gamma = h_\gamma[{\bf f}]$.
If $h_\gamma[{\bf f}] = 0$ for all $\gamma \in \Gamma$, or equivalently ${\bf f}|_\omega\gamma = {\bf f}$, then the function ${\bf f}$ is a vector-valued weight-$\omega$ modular form for $\Gamma$.

Failure of the above transformation equation can, however, occur in a somewhat controlled way. In fact, a remarkable class of modular-type objects is obtained by imposing that the difference function in Eq.~\eqref{eq: cocycle} has strong analyticity properties for all $\gamma \in \Gamma$, although not requiring it to be identically zero~\cite{ZagierI, ZagierII}. 
\begin{dfn} \label{def: vv-QMF}
    A holomorphic function ${\bf f}\colon\IH\to\CC^r$ is a \emph{vector-valued (holomorphic) quantum modular form} for a subgroup $\Gamma\subseteq\mathsf{SL}_2(\IZ)$ of weight $\omega\in\frac{1}{2}\IZ$ and multiplier $\Om \colon \Gamma \to \mathsf{GL}_r(\CC)$ if the function $h_\gamma[{\bf f}]\colon\IH\to\CC^r$ in Eq.~\eqref{eq: cocycle} extends holomorphically to\footnote{Note that $\CC_\gamma = \CC \smallsetminus \left(-\infty; \, -d/c\right]$ when $c>0$ and $\CC_\gamma = \CC \smallsetminus \left[-d/c; \, +\infty \right)$ when $c<0$. Requiring a holomorphic extension to $\CC_\gamma$ therefore amounts to demanding that $h_\gamma[{\bf f}]$ continue analytically across the open interval $\CC_\gamma\cap\IR$ into the lower half-plane.}
\be \label{eq: C_gamma}
\CC_\gamma:=\{\tau\in\CC\colon c \tau+d\in\CC\smallsetminus\IR_{\leq 0}\} \, , \quad \forall \gamma=\left( \begin{smallmatrix}
    a & b\\
    c & d
\end{smallmatrix} \right) \in\Gamma \, .
\ee
\end{dfn}

We note that the additive cocycle property in Eq.~\eqref{eq: slash-cocycle} implies that the holomorphy of $h_\gamma$ required in Definition~\ref{def: vv-QMF} need only be verified on a finite generating set $\{\gamma_i\}_{i \in I}$ of $\Gamma$. Besides, Definition~\ref{def: vv-QMF} can be written analogously for a holomorphic function ${\bf f}\colon\IH_{-}\to\CC^r$, where $\IH_{-}$ denotes the complex lower half-plane. When ${\bf f}$ is scalar-valued, that is, $r=1$ and $\Om$ is trivial, we refer to it as a (holomorphic) quantum modular form. Quantum modular forms of weight zero are also known as quantum modular functions.

\subsection{A primer on modular resurgence}\label{sec: primer}

A complete rigorous classification of quantum modular forms remains elusive. Nevertheless, when such forms exhibit a factorially divergent asymptotic expansion, progress can be made by applying the techniques of resurgence. Along these lines, the last two authors developed the framework of \emph{modular resurgence} in~\cite{FR24-2}. Before describing it, we introduce one last ingredient. 
\begin{dfn} \label{def: Lfunct}
A \emph{Dirichlet series} is a formal series
\begin{equation}\label{def:L-funct-gen}
L(s)=\sum_{m=1}^\infty\frac{A_m}{m^s} \, , 
\end{equation}
where $s$ is a complex variable and $\{A_m\}_{m \in \IZ_{>0}}$ is a sequence of complex numbers. If the series converges absolutely in a right half-plane $\{\Re(s)>\alpha\} \subset \CC$ and admits an Euler product expansion in that region, it is called an \emph{$L$-series}. If, moreover, it admits a meromorphic continuation to a larger domain containing $\{\Re(s)<0\} \subset \CC$, it defines an \emph{$L$-function}.
\end{dfn}

A remarkable class of resurgent asymptotic series, linked to $L$-functions and quantum modular forms through distinctive analytic and arithmetic properties, was identified and studied in~\cite{FR24-2}, where the notion of \emph{modular resurgent series} (MRS) was introduced.
\begin{dfn}[Def.~3.2 in~\cite{FR24-2}]\label{def:modular_res_struct}
A Gevrey-1 asymptotic series $\tfrakg \in \CC \llbracket \tau \rrbracket$ has a \emph{modular resurgent structure} if the following conditions hold.
\begin{enumerate}
    \item The Borel transform $\CB[\tfrakg](\zeta) \in \CC\{\zeta\}$ has a tower of singularities at the locations $\rho_m= \CA m$, $m\in\mathbb{Z}_{\ne 0}$, for some constant $\CA \in \CC$, in Borel plane.
    \item For every $m\in\IZ_{\ne 0}$, the resurgent series at the singularity $\rho_m$ is the constant function $S_m\in\CC$, \emph{i.e.}, the Stokes constant.
    \item The Stokes constants $S_m$, $m\in\IZ_{\ne 0}$, are the coefficients of the $L$-functions
    \be \label{eq: L-pm}
		L_{+}(s)=\sum_{m >0} \frac{S_{m}}{m^s}  \quad \text{ and } \quad L_{-}(s)=-\sum_{m >0} \frac{S_{- m}}{m^s} \, .
     \ee
\end{enumerate}
A Gevrey-1 asymptotic series with a modular resurgent structure is a \emph{modular resurgent series}. 
\end{dfn} 
Apart from the global constant $\CA$ measuring the distance between consecutive singularities in the tower, the MRS in Definition~\ref{def:modular_res_struct} can be alternatively characterized by the $q$-series
\begin{equation} \label{eq:f_Am}
\frakf(\tau)=\begin{cases}
\displaystyle\sum_{m>0}S_m q^m & \mbox{if} \quad \Im(\tau)>0 \,   \\ \\
-\displaystyle\sum_{m>0}S_{-m} q^{-m} & \mbox{if} \quad \Im(\tau)<0 \,  
\end{cases} \, , \quad q= \re^{2\pi \ri \tau} \, , \quad \tau \in \CC\smallsetminus\IR \, ,
\end{equation} 
which is the generating series of the Stokes constants.
In fact, there is a canonical correspondence between the $L$-functions $L_\pm$ and the $q$-series $\frakf$ via the Mellin transform and its inverse. Namely,
\be \label{eq: mellin-formula}
    L_{\pm}(s)= \frac{(2\pi)^{s}}{\Gamma(s)} \int_0^\infty t^{s-1}\, \frakf(\pm \ri t) \, dt \, , \quad \Re(s)>\mbox{max}(0, \alpha) \, ,
\ee
where $\alpha_\pm$ determines the right half-plane of absolute convergence of $L_{\pm}$ and $\alpha=\mbox{max}( \alpha_+,\alpha_-)$\footnote{In all examples studied in this paper, we have $\alpha_+=\alpha_-=\alpha$.}  (cf. Definition \ref{def: Lfunct}), while
\begin{equation} \label{eq: invmellin-formula}
\frakf(\pm \ri t)= \frac{1}{2\pi \ri} \int_{C-\ri\infty}^{C+\ri \infty} (2\pi t)^{-s} \Gamma(s)\, L_{\pm}(s)\, ds   \, , \quad t >0 \, , 
\end{equation}
for fixed $C>\mbox{max}(0, \alpha)$~\cite[Prop.~3.3]{FR24-2}.
Let us now denote by
\be
\tfrakf_\pm(\tau)=\sum_{n=0}^\infty b^{(\pm)}_n \tau^n \in \CC [\![\tau]\!]
\ee
the asymptotic expansions of $\frakf(\tau)$ as $\tau \to 0$ with $\Im(\tau)>0$ and $\Im(\tau)<0$, respectively, and introduce the new asymptotic series
\be \label{eq: def-tfrakf}
\tfrakf(\tau):=\tfrakf_+(\tau)+\tfrakf_-(-\tau) = \sum_{n=0}^\infty b_n \tau^n\, . 
\ee
Then, the perturbative coefficients $b_n$, $n \in \IZ_{\ge 0}$, are given by 
\be \label{eq: coeff-bn-1}
b_n=\frac{(2\pi \ri)^n}{n!} \left(L_+(-n) + (-1)^n L_-(-n)\right) \, ,
\ee
where $L_\pm(-n)$ are defined via meromorphic continuation of the corresponding $L$-functions $L_{\pm}(s)$ to $\{\Re (s)< 0\} \subset \CC$~\cite[Prop.~3.4]{FR24-2}.
Note that the functional equation for $L_{\pm}$ is generally expressed in terms of a (possibly different) pair of $L$-functions $L'_{\pm}(s)$ with $\{\Re(s)>\beta \} \subset \CC$.

Finally, the discontinuities of $\tfrakg$ are given by a Fricke-type transformation of the same $q$-series $\frakf$ in Eq.~\eqref{eq:f_Am} after a suitable change of variable. 
More precisely,  
\begin{subequations} \label{eq: disc-g}
\begin{align}
\mathrm{disc}_{\theta} [\tfrakg](\tau) &= \sum_{m=1}^{\infty} S_m \, \re^{-\CA m /\tau} = \frakf\left(-\tfrac{\CA}{2 \pi \ri \tau}\right)   \, , \quad \tau\in\IH_+ \cap \{\Re ( \re^{-\ri\theta} \tau)>0\} \, ,   \label{eq: disc-up-g}\\
\mathrm{disc}_{\theta+\pi} [\tfrakg] (\tau) &=\sum_{m=1}^{\infty} S_{-m} \, \re^{\CA m /\tau} = -\frakf\left(-\tfrac{\CA}{2 \pi \ri \tau}\right)   \, , \quad \tau\in\IH_- \cap \{\Re ( \re^{-\ri\theta} \tau)>0\} \, , \label{eq: disc-down-g}
\end{align}
\end{subequations}
where $0 \le \theta < \pi$ is the argument of the singularities of $\borel[\tfrakg]$ in the upper half-plane~\cite[Sec.~3.2]{FR24-2}.

\subsubsection*{The paradigm} 

Definition~\ref{def:modular_res_struct} constrains a \emph{formal} power series $\tfrakg$ by fixing the singularity structure of the Borel transform and identifying the Stokes constants as the coefficients of $L$-functions. It does not, on its own, assert that $\tfrakg$ is the asymptotic expansion of an actual function. When it is, the MRS acquires a much richer structure, which we now describe.

The content of Definition~\ref{def:modular_res_struct} already supplies a canonical candidate. The functional equation for $L_\pm$ expresses their meromorphic continuation through a (possibly different) pair of $L$-functions $L'_\pm$, introduced above after Eq.~\eqref{eq: coeff-bn-1}. Since the coefficients of an $L$-function grow at most polynomially, the inverse Mellin transform of $L'_\pm$ converges and defines a holomorphic $q$-series
\begin{equation} \label{eq:g_Am}
\frakg(\tau)=\begin{cases}
\displaystyle\sum_{m>0}R_m q^m & \mbox{if} \quad \Im(\tau)>0 \,   \\ \\
-\displaystyle\sum_{m>0}R_{-m} q^{-m} & \mbox{if} \quad \Im(\tau)<0 \,  
\end{cases} \, , \quad q= \re^{2\pi \ri \tau} \, , \quad \tau \in \CC\smallsetminus\IR \, ,
\end{equation} 
with coefficients $R_m \in \CC$. Namely, the $L$-functions
\be \label{eq: L2-pm}
    L'_{+}(s)=\sum_{m >0} \frac{R_{m}}{m^s} \quad \text{ and } \quad L'_{-}(s)=-\sum_{m >0} \frac{R_{- m}}{m^s} 
\ee
are related to $\frakg$ by the Mellin formulae in Eqs.~\eqref{eq: mellin-formula} and~\eqref{eq: invmellin-formula}. Let us denote by
\be
\tfrakg_\pm(\tau)=\sum_{n=0}^\infty a^{(\pm)}_n \tau^n \in \CC [\![\tau]\!]
\ee
the asymptotic expansions of $\frakg(\tau)$ as $\tau \to 0$ with $\Im(\tau)>0$ and $\Im(\tau)<0$, respectively.

Assume that the MRS $\tfrakg$ of Definition~\ref{def:modular_res_struct} is recovered from these asymptotic expansions as
\be \label{eq: def-tfrakg}
\tfrakg(\tau)=\tfrakg_+(\tau)+\tfrakg_-(-\tau) = \sum_{n=0}^\infty a_n \tau^n \, ,
\ee
and that the Borel transform $\borel[\tfrakg]$ has only simple poles.
When this is the case, the perturbative coefficients $a_n$, $n \in \IZ_{\ge 0}$, are given by
\be \label{eq: coeff-an-2}
a_n=\frac{(2\pi \ri)^n}{n!} \left(L'_+(-n) + (-1)^n L'_-(-n)\right) \, ,
\ee
and two notable consequences follow, describing properties of $(\tfrakf, \frakg, L'_\pm)$ analogous to the ones established in the previous section for $(\tfrakg, \frakf, L_\pm)$.
\begin{itemize}
    \item[1.] The asymptotic series $\tfrakf$ in Eq.~\eqref{eq: def-tfrakf} is an MRS whose Borel transform has only simple poles at the locations $\eta_m=\CA' m$, $m \in \IZ_{\ne 0}$, for some constant $\CA' \in \CC$, and with Stokes constants $R_m$. This is the dual of the structure in Eqs.~\eqref{eq:f_Am}--\eqref{eq: coeff-bn-1}, obtained by applying the construction of~\cite[Sec.~3.1]{FR24-2} to $L'_\pm$ in place of $L_\pm$.
    \item[2.] The discontinuities of $\tfrakf$ are given by the same $q$-series $\frakg$ in Eq.~\eqref{eq:g_Am} after a Fricke-type transformation, as in Eqs.~\eqref{eq: disc-up-g} and~\eqref{eq: disc-down-g}. Namely,
    \begin{subequations} \label{eq: disc-f}
    \begin{align}
        \mathrm{disc}_{\theta'} [\tfrakf](\tau) &= \sum_{m=1}^{\infty} R_m \, \re^{-\CA' m /\tau} = \frakg\left(-\tfrac{\CA'}{2 \pi \ri \tau}\right)   \, , \quad \tau\in\IH_+ \cap \{\Re ( \re^{-\ri\theta'} \tau)>0\} \, ,   \label{eq: disc-up-f}\\
        \mathrm{disc}_{\theta'+\pi} [\tfrakf] (\tau) &=\sum_{m=1}^{\infty} R_{-m} \, \re^{\CA' m /\tau} = -\frakg\left(-\tfrac{\CA'}{2 \pi \ri \tau}\right)   \, , \quad \tau\in\IH_- \cap \{\Re ( \re^{-\ri\theta'} \tau)>0\} \, , \label{eq: disc-down-f}
    \end{align}
    \end{subequations}
    where $0 \le \theta' < \pi$ is the argument of the singularities of $\borel[\tfrakf]$ in the upper half-plane. This is the dual of Eqs.~\eqref{eq: disc-up-g}--\eqref{eq: disc-down-g}, obtained by applying the discontinuity formulas of~\cite[Cor.~3.5]{FR24-2} to the pair $(\tfrakf, \frakg)$ in place of $(\tfrakg, \frakf)$.
\end{itemize}
The canonical pair of MRSs $\tfrakg$, $\tfrakf$, their generating $q$-series $\frakg$, $\frakf$, and their associated $L$-functions $L_\pm$, $L'_\pm$ fit into the so-called \emph{modular resurgence paradigm}, which is schematically illustrated by the commutative diagram below. We refer to~\cite[Sec.~3.1]{FR24-2} for more details. 
\begin{equation}\label{diag:resurgence-L funct}
\begin{tikzcd}[column sep=2.3em, row sep=2.8em]
       L'_{\pm}(s) \arrow[rr, "\text{Mellin}"', "\text{inverse}"] \arrow[ddrrrrrr,sloped,"\text{functional equation}"] & &  \frakg(\tau) \arrow[rr,"\tau \rightarrow 0"] &  & \tfrakg(\tau) 
       \arrow[ll,bend right=35,red!50,"\CS^{\rm med}_\theta",swap]
               \arrow[dd,gray!50, bend left=35, sloped,
                      "\text{disc}_\theta \, /\, \text{Fricke} \;\;\;"]
                      \arrow[rr,"\text{resurgence}"] & &  \{ S_m \}\arrow[dd,sloped,"\text{$L$-function}"] \\  \\
       \{ R_m \} \arrow[uu,sloped,"\text{$L$-function}"] & & \tfrakf(\tau) 
       \arrow[rr,bend right=35,red!50,"\CS^{\rm med}_\theta",swap]
               \arrow[uu,gray!50, bend left=35, sloped,
                      "\text{disc}_\theta \, /\, \text{Fricke} \;\;\;"]
                      \arrow[ll,"\text{resurgence}"] 
       &  & \frakf(\tau)  \arrow[ll,"\tau \rightarrow 0"] & &  L_{\pm}(s) \arrow[ll, "\text{inverse}"', "\text{Mellin}"] \arrow[uullllll,sloped,swap,"\text{}"]
\end{tikzcd}
\end{equation}

\subsubsection*{The conjectures} 

As alluded to above, an MRS arising from the asymptotic expansion of a $q$-series has additional distinctive features. Most importantly, the $q$-series is expected to be a holomorphic quantum modular form and, if it is the inverse Mellin transform of an $L$-function, to be reconstructed via median resummation.
\begin{con}[Conj.~1 in~\cite{FR24-2}]\label{conj:MR-conj1}
Let $q=\re^{2 \pi \ri \tau}$ and $\frakg\colon\IH\to\CC$ be a $q$-series whose Mellin transform is an $L$-function. If its asymptotic expansion $\tfrakg(\tau)$ as $\tau\to 0$ with $\Im(\tau)>0$ has a modular resurgent structure, then the \emph{median resummation} of $\tfrakg$ reconstructs the original function $\frakg$, that is, 
\be
\mathcal{S}_\theta^{\mathrm{med}}[\tfrakg](\tau)=\frakg(\tau) \, , \quad \tau\in\IH \cap \{\Re ( \re^{-\ri\theta} \tau)>0\} \, ,
\ee
where $\theta$ is the argument of the singularities in the upper half of the Borel plane.
\end{con}
\begin{con}[Conj.~2 in~\cite{FR24-2}]\label{conj:MR-conj2}
Let $q=\re^{2 \pi \ri \tau}$ and $\frakg\colon\IH\to\CC$ be a $q$-series. If its asymptotic expansion $\tfrakg(\tau)$ as $\tau\to 0$ with $\Im(\tau)>0$ has a modular resurgent structure, then $\frakg$ is a {holomorphic quantum modular form} for a subgroup $\Gamma\subseteq\mathsf{SL}_2(\mathbb{Z})$.
\end{con}

Instances of MRSs occur in combinatorics, quantum topology, and topological string theory. Definition~\ref{def:modular_res_struct} and its implications were originally inspired by the asymptotic behaviour of the spectral trace of the toric Calabi--Yau threefold known as local $\IP^2$ and its strong-weak resurgent symmetry~\cite{Rella22, FR24-1}. 
Further evidence of Conjectures~\ref{conj:MR-conj1} and~\ref{conj:MR-conj2} comes from the theory of Maass cusp forms~\cite{FR24-2} and from the study of certain linear combinations of $q$-Pochhammer symbols~\cite{FR25}. 
The latter appear as building blocks of the spectral traces of local weighted projective planes $\IP^{m,n}$ for $m,n \in \IZ_{>0}$. 
As we will further emphasize, the dissimilarities between the case of local $\IP^2$ and the case of local $\IP^{m,n}$ guided us towards recasting modular resurgence in the vector formalism. 

\subsection{Vector-valued modular resurgent structures}\label{sec:vector-valued-MR}

As mentioned in the introduction, conceptually, the ubiquity of vector-valued (quantum) modular forms naturally motivates the introduction of the framework of \emph{vector-valued modular resurgence}. 
More practically, examples that narrowly escape Definition~\ref{def:modular_res_struct} indeed arise in related studies. For instance, an infinite family of MRSs has been constructed by the last two authors in~\cite{FR25} by asymptotically expanding sums of $q$-Pochhammer symbols twisted by Dirichlet characters.
In the same work, it is also shown that the individual $q$-Pochhammer symbols give rise to asymptotic series that meet only some of the requirements of Definition~\ref{def:modular_res_struct}: they display a tower of equally distanced simple poles in the Borel plane, but the Dirichlet series of the Stokes constants are not $L$-functions as they do not satisfy an Euler product expansion over the prime numbers. Nevertheless, as we will show in Section~\ref{sec: qPochh-vv}, they can be explicitly written as appropriate linear combinations of $L$-functions. In the second family of examples studied in this paper (see Section~\ref{sec:eichler}), the linear combination is tied directly to the $S$-matrix of the associated vector-valued quantum modular form.
To generalize the main features of an MRS and accommodate the abovementioned examples, we introduce here the definition of \emph{vector-valued modular resurgent series}.
\begin{dfn}\label{def:vv-MRS}
A vector ${\boldsymbol \tfrakg}=(\tfrakg_1,\ldots,\tfrakg_r)\in\CC^r[\![\tau]\!]$ of Gevrey-$1$ asymptotic series has a \emph{vector-valued modular resurgent structure} if the following conditions hold.
\begin{enumerate}
\item There exists a constant $\CA\in\CC$ such that, for every $k=1,\ldots,r$, the Borel transform $\CB[\tfrakg_k](\zeta) \in \CC\{\zeta\}$ has a tower of singularities at the locations $\rho_m=\CA m$, $m\in\ZZ_{\neq 0}$, in Borel plane.
\item For every $k=1,\ldots,r$ and every $m\in\IZ_{\ne 0}$, the resurgent series at the singularity $\rho_m$ is the constant function $S_{k,m}\in\CC$, \emph{i.e.}, the Stokes constant. We collect the Stokes constants into the vectors $\boldsymbol{S}_m:=(S_{1,m},\ldots,S_{r,m})\in\CC^r$, $m\in\ZZ_{\neq 0}$.
\item There exist vectors of $L$-functions $\boldsymbol{L}_{\pm}=(L_{1,\pm},\ldots,L_{v,\pm})$, vectors $\boldsymbol{c}_{\pm}=(c_{1,\pm},\ldots,c_{v,\pm})\in\CC^v$, and matrices ${\bf M}_{\pm}\in{\rm Mat}_{r\times v}(\CC)$ such that the vectors of Dirichlet series $\boldsymbol{\CL}_{\pm}=(\CL_{1,\pm},\ldots,\CL_{r,\pm})$ with
\begin{equation}\label{eq:dir_ser_def}
\boldsymbol{\calL}_{+}(s)\:=\sum_{m=1}^\infty \frac{\boldsymbol{S}_m}{m^s} \quad \text{ and } \quad \boldsymbol{\calL}_{-}(s)\:=-\sum_{m=1}^\infty \frac{\boldsymbol{S}_{-m}}{m^s}
\end{equation}
admit the linear decomposition
\be \label{eq: def-Lfunct}
\boldsymbol{\calL}_{\pm}(s) \= {\bf M}_{\pm} \, \mathrm{diag}\big(c_{1,\pm}^s,\ldots,c_{v,\pm}^s\big) \, \boldsymbol{L}_{\pm}(s) \, .
\ee
\end{enumerate}
A vector of Gevrey-1 asymptotic series with a vector-valued modular resurgent structure is a \emph{vector-valued MRS}.
\end{dfn}

Notice that the Dirichlet series $\CL_{k, \pm}$ in Eq.~\eqref{eq:dir_ser_def} can be analytically continued to the whole complex plane by means of the linear decomposition in Eq.~\eqref{eq: def-Lfunct} and the functional equations satisfied by the $L$-functions $L_{j, \pm}$.
However, they may not possess an Euler product expansion because taking linear combinations breaks the multiplicativity of the coefficients. This highlights the fact that a vector-valued MRS is typically not simply a vector of (scalar) MRSs, although the opposite is obviously true.

\subsubsection{The paradigm}\label{sec:paradigm}

Crucially, the notion proposed in Definition~\ref{def:vv-MRS} retains the core analytic and arithmetic features of modular resurgence.
Let $\boldsymbol{\tfrakg}\in\CC^r[\![\tau]\!]$ satisfy Definition~\ref{def:vv-MRS} with Stokes vectors $\boldsymbol{S}_m$, $m \in \ZZ_{\ne 0}$, and Dirichlet vectors $\boldsymbol{\calL}_\pm$ admitting the decomposition in Eq.~\eqref{eq: def-Lfunct}. We replicate, in vector form, the constructions reviewed in Section~\ref{sec: primer} under Definition~\ref{def:modular_res_struct}; all Mellin, Borel, asymptotic-expansion, and discontinuity operations below act component-wise.
\begin{itemize}
    \item Let $\boldsymbol{\frakf}(\tau)=(\frakf_1(\tau),\ldots,\frakf_r(\tau))$, $\tau\in\CC\smallsetminus\IR$, denote the vector of generating $q$-series of the Stokes vectors $\{\boldsymbol{S}_m\}$, $m\in\IZ_{\ne0}$, as in Eq.~\eqref{eq:f_Am}, equivalently, the inverse Mellin transform of the Dirichlet vectors $\boldsymbol{\calL}_\pm(s)$ on $\{\Re(s)>\alpha\}\subset\CC$, as in Eq.~\eqref{eq: invmellin-formula}.
    \item Let $\boldsymbol{\tfrakf}(\tau):=\boldsymbol{\tfrakf}_+(\tau)+\boldsymbol{\tfrakf}_-(-\tau)\in\CC^r[\![\tau]\!]$ be the symmetrized vector of asymptotic expansions of $\boldsymbol{\frakf}$ as $\tau\to0$ with $\pm\Im(\tau)>0$, as in Eq.~\eqref{eq: def-tfrakf}. Its vectors of perturbative coefficients $\boldsymbol{b}_n\in\CC^r$, $n\in\IZ_{\ge0}$, satisfy 
    \be
    \boldsymbol{b}_n=\frac{(2\pi\ri)^n}{n!}\big(\boldsymbol{\calL}_+(-n)+(-1)^n\boldsymbol{\calL}_-(-n)\big) \, , 
    \ee
    the vector form of Eq.~\eqref{eq: coeff-bn-1}.
    \item Following Eq.~\eqref{eq: disc-g}, the vector of $q$-series $\boldsymbol{\frakf}$ equals the discontinuity of the vector of asymptotic series $\boldsymbol{\tfrakg}$ in the appropriate variable.
\end{itemize}
Let now $\boldsymbol{\calL}'_\pm(s)$ denote the Dirichlet vectors obtained from the meromorphic continuation of $\boldsymbol{\calL}_\pm(s)$ to $\{\Re(s)<0\}\subset\CC$, with vectors of coefficients $\boldsymbol{R}_m:=(R_{1,m},\ldots,R_{r,m}) \in\CC^r$, $m\in\IZ_{\ne0}$, namely, 
\begin{equation}\label{eq:dir_ser_def-2}
\boldsymbol{\calL}'_{+}(s)\:=\sum_{m=1}^\infty \frac{\boldsymbol{R}_m}{m^s} \quad \text{ and } \quad \boldsymbol{\calL}'_{-}(s)\:=-\sum_{m=1}^\infty \frac{\boldsymbol{R}_{-m}}{m^s} \, , 
\end{equation}
which admit the linear decomposition
\be \label{eq: def-Lfunct-2}
\boldsymbol{\calL}'_\pm(s) \= {\bf M}'_\pm\,\mathrm{diag}\big((c'_{1,\pm})^s,\ldots,(c'_{r,\pm})^s\big)\,\boldsymbol{L}'_\pm(s) \, ,
\ee
where $\boldsymbol{c}'_\pm\in\CC^r$, ${\bf M}'_\pm\in{\rm Mat}_{r\times v}(\CC)$, and $\boldsymbol{L}'_\pm$ is the vector of $L$-functions governing the functional equation for $\boldsymbol{L}_\pm$. Finally, let $\boldsymbol{\frakg}(\tau)=(\frakg_1(\tau),\ldots,\frakg_r(\tau))$, $\tau\in\CC\smallsetminus\IR$, denote the vector of generating $q$-series of the coefficient vectors $\{\boldsymbol{R}_m\}$, $m \in \ZZ_{\ne 0}$, as in Eq.~\eqref{eq:g_Am}, equivalently, the inverse Mellin transform of the Dirichlet vectors $\boldsymbol{\calL}'_\pm(s)$ on $\{\Re(s)>\beta\}\subset\CC$.

Assume that the asymptotic expansions of $\boldsymbol{\frakg}(\tau)$ as $\tau \to 0$ with $\pm \Im(\tau)>0$ reproduce the vector-valued MRS $\boldsymbol{\tfrakg}$ of Definition~\ref{def:vv-MRS} analogously to Eq.~\eqref{eq: def-tfrakg}, so that its vectors of perturbative coefficients $\boldsymbol{a}_n\in\CC^r$, $n \in \ZZ_{\ge 0}$, are given by the $L$-values $\boldsymbol{\calL}'_\pm(-n)$ via the vector analogue of Eq.~\eqref{eq: coeff-an-2}, that is, 
\be
\boldsymbol{a}_n=\frac{(2\pi\ri)^n}{n!}\big(\boldsymbol{\calL}'_+(-n)+(-1)^n\boldsymbol{\calL}'_-(-n)\big) \, .
\ee
Also assume that each Borel transform $\borel[\tfrakg_k]$ has only simple poles.\footnote{This holds for all the examples of Sections~\ref{sec: qPochh-vv} and~\ref{sec:eichler}, where it is verified directly.}
Once more, we generalize the construction of Section~\ref{sec: primer}.
\begin{itemize}
    \item The vector of asymptotic series $\boldsymbol{\tfrakf}$ is a vector-valued MRS: there exists a constant $\CA'\in\CC$ such that, for every $k=1,\ldots,r$, the Borel transform $\borel[\tfrakf_k]$ has only simple poles at the locations $\eta_m=\CA' m$, $m\in\IZ_{\ne0}$, witth Stokes vectors $\boldsymbol{R}_m$ that are the coefficients of the Dirichlet vectors $\boldsymbol{\calL}'_\pm$ admitting the decomposition into $L$-functions in Eq.~\eqref{eq: def-Lfunct-2}.
    \item Following Eq.~\eqref{eq: disc-f}, the vector of $q$-series $\boldsymbol{\frakg}$ equals the discontinuity of the vector of asymptotic series $\boldsymbol{\tfrakf}$ in the appropriate variable.
\end{itemize}
Thus, under the assumptions above and up to the small changes already described, the paradigm of modular resurgence straightforwardly extends to vector-valued MRSs, as in the diagram below.
\begin{equation}\label{diag:vvMR-diagram}
\begin{tikzcd}[column sep=2.3em, row sep=2.8em]
        \boldsymbol{\calL}'_\pm(s) \arrow[rr, "\text{Mellin}"', "\text{inverse}"]
        \arrow[ddrrrrrr,sloped,"\text{meromorphic continuation}"]
        & & \boldsymbol{\frakg}(\tau) \arrow[rr,"\tau \rightarrow 0"]
        & & \boldsymbol{\tfrakg}(\tau) \arrow[rr,"\text{resurgence}"]
               \arrow[ll,bend right=35,red!50,"\CS^{\rm med}_\theta",swap]
               \arrow[dd,gray!50, bend left=35, sloped,
                      "\text{disc}_\theta \, /\, \text{Fricke} \quad \quad"]
        & & \{ \boldsymbol{S}_m \}\arrow[dd,sloped,"\text{Dirichlet series}"] \\ \\
       \{ \boldsymbol{R}_m \} \arrow[uu,sloped,"\text{Dirichlet series}"]
        & & \boldsymbol{\tfrakf}(\tau) \arrow[ll,"\text{resurgence}"]
               \arrow[rr,bend right=35,red!50,"\CS^{\rm med}_\theta",swap]
               \arrow[uu,gray!50, bend left=35, sloped,
                      "\text{disc}_\theta \, /\, \text{Fricke} \quad \quad"]
        & & \boldsymbol{\frakf}(\tau) \arrow[ll,"\tau \rightarrow 0"]
        & & \boldsymbol{\calL}_\pm(s) \arrow{ll}{\text{Mellin}}[swap]{\text{inverse}}
               \arrow[uullllll,sloped,swap,"\text{via  } \boldsymbol{L}_\pm \text{  and  } \boldsymbol{L}'_\pm"]
\end{tikzcd}
\end{equation}
Observe that, as mentioned earlier, the paradigm closes only under the correct choice of $q$-series vector $\boldsymbol{\frakg}$, which is not uniquely determined from the knowledge of its asymptotic series $\boldsymbol{\tfrakg}$ alone; one might begin with a vector of $q$-series whose asymptotic expansions give rise to the vector-valued MRS $\boldsymbol{\tfrakg}$ but whose Mellin transform does not reproduce the Dirichlet vector $\boldsymbol{\calL}'_\pm$.

\subsubsection{The conjectures}\label{sec:conjectures}

When a vector-valued MRS $\tilde{\frakg}$ is constructed from the asymptotic expansion of a vector of $q$-series $\boldsymbol{\frakg}$, as in the description of the paradigm above, further properties hold. The vector $\boldsymbol{\frakg}$ is conjectured to be a vector-valued holomorphic quantum modular form, as in Definition~\ref{def: vv-QMF}; moreover, if $\boldsymbol{\frakg}$ is the (component-wise) inverse Mellin transform of a linear combination of $L$-functions, then the (component-wise) median resummation of $\boldsymbol{\tfrakg}$ reconstructs $\boldsymbol{\frakg}$.
\begin{con}\label{conj:vvMR-conj1}
Let $q=\re^{2\pi\ri\tau}$ and $\boldsymbol{\frakg}\colon\IH\to\CC^r$ be a vector of $q$-series whose Mellin transforms are linear combinations of $L$-functions. If the asymptotic expansion $\boldsymbol{\tfrakg}(\tau)$ as $\tau\to0$ with $\Im(\tau)>0$ is a vector-valued MRS, then the \emph{median resummation} of $\boldsymbol{\tfrakg}$ reconstructs the original vector $\boldsymbol{\frakg}$, that is,
\be
\mathcal{S}_\theta^{\mathrm{med}}[\boldsymbol{\tfrakg}](\tau)\=\boldsymbol{\frakg}(\tau)\,,\quad \tau\in\IH\cap\{\Re(\re^{-\ri\theta}\tau)>0\}\,,
\ee
where $\mathcal{S}_\theta^{\mathrm{med}}$ acts component-wise and $\theta$ is the argument of the singularities in the upper half of the Borel plane.
\end{con}
\begin{con}\label{conj:vvMR-conj2}
Let $q=\re^{2\pi\ri\tau}$ and $\boldsymbol{\frakg}\colon\IH\to\CC^r$ be a vector of $q$-series. If the asymptotic expansion $\boldsymbol{\tfrakg}(\tau)$ as $\tau\to0$ with $\Im(\tau)>0$ is a vector-valued MRS, then $\boldsymbol{\frakg}$ is a \emph{vector-valued holomorphic quantum modular form} for a subgroup $\Gamma\subseteq\mathsf{SL}_2(\mathbb{Z})$ with multiplier $\Om:\Gamma\to\mathsf{GL}_r(\CC)$.
\end{con}
In the following Sections~\ref{sec: qPochh-vv} and~\ref{sec:eichler}, we will present two non-trivial families of examples of vector-valued MRSs and show that they verify our paradigm and conjectures.

\section{The \texorpdfstring{$q$}{q}-Pochhammer examples} \label{sec: qPochh-vv}

Certain weighted sums of $q$-Pochhammer symbols have been shown to produce MRSs and satisfy the paradigm and conjectures of modular resurgence, although the same does not hold for the $q$-Pochhammer symbols themselves~\cite{FR25}. The latter give rise to resurgent asymptotic series that fulfil the analytic criteria (parts~$1$ and~$2$) but not the arithmetic ones (part~$3$) in Definition~\ref{def:modular_res_struct}. Yet, the desired number-theoretic properties can be recovered by adopting the vector formalism of Section~\ref{sec:vector-valued-MR}.
In this section, expanding on~\cite{FR25}, we prove that particular vectors of $q$-Pochhammer symbols give rise to vector-valued MRSs in the sense of Definition~\ref{def:vv-MRS} and provide evidence of the paradigm in Eq.~\eqref{diag:vvMR-diagram} and Conjectures~\ref{conj:vvMR-conj1} and~\ref{conj:vvMR-conj2}. 

Note that the triviality of the modular representation underlying the associated quantum modular forms implies that the vector-valued structure is in fact an artefact of our choice of basis: a Dirichlet-character twist diagonalises the vector-valued MRS into independent scalar MRSs, each tied to a single primitive Dirichlet $L$-function. The vector-valued framework therefore serves here as a convenient repackaging of the scalar results of~\cite{FR25}, with the discrete Fourier transform on the character group playing the role of the diagonalising change of basis. This marks a qualitative difference between this class of examples and those discussed in Section~\ref{sec:eichler}.

\subsection{Setting}

We adopt the notation and conventions of~\cite{FR25}. 
The (infinite) $q$-Pochhammer symbols are functions of two variables defined by the infinite products
\begin{equation} \label{eq: dilog}
(x q^{\alpha}; \, q)_{\infty} \= \prod_{n=0}^{\infty} (1- x q^{\alpha+n}) \, , \quad \alpha \in \RR \, ,
\end{equation}
which are analytic in $x,q \in \CC$ with $|q| <1$.
Let us fix\footnote{We exclude $N=2$ as a degenerate case: $G_2=(\ZZ/2)^\times=\{1\}$ is a fixed point of the involution $k\mapsto N-k$, so the antisymmetrized components introduced in Eq.~\eqref{eq:component-def} vanish identically.} $N\in\ZZ_{\geq 3}$. For all $k\in G_N=(\ZZ/N)^\times$, we introduce the functions $f_{N,k}, g_{N,k}\colon \HH \to \CC$ given by
\begin{equation}\label{eq:f_kN}
    f_{N,k}(\tau):=\log(\zeta_N^k;q)_\infty \, , \quad 
    g_{N,k}(\tau):=\log(q^k;q^N)_\infty \, , 
\end{equation}
where $\zeta_N=\re^{2\pi\ri/N}$ and $q=\re^{2\pi\ri \tau}$. 
We refer to these as $q$-Pochhammer symbols for simplicity.

Let us denote by $\underline{k}=N-k \in G_N$ the image of $k$ under the involution map $k\mapsto N-k$ and introduce the index sets
\be \label{eq: subsets-GN}
    G_N^+ := \{ a \in G_N : 1 \le a \le \lfloor N/2\rfloor \} \, , \quad
    G_N^- := \{ a \in G_N : \lfloor N/2\rfloor < a \le N-1 \}
\ee
of cardinality $r=\phi(N)/2 \in \IZ_{>0}$, where $\phi$ is Euler's totient function.
Note that the map $k \mapsto \underline{k}$ pairs $G_N^+$ with
$G_N^-$. For each $k \in G_N^+$, we define the \emph{antisymmetrized
$q$-Pochhammer components}
\begin{equation}\label{eq:component-def}
    \frakf_{N,k}(\tau) := f_{N,k}(\tau)-f_{N,\underline{k}}(\tau) \, , \quad
    \frakg_{N,k}(\tau) := g_{N,k}(\tau)-g_{N,\underline{k}}(\tau) \, ,
\end{equation}
and we collect them into the length-$r$ vectors
\begin{equation}\label{eq:vf-vg}
    {\bf f}_N(\tau):=\big(\frakf_{N,k}(\tau)\big)_{k \in G_N^+} \, , \quad
    {\bf g}_N(\tau):=\big(\frakg_{N,k}(\tau)\big)_{k \in G_N^+} \, .
\end{equation}
In the remainder of this section, we exploit the results of~\cite{FR25} on
the resurgent and number-theoretic properties of $f_{N,k}$ and $g_{N,k}$ to
show that ${\bf f}_N$ and ${\bf g}_N$ fit within the framework of
vector-valued modular resurgence of Section~\ref{sec:vector-valued-MR}.

\subsection{Modular resurgence}

Let $\tilde{f}_{N,k}(\tau)$ and $\tilde{g}_{N,k}(\tau)$ denote the asymptotic expansions of the $q$-series $f_{N,k}(\tau)$ and $g_{N,k}(\tau)$ for $\tau\to 0$ with $\Im(\tau)>0$. 
Explicitly~\cite[Eqs.~(3.2) and~(3.4)]{FR25},
\begin{subequations} \label{eq: expansions}
\begin{align} 
\tilde{f}_{N,k}(\tau)&=\frac{1}{2} \log(1-\zeta_N^k)+\frac{1}{2\pi \ri \tau} \mathrm{Li}_2(\zeta_N^k)+\psi_k(\tau)\, , \label{eq: expansion-fkN} \\
\tilde{g}_{N,k}(\tau/N) &=- \frac{\pi \ri }{12 \tau} - B_1\left(\tfrac{k}{N}\right) \log(- 2 \pi \ri \tau) - \log \frac{\Gamma\big(\tfrac{k}{N}\big)}{\sqrt{2 \pi}} - B_2\big(\tfrac{k}{N}\big) \frac{\pi \ri \tau}{2} - \varphi_k(\tau)\, ,  \label{eq: expansion-gkN}
\end{align}
\end{subequations}
where $\psi_k$ and $\varphi_k$ are the Gevrey-1 asymptotic series
\begin{subequations}
\begin{align} 
\psi_k(\tau)&=\sum_{\ell=1}^\infty (2\pi \ri \tau)^{2\ell-1}\frac{B_{2\ell}}{(2\ell)!}\mathrm{Li}_{2-2\ell}(\zeta_N^k) \, , \label{eq: tilde-psi} \\
\varphi_k(\tau)&= \sum_{\ell=1}^{\infty} (2 \pi \ri \tau)^{2\ell} \frac{B_{2\ell} B_{2\ell+1}\big(\tfrac{k}{N}\big)}{2\ell (2\ell+1)!} \, ,  \label{eq: tilde-phi}
\end{align}
\end{subequations}
with $B_\ell(\alpha)$ the $\ell$-th Bernoulli polynomial,
$B_\ell=B_\ell(0)$, and $\mathrm{Li}_\ell$ the polylogarithm of order
$\ell$.

\subsubsection{Borel transform and Stokes constants}

The resurgent structures of the formal power series in Eqs.~\eqref{eq: tilde-psi} and~\eqref{eq: tilde-phi} were determined in~\cite[Sec.~3.1]{FR25}. 
From Corollary~3.2 in~\cite{FR25}, the singularities of $\borel[\psi_k]$ are simple poles located at
\be \label{eq: eta-n}
\eta_m=\frac{2\pi\ri}{N}m \, , \quad m\in\IZ_{\neq 0} \, ,
\ee 
while the Stokes constants are 
\be\label{eq:Stokes-R-kN}
R_{k,m} = \sum_{d \mid |m|} \tfrac{d}{m} \left(\delta_{k,d \, (N)} - \delta_{k,-d \, (N)} \right) \, ,
\ee
where the sum is over the positive divisors of $|m|$ and $\delta_{k,n\,(N)}$ is the congruence delta function, \emph{i.e.},  
\begin{equation} \label{eq: mod-delta}
    \delta_{k,n\,(N)}:=\begin{cases}
        1 & \mbox{if} \;\; n\equiv k \pmod{N} \,   \\ 
        0 & \mbox{otherwise} \, 
        \end{cases} \, .
\end{equation} 
From Corollary~3.5 in~\cite{FR25}, the singularities of $\borel[\varphi_k]$ are simple poles at
\begin{equation} \label{eq: zeta-n}
\rho_m=2\pi\ri m \, , \quad m\in\ZZ_{\neq 0} \, ,
\end{equation} 
with Stokes constants
\begin{equation}\label{eq:Stokes-S-kN}
S_{k,m}= 2 \ri  \sum_{d\vert |m|} \tfrac{1}{d} \sin\left(\tfrac{2 \pi k d}{N} \right) \, .
\end{equation}
Note that the Stokes constants are odd under $k \mapsto \underline{k}$, that is, $S_{k,m}= -S_{\underline{k},m}$ and $R_{k,m}= -R_{\underline{k},m}$ for $k \in G_N$, while the parities under $m\mapsto-m$ can be read off directly from Eqs.~\eqref{eq:Stokes-R-kN} and~\eqref{eq:Stokes-S-kN}, which give
\be\label{eq:parity_stokes_m}
S_{k,-m}= S_{k,m}\,, \quad R_{k,-m}= -R_{k,m} \, , \quad m\in\ZZ_{\neq 0} \, .
\ee

Accordingly, the vectors of asymptotic series
\begin{equation}\label{eq:vf-vg-tilde}
    \Psi_N(\tau):=\big(\Psi_{N,k}(\tau)\big)_{k\in G_N^+}\, , \quad
    \Phi_N(\tau):=\big(\Phi_{N,k}(\tau)\big)_{k\in G_N^+}\, ,
\end{equation}
where we define the \emph{antisymmetrized asymptotic components}
\begin{equation}\label{eq:component-asy}
    \Psi_{N,k}(\tau):=\psi_k(\tau)-\psi_{\underline{k}}(\tau)\, , \quad
    \Phi_{N,k}(\tau):=\varphi_k(\tau)-\varphi_{\underline{k}}(\tau)\, ,
\end{equation}
govern the asymptotic expansions of ${\bf f}_N$ and
${\bf g}_N$ via Eqs.~\eqref{eq: expansion-fkN} and~\eqref{eq: expansion-gkN}
and satisfy parts~1 and~2 of Definition~\ref{def:vv-MRS}. To prove that
$\Psi_N$ and $\Phi_N$ are vector-valued MRSs, it remains to show that the Dirichlet series\footnote{By the parity properties in
Eq.~\eqref{eq:parity_stokes_m}, it suffices to consider one Dirichlet series
per collection of Stokes constants: $\CL_k=\pm\CL_{k,\pm}$ while
$\CL'_k=\CL'_{k,\pm}$.}
\begin{equation}\label{eq:L-poch}
    \CL'_{k}(s)\=\sum_{m=1}^\infty\frac{R_{k,m}}{m^s}\, , \quad
    \CL_{k}(s)\=\sum_{m=1}^\infty\frac{S_{k,m}}{m^s}
\end{equation}
are linear combinations of a common set of $L$-functions for all
$k\in G_N^+$. 
Before doing so, observe that the Stokes constants in Eqs.~\eqref{eq:Stokes-R-kN} and~\eqref{eq:Stokes-S-kN} are divisor sums of arithmetic functions on $G_N$ and thus decompose in the Dirichlet-character basis. 

\paragraph{A Fourier decomposition lemma.}

Let $\chi_{j} : \ZZ \to \CC$ with $j \in J_N$ be the Dirichlet characters of modulus~$N$, indexed by a set $J_N$ with $|J_N| = \phi(N)$, and let $\overline{\chi_{j}}=\chi_{j}^{-1}$ denote their complex conjugates. The characters $\chi_j$ provide a Fourier basis of the space of complex-valued functions on $G_N$ and constitute the character group $\widehat{G}_N$. 
The Fourier transform of a function $\omega : G_N\to\CC$ is\footnote{The characters $\chi_j$ and their conjugates $\overline{\chi_j}$ define the same basis of $\widehat{G}_N$ up to relabelling; accordingly, the Fourier transform and its inverse can be written using either convention, provided it is used consistently.}
\be \label{eq: FT}
\widehat{\omega}(\chi_{j})=\sum_{a\in G_N} \omega(a) \, \overline{\chi_{j}}(a) \,, \quad \chi_{j} \in \widehat{G}_N \, ,
\ee
with inverse
\be \label{eq: invFT}
\omega(a)=\frac{1}{\phi(N)}\sum_{j \in J_N} \widehat{\omega}(\chi_{j}) \, \chi_{j}(a) \,, \quad a \in G_N \, .
\ee

\begin{lemma}\label{lem: fourier}
Let $\omega:G_N\to\CC$ (extended to $\IZ_{>0}$ by zero) and $c:\IZ_{>0}\times\IZ_{>0}\to\CC$ a weighting factor. The Dirichlet series
\be\label{eq:L-om}
\CL_\omega(s):=\sum_{m=1}^\infty\frac{1}{m^s}\sum_{d\mid m}c(d,m)\,\omega(d)
\ee
decomposes as
\be \label{eq:L-om-decomp}
    \CL_\omega(s)=\frac{1}{\phi(N)}\sum_{j\in J_N}
                \widehat{\omega}(\chi_j)\, \sum_{m=1}^\infty\frac{1}{m^s}\sum_{d\mid m}c(d,m)\,\chi_j(d) \, .
\ee
Moreover, if $\omega(-a)=\pm\omega(a)$ for all $a\in G_N$, then $\widehat{\omega}(\chi_j)=0$ for every $\chi_{j} \in \widehat{G}_N$ with $\chi_j(-1)=\mp 1$.
\end{lemma}
\begin{proof}
Substitute Eq.~\eqref{eq: invFT} into the definition of $\CL_\omega$ and exchange the absolutely convergent sums to obtain Eq.~\eqref{eq:L-om-decomp}. For the parity claim: applying the change of variables $a\mapsto-a$ to Eq.~\eqref{eq: FT} produces $\widehat{\omega}(\chi_j)= \pm\chi_j(-1)\widehat{\omega}(\chi_j)$, which vanishes unless $\chi_j(-1)=\pm 1$.
\end{proof}

\paragraph{The vector-valued modular resurgent structure.}

\begin{prop}\label{prop:vv-qpoch}
The vectors of asymptotic series $\Psi_N$ and $\Phi_N$, defined in Eq.~\eqref{eq:vf-vg-tilde}, are vector-valued modular resurgent series.
\end{prop}
\begin{proof}
Parts~1 and~2 of Definition~\ref{def:vv-MRS} hold by the resurgent
structures recalled above. It remains to exhibit the Dirichlet series $\CL_k'$ and $\CL_k$ in Eq.~\eqref{eq:L-poch} as
linear combinations of a common set of $L$-functions.

\emph{Step 1 (divisor-sum form).} Define the arithmetic functions $\omega'_k, \omega_k : G_N \to \CC$ as
\begin{equation}
    \omega'_{k}(a):= \delta_{k,a\,(N)} - \delta_{k,-a\,(N)}
    \,, \quad \omega_{k}(a):=2\ri \sin\left( \tfrac{2 \pi k a}{N} \right)\,.
\end{equation}
By Eqs.~\eqref{eq:Stokes-R-kN} and~\eqref{eq:Stokes-S-kN}, the Stokes constants $R_{k,m}$ and $S_{k,m}$, $m \in \IZ_{\ne 0}$, are the divisor sums
\be
    R_{k,m} =\sum_{d\vert |m|}\frac{d}{m}\omega'_{k}(d)\, , \quad
    S_{k,m} =\sum_{d\vert |m|} \frac{1}{d} \omega_{k}(d) \,.
\ee
Hence, $\CL'_k$ and $\CL_k$ are instances of $\CL_\omega$ in Eq.~\eqref{eq:L-om} with weights $c(d,m)=d/m$ and $c(d,m)=1/d$, respectively, and Lemma~\ref{lem: fourier} applies.

\emph{Step 2 (parity).} Let us partition the index set $J_N$ by parity as\footnote{$\Jodd_N = \emptyset$ in the degenerate cases $N = 1, 2$.}
\be\label{eq:Jeven-Jodd}
\Jeven_N := \{\, j \in J_N : \chi_j(-1) = +1 \,\}\, , \quad \Jodd_N := \{\, j \in J_N : \chi_j(-1) = -1 \,\}\, ,
\ee
so that $J_N = \Jeven_N \sqcup \Jodd_N$.
Since $\omega'_k$ and $\omega_k$ are odd on $G_N$, Lemma~\ref{lem: fourier} restricts the Fourier sum in Eq.~\eqref{eq:L-om-decomp} to $j\in\Jodd_N$, which we assume throughout.

\emph{Step 3 (Fourier coefficients).} Recall that a Dirichlet character $\chi_j$ modulo $N$, with conductor $D_j \mid N$, is induced by a unique primitive Dirichlet character $\chi_{j, \ast}$ modulo $D_j$ and satisfies
\begin{equation} \label{eq: chi-ast}
    \chi_j(\ell)=\begin{cases}
        \chi_{j, \ast}(\ell) & \mbox{if} \;\; \ell \in G_N \,   \\
        0 & \mbox{otherwise} \, 
        \end{cases} \,.
\end{equation}
The character $\chi_j$ is imprimitive when $h_j:= N/D_j>1$. If $h_j=1$, it is primitive and equal to $\chi_{j, \ast}$. 
From Eq.~\eqref{eq: FT}, the finite Fourier transform of $\omega'_{k}$ evaluates directly to
\begin{equation} \label{eq: FT-omega1}
    \begin{aligned}
        \widehat{\omega'_{k}}(\chi_{j})&
        =2 \overline{\chi_{j,\ast}}(k) \,  , \quad j \in \Jodd_N \, .
    \end{aligned}
\end{equation}
The case of $\omega_k$ requires more care. For any $k\in G_N$, $k$ is invertible modulo $N$, so the substitution $a=k^{-1}b$ re-indexes the following sum
\begin{equation}
    T(k):=\sum_{a\in G_N}\chi_j(a)\,\re^{2\pi\ri\frac{ka}{N}}\=\sum_{b\in G_N}\chi_j(k^{-1}b)\,\re^{2\pi\ri\frac{b}{N}}\=\overline{\chi_j}(k)\,\mathscr{G}(\chi_j)\,,
\end{equation}
where $\mathscr{G}(\chi_j):=\sum_{a\in G_N}\chi_j(a)\,\re^{2\pi\ri a/N}$ is the Gauss sum of the character $\chi_j$, and implies 
\begin{equation}
    \widehat{\omega_k}(\overline{\chi_j})\=T(k)-T(-k)\=\overline{\chi_j}(k)\big(1-\overline{\chi_j}(-1)\big)\mathscr{G}(\chi_j)\, .
\end{equation}
Using $\chi_j(-1)=-1$ for $j\in\Jodd_N$ together with $\chi_j(k)=\chi_{j,\ast}(k)$ for $k\in G_N$, we conclude
\begin{equation} \label{eq: FT-omega2}
    \widehat{\omega_k}(\overline{\chi_j})\=2\,\overline{\chi_{j,\ast}}(k)\,\mathscr{G}(\chi_j)\,,\quad j\in\Jodd_N\,.
\end{equation}

\emph{Step 4 (linear decomposition).} Let us introduce the vectors of Dirichlet series 
\be
\boldsymbol{\CL}'(s):=(\CL'_k(s))_{k\in G_N^+}\, , \quad \boldsymbol{\CL}(s):=(\CL_k(s))_{k\in G_N^+}\, .
\ee
Substituting Eqs.~\eqref{eq: FT-omega1} and~\eqref{eq: FT-omega2} into Lemma~\ref{lem: fourier} yields
\begin{equation}\label{eq:decomp_poch1}
    \boldsymbol{\CL}'(s) \= {\bf M}\,\boldsymbol{L}'(s)\,, \quad
    \boldsymbol{\CL}(s) \= {\bf M}\,\mathrm{diag}_{j\in\Jodd_N}\!\big(\mathscr{G}(\chi_j)\big)\,\boldsymbol{L}(s)\,,
\end{equation}
where the square matrix ${\bf M} \in{\rm Mat}_{r \times r}(\CC)$ with $r=\phi(N)/2=|G_N^+|=|\Jodd_N|$ has entries
\be\label{eq:decomp_poch2}
    {\bf M}_{kj}= 2 \frac{\overline{\chi_{j, \ast}}(k)}{\phi(N)} \, , \quad k \in G_N^+ \, , \quad j \in \Jodd_N \, ,
\ee
and the basis vectors are
\be
\boldsymbol{L}'(s):=(L'_j(s))_{j\in\Jodd_N} \, , \quad \boldsymbol{L}(s):=(L_j(s))_{j\in\Jodd_N} \, .
\ee
The building-block Dirichlet series
\begin{equation}\label{eq:calL}
    L_j'(s):=\sum_{m=1}^\infty\frac{1}{m^s}\bigg[\sum_{d|m}\frac{d}{m}\chi_{j}(d)\bigg] \,, \quad L_j(s):=\sum_{m=1}^\infty\frac{1}{m^s}\bigg[\sum_{d|m}\frac{1}{d} \overline{\chi_{j}}(d)\bigg]
\end{equation}
admit the factorizations
\begin{equation}\label{eq:Lj-factors}
    L'_j(s)=L(\chi_j,s)\,\zeta(s+1)\, , \quad
    L_j(s)=L(\overline{\chi_j},s+1)\,\zeta(s)\, ,
\end{equation}
which follow by writing $m=dq$ in Eq.~\eqref{eq:calL} (so that $d/m=1/q$) and decoupling the double sums as
\be
L'_j(s)=\sum_{d,q\ge1}\frac{\chi_j(d)}{d^s\,q^{s+1}}=L(\chi_j,s)\,\zeta(s+1)\, , \quad L_j(s)=\sum_{d,q\ge1}\frac{\overline{\chi_j}(d)}{d^{s+1}\,q^s}=L(\overline{\chi_j},s+1)\,\zeta(s)\, .
\ee
Here, $L(\chi_{j},s)$ is the Dirichlet $L$-function associated with $\chi_{j}$ and $\zeta(s)$ is the Riemann zeta function. The factorization above guarantees absolute convergence in the right half-plane $\{\Re(s)>1\} \subset \CC$, meromorphic continuation throughout the complex $s$-plane, and an Euler product expansion. Therefore, $L_j'$ and $L_j$ are genuine $L$-functions.

We conclude that $\CL'_k$ and $\CL_k$ are $\CC$-linear combinations
of a finite set of Dirichlet $L$-functions, and $\Psi_N$ and $\Phi_N$ satisfy part~3 of
Definition~\ref{def:vv-MRS}.
\end{proof}

\begin{rmk} \label{rmk: inv-M}
    As a consequence of the character orthogonality relations
    \be
        \frac{2}{\phi(N)}\sum_{j \in \Jodd_N} \chi_{j}(k^{-1}) \, {\chi_j}(a) = \delta_{k,a\,(N)} \, , \quad k \in G_N^+ \, ,
    \ee
    the matrix ${\bf M}$ in Eq.~\eqref{eq:decomp_poch2} is invertible with inverse given by ${\bf M}^{-1} = \tfrac{\phi(N)}{2} \, \overline{{\bf M}}^{\mathsf{T}}$.
\end{rmk}

By the factorization in Eq.~\eqref{eq:Lj-factors}, each $L'_j$ and $L_j$ inherits a functional equation from its factors (\emph{i.e.}, the Dirichlet $L$-function $L(\chi_j,s)$ and the Riemann zeta function). We derive this explicitly in Lemma~\ref{lem: funct-eq}. Then, via the linear decomposition in Eq.~\eqref{eq:decomp_poch1}, we translate it into functional equations for $\CL'_k$ and $\CL_k$ in Proposition~\ref{prop:vv-qpoch-Ljs}.

\begin{lemma} \label{lem: funct-eq}
Let $\chi_{j}$ with $j\in\Jodd_N$ be an odd Dirichlet character modulo $N$, with conductor $D_j\mid N$, and let $\chi_{j,\ast}$ be the primitive Dirichlet character modulo $D_j$ that induces it.\footnote{Recall that $D_j=N$ precisely when $\chi_j$ is primitive.}
The $L$-functions $L'_{j}(s)$ and $L_j(s)$, defined in Eq.~\eqref{eq:calL}, obey the combined functional equation
\be\label{eq:functional_eq}
\Lambda'_j(-s) = \frac{\mathscr{G}(\chi_{j,\ast})}{\ri \sqrt{D_j}} \Lambda_j(s) \, , \quad s \in \CC \, ,
\ee
where $\Lambda_j'$ and $\Lambda_j$ are the meromorphic completions of $L_j'$ and $L_{j}$, respectively, and $\mathscr{G}$ is the Gauss sum.
\end{lemma}

\begin{proof}
    The proof treats the primitive ($D_j=N$) and imprimitive ($D_j<N$) cases uniformly. 
    By definition of $\chi_{j,\ast}$, we can write $\chi_j(n)=\chi_{j,\ast}(n)\,\delta_{\gcd(n,N),1}$ for $n \in \IZ_{>0}$. Applying the M\"obius identity $\sum_{d\mid m}\mu(d)=\delta_{m,1}$, where $\mu$ is the Möbius function, with $m=\gcd(n,N)$ and using the complete multiplicativity of $\chi_{j,\ast}$ yields the factorization
    \begin{equation}\label{eq:L-funct-imp}
    L(\chi_j,s)\=\sum_{n\ge1}\frac{\chi_{j,\ast}(n)}{n^s}\sum_{d\mid\gcd(n,N)}\mu(d)\= L(\chi_{j,\ast},s)\sum_{d|N} \mu(d)\,\frac{\chi_{j,\ast}(d)}{d^s}\,.
    \end{equation}
    For primitive characters, the M\"obius factor reduces to $1$ since $\chi_{j,\ast}(d)=0$ for every $d\mid N$ with $d>1$.
     
    Recall that the classical functional equation of the completed Dirichlet $L$-function of the primitive odd Dirichlet character $\chi_{j,\ast}$ is
    \begin{equation}\label{eq:Lambda-dir}
    \xi(\chi_{j,\ast},s):=D_j^{\frac{s}{2}}\pi^{-\frac{s+1}{2}}\Gamma\left(\tfrac{s+1}{2}\right)L(\chi_{j,\ast},s)\=\frac{\mathscr{G}(\chi_{j,\ast})}{\ri\sqrt{D_j}}\xi(\overline{\chi}_{j,\ast},1-s)\, .
    \end{equation}
    As a result, the completed $L$-function
    \begin{equation}\label{eq:Lambda-imp-chi}
    \xi(\chi_j,s):=\xi(\chi_{j,\ast},s)\Bigg(\sum_{d|N} \mu(d)\tfrac{\chi_{j,\ast}(d)}{d^s}\Bigg)\Bigg(\sum_{d|N} \mu(d)\tfrac{\overline{\chi_{j,\ast}}(d)}{d^{1-s}}\Bigg)
    \end{equation}
    obeys the functional equation
    \begin{equation} \label{eq: feq-step}
        \xi(\chi_j,s)\=\frac{\mathscr{G}(\chi_{j,\ast})}{\ri\sqrt{D_j}}\xi(\overline{\chi}_j,1-s)\,.
    \end{equation}
    Accordingly, the $L$-functions in Eq.~\eqref{eq:Lj-factors} admit the M\"obius-corrected decompositions
    \begin{subequations}\label{eq:Lj-factors-imprim}
    \begin{align}
        L_j'(s)&= L(\chi_{j,\ast},s) \zeta(s+1) \sum_{d|N} \mu(d)\,\tfrac{\chi_{j,\ast}(d)}{d^s} \label{eq:Lj-factors-imprim-1} \, , \\
        L_j(s)&= L(\overline{\chi_{j,\ast}},s+1) \zeta(s) \sum_{d|N} \mu(d)\,\tfrac{\overline{\chi_{j,\ast}}(d)}{d^{s+1}} \label{eq:Lj-factors-imprimm-2}\, , 
    \end{align}
    \end{subequations}
    and we define their completions to be
    \begin{subequations}\label{eq:Lambda-imprim}
    \begin{align}
        \Lambda_j'(s)&:= D_j^{\frac{s}{2}} \pi^{-s-1} \, \Gamma\left(\tfrac{s+1}{2} \right) \Gamma\left( \tfrac{s+1}{2} \right) L_j'(s) \sum_{d|N} \mu(d)\tfrac{\overline{\chi_{j,\ast}}(d)}{d^{1-s}} \label{eq:Lambda-imprim-1} \, , \\
        \Lambda_j(s)&:= D_j^{\frac{s+1}{2}} \pi^{-s-1} \, \Gamma\left(\tfrac{s+2}{2} \right) \Gamma\left( \tfrac{s}{2} \right) L_j(s) \sum_{d|N} \mu(d)\tfrac{\chi_{j,\ast}(d)}{d^{-s}} \label{eq:Lambda-imprim-2}\, .
    \end{align}
    \end{subequations}
    The functional equation in Eq.~\eqref{eq:functional_eq} then follows from Eq.~\eqref{eq: feq-step} together with the well-known meromorphic continuation of the Riemann zeta function.\footnote{Note that $\frac{\mathscr{G}(\overline{\chi_{j,\ast}})}{\ri \sqrt{D_j}}=\frac{\ri \sqrt{D_j}}{\mathscr{G}(\chi_{j,\ast})}$ by the standard property $\mathscr{G}(\chi_{j,\ast})\mathscr{G}(\overline{\chi_{j,\ast}})=-D_j$ for odd primitive Dirichlet characters.}
    For primitive\footnote{The functional equation in the primitive case was originally derived in~\cite[Sec.~3.3]{FR24-1} and~\cite[Sec.~4.2.1]{FR25}.} $\chi_j$, the Möbius sums in Eqs.~\eqref{eq:Lj-factors-imprim} and~\eqref{eq:Lambda-imprim} equal $1$, so that
    Eq.~\eqref{eq:functional_eq} reduces to $\Lambda'_j(-s)=\frac{\mathscr{G}(\chi_j)}{\ri\sqrt N}\Lambda_j(s)$, where 
    \be
        \Lambda_j'(s):= N^{\frac{s}{2}}\pi^{-s-1}\Gamma(\tfrac{s+1}{2})^2L'_j(s) \, , \quad \Lambda_j(s):= N^{\frac{s+1}{2}}\pi^{-s-1}\Gamma(\tfrac{s+2}{2})\Gamma(\tfrac s2)L_j(s)\, .
    \ee
\end{proof}

Combining Lemma~\ref{lem: funct-eq} with the linear decomposition in Eq.~\eqref{eq:decomp_poch1}, we are now ready to translate the building-block functional equations for $L'_j$ and $L_j$ into functional equations for $\CL'_k$ and $\CL_k$.

\begin{prop}\label{prop:vv-qpoch-Ljs}
The Dirichlet series $\CL'_k(s)$ and $\CL_k(s)$, defined in Eq.~\eqref{eq:L-poch}, are absolutely convergent for $s \in \{\Re(s)>1\} \subset \CC$ and admit meromorphic continuation to $s \in \{\Re(s) <0\} \subset \CC$ through each other.
\end{prop}
\begin{proof}
Since $\omega'_k$ and $\omega_k$ are odd on $G_N$, only the odd Dirichlet characters modulo $N$ contribute to the decompositions in Eq.~\eqref{eq:decomp_poch1} (see Lemma~\ref{lem: fourier}), so we take $j\in\Jodd_N$ throughout. Unwrapping the functional equation of Lemma~\ref{lem: funct-eq}---that is, substituting the M\"obius-corrected completions of Eq.~\eqref{eq:Lambda-imprim} and the factorizations of Eq.~\eqref{eq:Lj-factors-imprim} and applying the gamma reflection and duplication identities---yields
    \begin{subequations} \label{eq: funct-eq-odd-imp}
    \begin{align}
            L_j'(-s)&= \frac{D_j^s}{N^s} t(s) \, \mathscr{G}(\chi_{j,\ast}) L(\overline{\chi_{j,\ast}},s+1) \zeta(s) \sum_{d|N} \mu(d)\tfrac{\chi_{j,\ast}(d)}{d^{-s}}  \, , \\
            \mathscr{G}(\chi_{j,\ast}) \,L_j(-s)&= \frac{D_j^s}{N^s} t'(s) \, L(\chi_{j,\ast},s) \zeta(s+1)\sum_{d|N} \mu(d)\tfrac{\overline{\chi_{j,\ast}}(d)}{d^{1-s}} \, ,
    \end{align}
    \end{subequations}
    where the meromorphic functions $t,t' : \CC \to \CC$ are independent of $j$ and explicitly given by
    \be \label{eq: funct-eq-odd-ts}
    t(s) = (1+\cos(\pi s)) \, \frac{\Gamma(s) \Gamma(s+1)}{\pi \ri (4 \pi^2/N)^{s}} \, , \quad t'(s) = (1-\cos(\pi s)) \,  \frac{\Gamma(s) \Gamma(s+1)}{\pi \ri (4 \pi^2/N)^{s}} \, .
    \ee
    When $\chi_j$ is primitive ($D_j=N$), the M\"obius sums reduce to $1$ and Eq.~\eqref{eq: funct-eq-odd-imp} collapses to
    \be \label{eq: funct-eq-odd}
    L_j'(-s)= t(s) \, \mathscr{G}(\chi_{j}) L_j(s)  \, , \quad \mathscr{G}(\chi_{j}) \,L_j(-s)= t'(s) \, L'_j(s) \, .
    \ee

Absolute convergence of $\CL'_k(s)$ and $\CL_k(s)$ on $\{\Re(s)>1\}$ is inherited from that of the building blocks $L'_j(s)$ and $L_j(s)$ (established after Eq.~\eqref{eq:Lj-factors}) through the linear decomposition in Eq.~\eqref{eq:decomp_poch1}. For the analytic continuation to $\{\Re(s)<0\}$, the right-hand sides of Eqs.~\eqref{eq: funct-eq-odd} and~\eqref{eq: funct-eq-odd-imp} are products of the entire completed Dirichlet $L$-functions of the primitive characters $\chi_{j,\ast}$, the Riemann zeta function, and finite M\"obius polynomials, hence they are meromorphic on all of $\CC$; thus each $L'_j$ and $L_j$, and therefore each $\CL'_k$ and $\CL_k$ via Eq.~\eqref{eq:decomp_poch1}, continue meromorphically to $s\in\CC$. Since the coefficient relating $L'_j(-s)$ to $L_j(s)$ depends on $j$, this continuation holds \emph{through each other} at the level of the vectors: by the invertibility of the matrix ${\bf M}$ (see Remark~\ref{rmk: inv-M}), $\boldsymbol{\CL}'(-s)$ is a meromorphic combination of the components of $\boldsymbol{\CL}(s)$, and vice versa.
\end{proof}

\subsubsection{The paradigm} \label{sec: paradigm-pochh}

To conclude that the vector-valued MRSs $\Psi_N$ and $\Phi_N$ in Eq.~\eqref{eq:vf-vg-tilde} form a canonical pair connected via the modular resurgence paradigm of Section~\ref{sec:vector-valued-MR}, we must verify one last statement. 
Particularly, for all $k \in G_N^+$, we recall from~\cite[Cor.~3.3 and~3.6]{FR25} that
\begin{subequations}
    \begin{align}
        {\rm disc}_{\frac{\pi}{2}}[\psi_k](\tau) & =\sum_{m=1}^\infty R_{k,m} \, \re^{-\eta_m/\tau}=-\frakg_{N,k}\Big(-\tfrac{1}{N\tau}\Big) \, ,  \label{eq:disc_psi}\\
        {\rm disc}_{\frac{\pi}{2}}[\varphi_k](\tau) & =\sum_{m=1}^\infty S_{k,m} \, \re^{-\rho_m/\tau}=-\frakf_{N,k}\Big(-\tfrac{1}{\tau}\Big)+\frac{\pi\ri}{N}(2k-N) \, , \label{eq:disc_phi}
    \end{align}
\end{subequations}
that is, the discontinuities of $\Psi_N$ and $\Phi_N$ reproduce the vectors of $q$-series ${\bf g}_{N}$ and ${\bf f}_{N}$ in Eq.~\eqref{eq:vf-vg}, up to a change of coordinate and up to an overall constant. 
As discussed in Sections~\ref{sec: primer} and~\ref{sec:vector-valued-MR}, these discontinuities can be suitably interpreted as the inverse Mellin transforms of the Dirichlet series $\CL_k'$ and $\CL_k$ in Eq.~\eqref{eq:L-poch}, respectively.

Therefore, the modular resurgence paradigm holds component-wise, namely, for each antisymmetrized pair $\frakg_{N,k}(\tau)$ and $\frakf_{N,k}(\tau)$ separately, with the meromorphic continuation of $\CL_k'(s)$ and $\CL_k(s)$ dictated by Proposition~\ref{prop:vv-qpoch-Ljs}. See the diagram below.
\begin{equation}\label{diag:vvMR-diagram-poch}
\begin{tikzcd}[column sep=2.2em, row sep=2.9em]
    4\CL'_{k}(s)  
    \arrow[ddrrrrrr,sloped,"\text{mer. continuation via } L_j \text{  and  } L'_j", start anchor={[yshift=-2ex]}, end anchor={[yshift=2ex]}]
    \arrow[rr, "\text{inverse Mellin}","\text{Eq.}~\eqref{eq:disc_psi}" swap] 
    &  & \frakg_{N,k}(\tau) \arrow[rr,"\tau \rightarrow 0","\text{Eq}.~\eqref{eq: expansion-gkN}" swap] &  & -\Phi_{N,k}(\tau)
    \arrow[dd, gray!50, bend left=35,"\text{disc}\, /\, \text{Fricke} \quad \quad", "\text{Eq.}~\eqref{eq:disc_phi} \quad "'{yshift=-0.5ex}, sloped, allow upside down]
    \arrow[rr,"\text{resurgence}","\text{Eq.}~\eqref{eq:Stokes-S-kN}" swap]\arrow[ll,bend right=35,red!50,"\text{Eq.~\eqref{eq:median-pochh}}","\CS^{\rm med}_{\pi/2}" swap] & & \{ -2 S_{k,m} \}\arrow[dd,sloped,"\text{Dirichlet series}","\text{Eq.}~\eqref{eq:L-poch}" swap] \\  \\
    \{ 4 R_{k,m} \} \arrow[uu,sloped,"\text{Dirichlet series}"] & & 2\Psi_{N,k}(\tau)
    \arrow[uu, gray!50, bend left=35,"\text{disc}\, /\, \text{Fricke} \quad \quad", "\text{Eq.}~\eqref{eq:disc_psi} \quad "'{yshift=-0.5ex}, sloped, allow upside down]
    \arrow[ll,"\text{resurgence}","\text{Eq.}~\eqref{eq:Stokes-R-kN}" swap]\arrow[rr,bend right=35,red!50,"\CS^{\rm med}_{\pi/2}" swap]  
    &  & 2\frakf_{N,k}(\tau) \arrow[ll,"\tau \rightarrow 0","\text{Eq.}~\eqref{eq: expansion-fkN}" swap] &  & -2\CL_{k}(s) \arrow[ll,"\text{inverse Mellin}","\text{Eq.}~\eqref{eq:disc_phi}" swap]\arrow[uullllll,sloped,swap,"\text{Prop.}~\ref{prop:vv-qpoch-Ljs}", start anchor={[yshift=2ex]}, end anchor={[yshift=-2ex]}]
\end{tikzcd}
\end{equation}
Note that this corresponds to the vector-valued paradigm in Eq.~\eqref{diag:vvMR-diagram}, here written in terms of the individual components of the vectors. In fact, the $q$-Pochhammer example discussed in this section is a rather special instance of the general paradigm where the vector-valued modular resurgent structures diagonalize into collections of scalar ones.

\begin{rmk}
We consider $j \in \Jodd_N$, as before. Taking the inverse Mellin transform of the $L$-functions\footnote{$\mathscr{G}(\chi_j)\,L_j(s)$ is the natural building block on the $\CL_k$-side of Eq.~\eqref{eq:decomp_poch1}.} $L_j'(s)$ and $\mathscr{G}(\chi_j)L_j(s)$ in Eq.~\eqref{eq:calL} produces the pairs of $q$-series\footnote{In the notation of~\cite{BD_Lambert25}, the $q$-series $\Xi_j(\tau)$ and $\Xi'_j(\tau)$ in Eq.~\eqref{eq: weighted1} correspond to the twisted Lambert series $-\calL_1(\overline{\chi_{j}} ; \, \tau)$ and $-\mathscr{G}(\chi_j)\tilde{\calL}_1(\chi_{j} ; \, \tau)$, respectively.}
\be \label{eq: weighted1}
    \Xi'_j(\tau)\:= \sum_{m=1}^\infty\sum_{d|m}\frac{d}{m}\chi_{j}(d) q^m
    \, , \quad
    \Xi_j(\tau)\:=\mathscr{G}(\chi_j)\sum_{m=1}^\infty \sum_{d|m}\frac{1}{d} \overline{\chi_{j}}(d) q^m \, .
\ee
We collect these into the vectors $\boldsymbol{\Xi}'(\tau):=\big(\Xi'_j(\tau)\big)_{j\in\Jodd_N}$ and $\boldsymbol{\Xi}(\tau):=\big(\Xi_j(\tau)\big)_{j\in\Jodd_N}$, which satisfy
\begin{equation}\label{eq: weighted-vec}
    -\boldsymbol{\Xi}'(\tau) \= {\bf M}^{-1} \, {\bf g}_N(\tau) \, , \quad -\boldsymbol{\Xi}(\tau) \= {\bf M}^{-1} \, {\bf f}_N(\tau) \, ,
\end{equation}
where ${\bf M}^{-1}$ is the matrix of Remark~\ref{rmk: inv-M}.
When the character $\chi_j$ is primitive, the $q$-series
\be
    \sum_{k \in G_N^+} \chi_{j}(k) \frakg_{N,k}(\tau) = -\Xi'_j(\tau)\, , \quad \sum_{k \in G_N^+} \chi_{j}(k) \frakf_{N,k}(\tau)= -\Xi_j(\tau)
\ee
satisfy the original modular resurgence paradigm and conjectures of Section~\ref{sec: primer}~\cite[Sec.~4]{FR25}. 
As a result, after a change of basis corresponding to the twisting by the appropriate Dirichlet characters, the vector-valued MRSs $\Psi_N$ and $\Phi_N$ become simply a collection of scalar MRSs. 
\end{rmk}

\subsubsection{Verification of the conjectures}

In this subsection, we show that the $q$-series vectors ${\bf f}_N$ and ${\bf g}_N$ satisfy Conjectures~\ref{conj:vvMR-conj1} and~\ref{conj:vvMR-conj2}, as inherited component-wise from~\cite{FR25}. 
By Theorem~3.9 in~\cite{FR25}, the median resummation of the asymptotic expansion $\tilde{f}_{N,k}$ in Eq.~\eqref{eq: expansion-fkN} fails to reconstruct the $q$-Pochhammer symbol $f_{N,k}$ by a correction that is symmetric under $k \leftrightarrow \underline{k}$. The analogous statement for $\tilde{g}_{N,k}$ in Eq.~\eqref{eq: expansion-gkN} is Conjecture~5 in~\cite{FR25}.
Therefore, the effectiveness of the median resummation holds component-wise for the vectors ${\bf f}_N$ and ${\bf g}_N$.
\begin{lemma}\label{lem:conj-median-f-g}
   The vectors of $q$-series ${\bf f}_N(\tau)$ and ${\bf g}_N(\tau)$, $\tau \in \IH$, in Eq.~\eqref{eq:vf-vg} are reconstructed via median resummation of their asymptotic expansions $\tilde{\bf f}_N(\tau)$ and $\tilde{\bf g}_N(\tau)$ for $\tau\to0$ with $\Im(\tau)>0$, that is,
        \begin{equation}\label{eq:median-pochh}
        \CS_\frac{\pi}{2}^{\rm med}[\tilde{\bf f}_N](\tau)={\bf f}_N(\tau) \quad \text{ and }\quad \CS_\frac{\pi}{2}^{\rm med}[\tilde{\bf g}_N](\tau)={\bf g}_N(\tau)\,.
    \end{equation}
\end{lemma}
\begin{proof}
    The statement follows from~\cite[Eqs.~(3.34) and~(3.38)]{FR25}.
\end{proof}

By Theorem~3.10 in~\cite{FR25}, the $q$-Pochhammer symbols $f_{k,N}$ and $g_{k,N}$ are holomorphic quantum modular functions for the subgroup $\Gamma_N \subset {\sf SL}_2(\ZZ)$ generated by 
\be \label{eq: GammaN-gen}
T=\left(\begin{matrix}
    1 & 1 \\
    0 & 1
\end{matrix}\right) \, , \quad \gamma_N=\left(\begin{matrix}
    1 & 0 \\
    N & 1
\end{matrix}\right) \, .
\ee
The same holds for the antisymmetrized components $\frakf_{N,k}$ and $\frakg_{N,k}$. 
\begin{lemma}\label{lem:mod_qpochh}
   The vectors of $q$-series ${\bf f}_N(\tau)$ and ${\bf g}_N(\tau)$, $\tau \in \IH$, in Eq.~\eqref{eq:vf-vg} are vector-valued holomorphic quantum modular functions for $\Gamma_N \subset {\sf SL}_2(\ZZ)$ with trivial multiplier system.
\end{lemma}
\begin{proof}
    The statement follows from~\cite[Thm.~(3.10)]{FR25}. 
\end{proof}
As mentioned at the beginning of this section, the trivial multiplier system in Lemma~\ref{lem:mod_qpochh} implies that the vector-valued modular resurgent structure of ${\bf f}_N$ and ${\bf g}_N$ factors through the scalar resurgence of their components established in~\cite{FR25} and amounts to a Dirichlet-character repackaging of the latter. The situation differs sharply for the Eichler integrals of unary theta series of Section~\ref{sec:eichler}: there, the vectors of $q$-series carry a \emph{non-trivial} multiplier system over the full modular group, and their quantum modular behaviour is intrinsically vector-valued, with no scalar reduction available.

\section{The Eichler integral examples} \label{sec:eichler}

Eichler integrals of unary theta series~\cite{eichler1957verallgemeinerung, shimura1959integrales, LawrenceZagier99} play a recurring role across mathematics and mathematical physics.
In quantum topology, for instance, they appear as building blocks of the WRT invariants of Seifert-fibered homology spheres~\cite{LawrenceZagier99, HIKAMI_2005} and of the $\widehat{Z}$-invariants of weakly negative-definite plumbed three-manifolds~\cite{Gukov:2017kmk, Cheng:2018vpl, Cheng:2024vou}.
They also arise in quantum knot theory, in the computation of colored Jones polynomials of alternating knots~\cite{Garoufalidis:2011np, BeirneOsburn}, while  
in the representation theory of non-rational vertex operator algebras, they enter the characters of modules of $(1,N)$-singlet vertex algebras~\cite{Adamovic:2007er} and $(N,N')$-singlet vertex algebras~\cite{BringmannMilas15}.

From the resurgent and modular perspectives developed in this work, Eichler integrals of unary theta series provide a family of examples of vector-valued MRSs, satisfying the paradigm illustrated in Eq.~\eqref{diag:vvMR-diagram} and Conjectures~\ref{conj:vvMR-conj1} and~\ref{conj:vvMR-conj2}.
In contrast to the $q$-Pochhammer example of Section~\ref{sec: qPochh-vv}, the modular representation underlying the associated vector-valued quantum modular forms is non-trivial, and no change of basis reduces the vector-valued MRS to a collection of scalar MRSs; this example therefore demonstrates the necessity of the vector-valued framework introduced in Section~\ref{sec:vector-valued-MR}.

\subsection{Setting} \label{sec:eichler-setting}

Let us fix $N\in\IZ$ and $k\in\ZZ/{2N}$.
We consider the collection of functions ${\sf EI}[\theta^{(\nu)}_{N,k}] : \IH \to \CC$ given by
\begin{equation}\label{eq:tilde-th}
    {\sf EI}[\theta^{(\nu)}_{N,k}](\tau):=\sum_{n\in\ZZ}\text{sgn}(n)\,n^{1-\nu}\,\delta_{k,n\,(2N)}\,q^{\frac{n^2}{4N}}\,, \quad  \nu=0,1 \, ,
\end{equation}
where $\delta_{k,n\,(2N)}$ is the congruence delta function and $q=\re^{2 \pi \ri \tau}$, as in Section~\ref{sec: qPochh-vv}. These are Eichler integrals of the unary theta series
\begin{equation}\label{eq:theta-un}
        \theta^{(\nu)}_{N,k}(\tau):=\sum_{n\in\ZZ} n^{\nu}\,\delta_{k,n\,(2N)}\, q^{\frac{n^2}{4N}}\,.
\end{equation}
More precisely, ${\sf EI}[\theta^{(\nu)}_{N,k}]$ admits the integral representations\footnote{For fixed $t$ on the contour, the kernel $(t-\tau)^{-1/2}$ is normalized by $\arg(t-\tau)\to\tfrac{\pi}{2}$ as $t\to\ri\infty$, corresponding to the branch $\arg(t-\tau)\in(-\tfrac{\pi}{2},\tfrac{3\pi}{2})$, so that the cut in $t$ lies along the downward ray $\tau-\ri\,\IR_{\geq0}$, and the kernel is continuous on the contour for every $\tau$ off it. The cut in $\tau$ lies instead along the upward ray $t+\ri\,\IR_{\geq0}$, so the straight-contour integral defines a holomorphic function in $\tau$ on the complement of the vertical ray $L_\gamma:=-d/c+\ri\,\IR_{\geq0}$. We adopt this convention throughout Section~\ref{sec:eichler} and explicitly refer to it in the proof of Proposition~\ref{prop:qm-ei}.}
\begin{subequations}
\begin{align}
    {\sf EI}[\theta^{(0)}_{N,k}](\tau)&\=-\frac{\sqrt{2N\ri}}{\pi}\int_\tau^{\ri\infty}\frac{d\theta^{(0)}_{N,k}}{dt}(t)\,(t-\tau)^{-1/2}\, dt\,, \label{eq:tilde_int_0-setting} \\
    {\sf EI}[\theta_{N,k}^{(1)}](\tau)&\=\frac{1}{\sqrt{2N\ri}}\int_{\tau}^{\ri\infty} {\theta}_{N,k}^{(1)}(t)\, (t-\tau)^{-1/2}\, dt\,. \label{eq:tilde_int_1}
\end{align}
\end{subequations}
In particular, the $q$-series ${\sf EI}[\theta^{(1)}_{N,k}]$ is a \emph{false theta function}, the name deriving from the fact that its $q$-expansion is identical to that of the theta function $\theta^{(0)}_{N,k}$ apart from the sign factors, and is closely related to partial theta functions.

Similarly to Section~\ref{sec: qPochh-vv}, we can package the $q$-series in Eqs.~\eqref{eq:tilde-th} and~\eqref{eq:theta-un} into two pairs of length-$2N$ vectors. Namely,
\begin{equation} \label{eq:EI-vg}
    \bEI^{(\nu)}_{{\rm all},N}(\tau):=\big({\sf EI}[\theta^{(\nu)}_{N,k}](\tau)\big)_{k\in\IZ/2N}\,, \quad
    \boldsymbol\theta^{(\nu)}_{{\rm all},N}(\tau):=\big(\theta^{(\nu)}_{N,k}(\tau)\big)_{k\in\IZ/2N}\,, \quad \nu=0,1~.
\end{equation}
\begin{rmk}\label{rmk: parity-EI}
Under the change of variable $k \mapsto -k$ in Eqs.~\eqref{eq:tilde-th} and~\eqref{eq:theta-un}, the components of $\bEI^{(\nu)}_{{\rm all},N}$ and $\boldsymbol\theta^{(\nu)}_{{\rm all},N}$ satisfy the parity relations
\begin{equation}\label{eq:parity-EI}
{\sf EI}[\theta^{(\nu)}_{N,-k}](\tau) = (-1)^\nu\,{\sf EI}[\theta^{(\nu)}_{N,k}](\tau)\,, \quad \theta^{(\nu)}_{N,-k}(\tau) = (-1)^\nu\,\theta^{(\nu)}_{N,k}(\tau)\,,
\end{equation}
for $k \in \ZZ/2N$ and $\nu = 0,1$. For $\nu=1$, these force the fixed points $k=0$ and $k=N$ of the involution $k \mapsto -k$ to vanish identically: ${\sf EI}[\theta^{(1)}_{N,0}] = {\sf EI}[\theta^{(1)}_{N,N}] = 0$ and $\theta^{(1)}_{N,0} = \theta^{(1)}_{N,N} = 0$. Accordingly, we introduce the folded index set
\begin{equation}\label{eq:G-tilde-folded}
\widetilde{G}^{(\nu)}_{N} := \begin{cases} 
\{0, 1, 2, \ldots, N\} \subset \ZZ/2N & \mbox{if} \;\; \nu = 0 \\ 
\{1, 2, \ldots, N-1\} \subset \ZZ/2N & \mbox{if} \;\;\nu = 1 \end{cases} \, ,
\end{equation}
of cardinality $r_\nu:=|\widetilde{G}^{(\nu)}_{N}| = N + 1 - 2\nu$, and the length-$r_\nu$ folded vectors
\begin{equation}\label{eq:bEI-folded}
\bEI^{(\nu)}_N(\tau) := \big({\sf EI}[\theta^{(\nu)}_{N,k}](\tau)\big)_{k \in \widetilde{G}^{(\nu)}_{N}}\,, \quad
\boldsymbol\theta^{(\nu)}_N(\tau) := \big(\theta^{(\nu)}_{N,k}(\tau)\big)_{k \in \widetilde{G}^{(\nu)}_{N}}\,.
\end{equation}
By the parity relations in Eq.~\eqref{eq:parity-EI}, the folded vectors $\bEI^{(\nu)}_N$ and $\boldsymbol\theta^{(\nu)}_N$ of Eq.~\eqref{eq:bEI-folded} carry the same information as the full vectors $\bEI^{(\nu)}_{{\rm all},N}$ and $\boldsymbol\theta^{(\nu)}_{{\rm all},N}$ of Eq.~\eqref{eq:EI-vg}.
\end{rmk}
In this section, we study the resurgent and number-theoretic properties of the component functions ${\sf EI}[\theta^{(\nu)}_{N,k}]$ for $k\in\widetilde{G}^{(\nu)}_N$ and show that the vectors $\bEI^{(\nu)}_N$ for $\nu=0,1$ give rise to vector-valued MRSs according to Definition~\ref{def:vv-MRS} and satisfy the conjectures and (a variant of) the paradigm of Section~\ref{sec:vector-valued-MR}.

\paragraph{Connection to previous work.}

The interplay between modularity and resurgence for Eichler integrals of unary theta series has been investigated recently from several complementary directions. In particular, the work by Costin, Dunne, Gruen, and Gukov~\cite{Costin:2023kla} and the subsequent works~\cite{adams2025orientationreversalchernsimonsnatural, Costin:Mock26} approach the relation between Eichler integrals and their mock theta companions via the  extension $\hbar\mapsto-\hbar$ of Borel-resummed trans-series of Mordell-type integrals, decomposed into real and imaginary parts, similar in its procedure to the Eichler duality method discussed in~\cite{Cheng:2012qc}.
The quantum modularity of the holomorphic Eichler integrals ${\sf EI}[\theta^{(1)}_{N,k}]$ under the action of congruence subgroups was established by Goswami and Osburn~\cite{GO}, and the case of non-holomorphic Eichler integrals was treated by Bringmann and Rolen~\cite{BR}; the cocycle computation we carry out here extends these results to vector-valued $\mathsf{SL}_2(\IZ)$-modularity, uniformly for $\nu=0,1$.

\subsection{Theta series and their \texorpdfstring{$L$}{L}-functions}

We first recall some properties of unary theta series and their associated Dirichlet series that will be useful later. To begin with, the vectors of unary theta series are modular objects.
\begin{lemma}\label{lem:mod_theta}
    For $\nu=0,1$, the vector of unary theta series $\boldsymbol\theta^{(\nu)}_N(\tau)$, $\tau \in \IH$, defined in Eq.~\eqref{eq:bEI-folded}, is a vector-valued modular form for $\mathsf{SL}_2(\IZ)$ of weight $\omega = 1/2 + \nu$. Namely,
    \be \label{eq:mod_trans}
        \boldsymbol\theta^{(\nu)}_N|_\omega\gamma = \boldsymbol\theta^{(\nu)}_N \, , \quad \forall \gamma\in\mathsf{SL}_2(\ZZ) \, ,
    \ee
    where $|_\omega$ is the weight-$\omega$ slash operator in Eq.~\eqref{eq: slash} with multiplier system $\Om^{(\nu)} : \mathsf{SL}_2(\IZ) \to \mathsf{GL}_{r_\nu}(\CC)$ specified in Remark~\ref{rmk: unary-identities}. Equivalently, 
    \be \label{eq:mod_trans-2}
        \boldsymbol\theta^{(\nu)}_N\left(\tfrac{a\tau+b}{c\tau+d}\right) = (c \tau + d)^\omega \Om^{(\nu)}(\gamma) \boldsymbol\theta^{(\nu)}_N(\tau) \, , \quad \forall \gamma=\left( \begin{smallmatrix}
    a & b\\
    c & d
\end{smallmatrix} \right)\in\mathsf{SL}_2(\IZ) \, .
    \ee
\end{lemma}
\begin{proof}
It suffices to verify Eq.~\eqref{eq:mod_trans} on the generators $T = \left(\begin{smallmatrix} 1 & 1 \\ 0 & 1 \end{smallmatrix}\right)$ and $S = \left(\begin{smallmatrix} 0 & -1 \\ 1 & 0 \end{smallmatrix}\right)$ of $\mathsf{SL}_2(\IZ)$. The verification is immediate for $\gamma = T$ and follows from Poisson summation for $\gamma = S$. Folding the resulting transformation by the parity relation in Eq.~\eqref{eq:parity-EI} yields the explicit multiplier system in Remark~\ref{rmk: unary-identities}.
\end{proof}

\begin{rmk} \label{rmk: unary-identities}
By definition, the multiplier system $\Om^{(\nu)}: \mathsf{SL}_2(\IZ) \to \mathsf{GL}_{r_\nu}(\CC)$ of Lemma~\ref{lem:mod_theta} acts on the vector $\boldsymbol\theta^{(\nu)}_N$ of length $r_\nu= N + 1 - 2\nu$. 
On the generators of $\mathsf{SL}_2(\IZ)$, its entries are
\be \label{eq: Om}
\Om^{(\nu)}_{kk'}(T) \= \delta_{kk'}\, \re^{2\pi\ri \frac{k^2}{4N}}\,, \quad
\Om^{(\nu)}_{kk'}(S) \= \frac{(-\ri)^\nu \, {|\mathcal{O}_{k'}|}}{\sqrt{2\ri N}} \times
\begin{cases} \cos\!\big(\tfrac{k k' \pi}{N}\big) & \mbox{if} \;\; \nu = 0 \\[2pt] \sin\!\big(\tfrac{k k' \pi}{N}\big) & \mbox{if} \;\; \nu = 1 \end{cases} \, ,
\ee
for $k, k' \in \widetilde{G}^{(\nu)}_N$. Here, $\mathcal{O}_{k'} := \{k', -k'\} \subset \ZZ/2N$ denotes the orbit of $k'$ under the parity involution $k' \mapsto -k'$, so that $|\mathcal{O}_{k'}| = 1$ at the fixed points $k' \in \{0, N\}$ and $|\mathcal{O}_{k'}| = 2$ at the other points. In particular, $|\mathcal{O}_{k'}| = 2$ for all $k' \in \widetilde{G}^{(1)}_N$.

Observe that the unfolded vector $\boldsymbol\theta^{(\nu)}_{{\rm all},N}(\tau)$, $\tau \in \HH$, defined in Eq.~\eqref{eq:EI-vg}, satisfies a modular transformation analogous to Eq.~\eqref{eq:mod_trans}, that is, 
\be \label{eq:mod_trans-all}
\boldsymbol\theta^{(\nu)}_{{\rm all},N}|_\omega\gamma = \boldsymbol\theta^{(\nu)}_{{\rm all},N} \, , \quad  \forall \gamma \in \mathsf{SL}_2(\IZ)\, , 
\ee
with unfolded multiplier system $\widehat\Om^{(\nu)} : \mathsf{SL}_2(\IZ) \to \mathsf{GL}_{2N}(\CC)$. On the generators of $\mathsf{SL}_2(\IZ)$, the latter has entries
\be \label{eq: Om-tilde}
\widehat\Om^{(\nu)}_{kk'}(T) \= \delta_{kk'}\, \re^{2\pi\ri \frac{k^2}{4N}}\,, \quad
\widehat\Om^{(\nu)}_{kk'}(S) \= \frac{(-1)^\nu}{\sqrt{2\ri N}}\, \re^{\pi\ri \frac{k k'}{N}}\,, 
\ee
for $k, k' \in \ZZ/2N$, from which the elements in Eq.~\eqref{eq: Om} are induced simply by folding the sum on the right-hand side of Eq.~\eqref{eq:mod_trans-all} by means of the parity relations in Eq.~\eqref{eq:parity-EI}.

Since $\omega = 1/2 + \nu$ is half-integral, $\Om^{(\nu)}$ is a projective representation of $\mathsf{SL}_2(\IZ)$ and lifts to a genuine homomorphism of the metaplectic double cover $\mathsf{Mp}_2(\IZ)$, while $S, T$ lift to metaplectic generators. In particular, $\Om^{(\nu)}(S)$ has order $8$: a direct computation starting from Eq.~\eqref{eq: Om} gives $\Om^{(\nu)}(S)^2 = (-1)^{\nu+1}\ri\,\boldsymbol{1}$ and hence $\Om^{(\nu)}(S)^4 = -\boldsymbol{1}$, reflecting the fact that the lift of $S$ in $\mathsf{Mp}_2(\IZ)$ has order $8$.
\end{rmk}

The unary theta series $\theta^{(\nu)}_{N,k}$ are naturally associated with the \emph{theta-Mellin Dirichlet series}
\begin{equation}\label{eq:L}
    \CCL^{(\nu)}_k(s):=\sum_{m=1}^\infty\frac{\omega^{(\nu)}_k(m;N)}{m^{2s-\nu}}\,, \quad k \in \IZ/2N \, , \quad \nu =0,1 \, , 
\end{equation}
where $\omega^{(\nu)}_k(a;N)$ with $a\in\ZZ$ is the residue-class arithmetic function
\begin{equation}\label{eq:omega-def}
    \omega^{(\nu)}_k(a;N):=\delta_{k,a\,(2N)}+(-1)^\nu\delta_{k,-a\,(2N)}\,.
\end{equation}
The latter is $2N$-periodic in both $a$ and $k$.
For later convenience, we introduce the \emph{canonical Dirichlet series}
\begin{equation}\label{eq:D-canonical}
    D^{(\nu)}_k(u):=\sum_{m=1}^\infty\frac{\omega^{(\nu)}_k(m;N)}{m^{u}}\,, 
\end{equation}
which is related to $\CCL^{(\nu)}_k(s)$ by the shift identity
\begin{equation}\label{eq:CCL-D-identity}
    \CCL^{(\nu)}_k(s) \= D^{(\nu)}_k(2s-\nu)\,.
\end{equation}
Note that the factor of $2$ in the input $2s-\nu$ on the right-hand side matches the quadratic progression of the exponent of $q^{n^2}$ in the expression for the theta series. This factor of $2$ will also be responsible for the relation between the naive and the natural Borel variables $\xi=\zeta^2$ in the following subsection. 
We further define the \emph{completed Dirichlet series}
\begin{equation}\label{eq:Lambda-def}
    \Lambda^{(\nu)}_k(s) := \left(\frac{2N}{\pi}\right)^s \Gamma(s)\, \CCL^{(\nu)}_k(s) \= \left(\frac{2N}{\pi}\right)^s \Gamma(s)\, D^{(\nu)}_k(2s-\nu)\,,
\end{equation}
which, as we will see in the proof of Lemma~\ref{lem:functionalRel-Lfn-theta}, coincides with the Mellin transform of $\theta^{(\nu)}_{N,k}$ along the imaginary axis.
The parity relation for the theta series in Eq.~\eqref{eq:parity-EI} carries through to the associated Dirichlet series, implying
\be
\CCL^{(\nu)}_{-k}(s) = (-1)^\nu\,\CCL^{(\nu)}_k(s) \, ,
\ee
and analogously for $D^{(\nu)}_k$ and $\Lambda^{(\nu)}_k$. Therefore, similarly to Eq.~\eqref{eq:bEI-folded}, we collect them into the length-$r_\nu$ folded vectors
\begin{equation}\label{eq:Dirichlet-vec}
\pmb{\mathscr{L}}^{(\nu)}_N(s) := \big(\CCL^{(\nu)}_k(s)\big)_{k \in \widetilde{G}^{(\nu)}_N}\,, \quad
\boldsymbol{D}^{(\nu)}_N(u) := \big(D^{(\nu)}_k(u)\big)_{k \in \widetilde{G}^{(\nu)}_N}\,, \quad
\boldsymbol{\Lambda}^{(\nu)}_N(s) := \big(\Lambda^{(\nu)}_k(s)\big)_{k \in \widetilde{G}^{(\nu)}_N}\,,
\end{equation}
which admit meromorphic continuation to the whole complex plane as follows.
\begin{lemma}\label{lem:functionalRel-Lfn-theta}
    For $\nu=0,1$, the vector of completed Dirichlet series $\boldsymbol\Lambda^{(\nu)}_N$, defined in Eq.~\eqref{eq:Dirichlet-vec}, satisfies the functional equation
    \begin{equation}\label{eq:func-L}
        \boldsymbol\Lambda^{(\nu)}_N(s) \= \ri^\omega \, \Om^{(\nu)}(S) \, \boldsymbol\Lambda^{(\nu)}_N(\omega-s) \, , \quad s \in \CC \, ,
    \end{equation}
    where $\omega = 1/2 + \nu$ and $\Om^{(\nu)}(S) \in \mathrm{Mat}_{r_\nu \times r_\nu}(\CC)$ is the $S$-matrix introduced in Lemma~\ref{lem:mod_theta}.
\end{lemma}
\begin{proof}
    For $k \in \widetilde{G}^{(\nu)}_N$, the Mellin transform of the theta series $\theta_{N,k}^{(\nu)}$ is
    \begin{equation}
    \begin{aligned}
        \int_0^\infty \theta_{N,k}^{(\nu)}(\ri t)\, t^{s-1}\;dt&= \int_0^\infty \sum_{n\in\ZZ}n^\nu\,\delta_{k,n\,(2N)}\, \re^{-\frac{\pi n^2}{2N} t}\, t^{s-1}\;dt\\
        &=\sum_{n\in\ZZ}n^\nu\,\delta_{k,n\,(2N)} \left( \frac{2N}{\pi n^2} \right)^s \,\Gamma(s)\\
        &=\Lambda^{(\nu)}_{k}(s)\,,
    \end{aligned}
    \end{equation}
    which we write equivalently in vector form as
    \be
    \int_0^\infty \boldsymbol\theta^{(\nu)}_N(\ri t)\, t^{s-1}\;dt = \boldsymbol\Lambda^{(\nu)}_N(s) \, . 
    \ee
    Applying the modular property in Lemma~\ref{lem:mod_theta} and integrating term-by-term yields
    \begin{equation}
        \begin{aligned}
            \int_0^\infty \boldsymbol\theta^{(\nu)}_N(\ri t)\,t^{s-1}\,dt&=\int_0^\infty \boldsymbol\theta^{(\nu)}_N(-1/\ri t)\,t^{-1-s}\,dt\\
            &=\ri^\omega \,\Om^{(\nu)}(S)\, \int_0^\infty \boldsymbol\theta^{(\nu)}_N(\ri t)\,t^{-1+\omega-s}\,dt\\
            &=\ri^\omega \,\Om^{(\nu)}(S)\, \boldsymbol\Lambda^{(\nu)}_N(\omega-s)\,,
        \end{aligned}
    \end{equation}
    where $\omega=1/2 +\nu$, thus proving Eq.~\eqref{eq:func-L}.
\end{proof}

When $(k,2N)=g>1$, it is natural to consider the theta series
$\theta^{(\nu)}_{N/g,k/g}(\tau)=g^{-\nu} \theta^{(\nu)}_{N,k}(\tau/g)$ of \emph{reduced index} $N/g$ and define accordingly 
\begin{equation}\label{eq:L-new}
  \tilde\CCL^{ (\nu)}_k(s):=\sum_{m=1}^\infty\frac{\omega^{(\nu)}_{k/g}(m;N/g)}{m^{2s-\nu}}= g^{2s-\nu} \CCL^{(\nu)}_k(s) \,, \quad \nu = 0,1,
\end{equation}
using that $\omega^{(\nu)}_{k/g}(m;N/g)= \omega^{(\nu)}_{k}(mg;N)$.
Note that the Dirichlet series $\CCL^{(\nu)}_{k}$ and $\tilde\CCL^{(\nu)}_{k}$ in Eqs.~\eqref{eq:L} and~\eqref{eq:L-new} do not satisfy an Euler product expansion because their coefficients are not multiplicative. Thus, they are not $L$-functions in the sense of Definition~\ref{def: Lfunct}. Nonetheless, we prove below that they can be expressed as linear combinations of a (finite) collection of Dirichlet $L$-functions.

\begin{lemma} \label{lem: decomp-ei}
For $\nu=0,1$ and $k \in \IZ/2N$ with $(k,2N)=1$, the Dirichlet series $\CCL^{(\nu)}_{k}$, defined in Eq.~\eqref{eq:L}, is in the $\CC$-span of Dirichlet $L$-functions modulo $2N$. 
More generally, for $k \in \IZ/2N$ satisfying $(k,2N)=g$, the Dirichlet series $\tilde\CCL^{(\nu)}_{k}$, defined in Eq.~\eqref{eq:L-new}, is in the $\CC$-span of Dirichlet $L$-functions modulo $2N/g$. 
\end{lemma}
\begin{proof}
We will first consider the case of $k$ coprime to $2N$. Let $\phi$ be Euler's totient function and $\chi_{j}$ with $j \in J_{2N}$ be the Dirichlet characters of modulus~$2N$, which we index analogously to Section~\ref{sec: qPochh-vv}. In particular, $|J_{2N}| = \phi(2N)$ and $J_{2N} = \Jeven_{2N} \sqcup \Jodd_{2N}$ (see Eq.~\eqref{eq:Jeven-Jodd}).
The characters $\chi_j$ provide a Fourier basis of the space of complex-valued functions on $G_{2N}$ and form the character group $\widehat{G}_{2N}$. 
We follow closely the arguments presented in the proof of Proposition~\ref{prop:vv-qpoch} in Section~\ref{sec: qPochh-vv}.

We consider the arithmetic functions $\omega^{(\nu)}_k(\,\cdot\,;N): G_{2N} \to \CC$, defined in Eq.~\eqref{eq:omega-def}, which appear as coefficients of the Dirichlet series in Eq.~\eqref{eq:L}, and express them as inverse Fourier transforms via Eq.~\eqref{eq: invFT}. Namely,
\be
    \omega^{(\nu)}_k(a;N)\=\frac{1}{\phi(2N)}\sum_{j \in J_{2N}}\,\widehat{\omega_{k}}^{(\nu)}(\chi_{j})\,\chi_{j}(a)\,. \label{eq: invFT-omega-Ei}
\ee
The functions $\widehat{\omega_k}^{(\nu)} : \widehat{G}_{2N}\to\CC$ can then be computed via Eq.~\eqref{eq: FT} as
\begin{equation} \label{eq: FT-omegaEI}
    \widehat{\omega_k}^{(\nu)}(\chi_{j})\=\big[1+(-1)^{\nu}\chi_{j,\ast}(-1) \big] \overline{\chi_{j,\ast}}(k) \, , 
\end{equation}
where $\chi_{j, \ast}$ is the primitive Dirichlet character of modulus $D_j$ that induces $\chi_j$.
Note that odd (resp. even) Dirichlet characters do not contribute to the linear combinations in Eq.~\eqref{eq: invFT-omega-Ei} for $\nu=0$ (resp. $\nu=1$). 
As a result, substituting Eq.~\eqref{eq: FT-omegaEI} into Eq.~\eqref{eq: invFT-omega-Ei} yields
\be\label{eq:omega-decomp}
    \omega^{(0)}_k(a;N)\=\sum_{j\in \Jeven_{2N}}\,2\frac{\overline{\chi_{j,\ast}}(k)}{\phi(2N)}\,\chi_{j}(a) \, , \quad
    \omega^{(1)}_k(a;N)\=\sum_{j\in \Jodd_{2N}}\,2\frac{\overline{\chi_{j,\ast}}(k)}{\phi(2N)}\,\chi_{j}(a) \, ,
\ee
which implies, at the level of the canonical Dirichlet series $D^{(\nu)}_k$ in Eq.~\eqref{eq:D-canonical}, the decompositions
\be\label{eq:decomp_ei}
    D^{(0)}_{k}(u)\=\sum_{j\in \Jeven_{2N}}\,2\frac{\overline{\chi_{j,\ast}}(k)}{\phi(2N)}\,L_j(u) \, , \quad
    D^{(1)}_{k}(u)\=\sum_{j\in \Jodd_{2N}}\,2\frac{\overline{\chi_{j,\ast}}(k)}{\phi(2N)}\,L_j(u) \, ,
\ee
where we introduced the Dirichlet $L$-functions
\begin{equation}\label{eq:calL-Ei}
    L_{j}(s):=L(\chi_{j},s)=\sum_{m=1}^\infty\frac{\chi_{j}(m)}{m^s} \, .
\end{equation}
Via the identity in Eq.~\eqref{eq:CCL-D-identity}, this is equivalent to the theta-Mellin Dirichlet series decompositions
\begin{equation}\label{eq:decomp_ei1-new}
    \CCL^{(0)}_{k}(s)=\sum_{j \in \Jeven_{2N}} 2 \frac{\overline{\chi_{j,\ast}}(k)}{\phi(2N)}\, L_j(2s)\,, \quad
    \CCL^{(1)}_{k}(s)=\sum_{j \in \Jodd_{2N}} 2 \frac{\overline{\chi_{j,\ast}}(k)}{\phi(2N)}\, L_j(2s-1)\,,
\end{equation}
in which the factor of $2$ is restored.

In the general case where $(k,2N)=g$, the proof follows identically. Indeed, the Dirichlet series $\tilde\CCL^{(\nu)}_{k}$ in Eq.~\eqref{eq:L-new} correspond to the theta series $\theta^{(\nu)}_{N/g,k/g}$ of reduced index $N/g$, and by construction $(k/g, 2N/g)=1$.
In particular, the $\phi(2N/g)$ values of $k\in \ZZ/2N$ that satisfy $(k,2N)=g$ lead to as many Dirichlet series $\tilde\CCL^{(\nu)}_k$ with $k/g \in G_{2N/g}$. The latter are shown to be in the $\CC$-linear span of the Dirichlet $L$-functions associated with characters $\chi_j \in \widehat{G}_{2N/g}$ by the same arguments above. Explicitly,
\be\label{eq:decomp_ei2}
    \tilde\CCL^{(\nu)}_{k}(s)\=\sum_{j }\,2\frac{\overline{\chi_{j,\ast}}(k/g)}{\phi(2N/g)}\,L_j(2s-\nu) \, ,
\ee
where the sum is over $\Jeven_{2N/g}$ for $\nu=0$ and $\Jodd_{2N/g}$ for $\nu=1$.
\end{proof}

{In the spirit of Definition~\ref{def:vv-MRS}, the linear decomposition in Eq.~\eqref{eq:decomp_ei2} translates, at the level of the canonical Dirichlet series $D^{(\nu)}_k$ in Eq.~\eqref{eq:D-canonical}, into the expressions
\be\label{eq:decomp_ei2-bis}
    D^{(0)}_{k}(u)\= \sum_{j\in J^{\rm even}_{{\rm all},\,2N}}\, {\bf O}_{kj} \, c_j^{-u} \,L_j(u) \, , \quad
    D^{(1)}_{k}(u)\= \sum_{j\in J^{\rm odd}_{{\rm all},\,2N}}\, {\bf O}_{kj} \, c_j^{-u} \,L_j(u) \, ,
\ee
where $c_j=2N/m_j$ for a Dirichlet $L$-function $L_j$ of modulus~$m_j$,
\be
{\bf O}_{kj} = \begin{cases} 2\frac{\overline{\chi_{j,\ast}}(k/c_j)}{\phi(m_j)} & \mbox{if} \;\; k/c_j \in G_{m_j}\\
0 & \mbox{otherwise} \;\;
\end{cases} \, ,
\ee
and the index sets $J^{\rm even}_{{\rm all},\,2N}$ and $J^{\rm odd}_{{\rm all},\,2N}$ collect the Dirichlet characters of all moduli dividing $2N$ by parity, that is, 
\be\label{eq:Jall-defs}
J^{\rm even}_{{\rm all},\,2N} := \bigsqcup_{m\vert2N} \Jeven_m\,, \quad
J^{\rm odd}_{{\rm all},\,2N} := \bigsqcup_{m\vert2N} \Jodd_m\,.
\ee
Equivalently, via the identity in Eq.~\eqref{eq:CCL-D-identity}, the theta-Mellin Dirichlet series read as
\be
\CCL^{(0)}_{k}(s)= \sum_{j\in J^{\rm even}_{{\rm all},\,2N}}\, {\bf O}_{kj} \, c_j^{-2s} \,L_j(2s) \, , \quad
\CCL^{(1)}_{k}(s)= \sum_{j\in J^{\rm odd}_{{\rm all},\,2N}}\, {\bf O}_{kj} \, c_j^{1-2s} \,L_j(2s-1) \, .
\ee
As we will see in the proof of Proposition~\ref{prop:falsetheta-vvmrs}, when linearly combining $\CCL^{(\nu)}_k$ for $k \in \widetilde{G}^{(\nu)}_N$ into the Dirichlet series of the Stokes constants for the asymptotic expansions of the Eichler integrals, which we compute in Eq.~\eqref{eq:L-funct-Stheta}, the expression above leads precisely to the decomposition required by part~3 of Definition~\ref{def:vv-MRS}.

\subsection{Modular resurgence}\label{subsec:modularresugence}

In order to prove that the vectors of Eichler integrals $\bEI^{(\nu)}_N$ for $\nu=0,1$ are vector-valued modular resurgent $q$-series, we begin by computing their asymptotic expansions.

\subsubsection{Asymptotic expansion}

\begin{lemma}\label{lem:asympfalsetheta}
    Asymptotically expanding the vectors of Eichler integrals $\bEI^{(0)}_N(\tau)-\frac{\ri}{\pi\tau}\,\boldsymbol{1}$ and $\bEI^{(1)}_N(\tau)$, defined in Eq.~\eqref{eq:bEI-folded}, for $\tau\to 0$ with $\Im(\tau)>0$ produces the vectors of Gevrey-1 asymptotic series
    \begin{equation} \label{eq: tilde-Phi}
        \tbEI^{(\nu)}_N(\tau)\=\sum_{\ell=0}^\infty {\bf c}_{\ell}^{(\nu)} \, \tau^\ell \, ,
    \end{equation}
    where the coefficient vectors ${\bf c}_{\ell}^{(\nu)} \in \CC^{r_\nu}$, $\ell \in \ZZ_{\ge 0}$, are given explicitly by
    \begin{equation} \label{eq: cj-coeffs}
        {\bf c}_{\ell}^{(\nu)}\=\frac{(-1)^\nu}{\pi} \left(\frac{2N}{\pi \ri}\right)^{3/2-\nu+\ell}\Gamma\left( \tfrac{3}{2}-\nu+\ell \right) \, \Om^{(\nu)}(S) \, \pmb{\mathscr{L}}^{(\nu)}_N(1+\ell)\,,
    \end{equation}
    with $\Om^{(\nu)}(S)$ the $S$-matrix from Lemma~\ref{lem:mod_theta} and $\pmb{\mathscr{L}}^{(\nu)}_N$ the vector of Dirichlet series in Eq.~\eqref{eq:Dirichlet-vec}.
\end{lemma}
\begin{proof}
    By construction, the Eichler integral ${\sf EI}[\theta^{(\nu)}_{N,k}]$ in Eq.~\eqref{eq:tilde-th} can be interpreted as the generating $q$-series of the coefficients of the Dirichlet series $\CCL^{(\nu)}_{k}$ in Eq.~\eqref{eq:L}.
    More precisely, the Mellin transform of ${\sf EI}[\theta^{(\nu)}_{N,k}]$ for $k \in \widetilde{G}^{(\nu)}_{N}$ gives
     \begin{equation}\label{eq:mellin_eichler}
    \begin{aligned}
        \int_0^\infty {\sf EI}[\theta_{N,k}^{(\nu)}](\ri t)\, t^{s-1}\;dt&= \int_0^\infty \sum_{n\in\ZZ}\text{sgn}(n)\,n^{1-\nu}\,\delta_{k,n\,(2N)}\, \re^{-\frac{\pi n^2}{2N} t}\, t^{s-1}\;dt\\
        &=\sum_{n\in\ZZ}\text{sgn}(n)\,n^{1-\nu}\,\delta_{k,n\,(2N)}  \left( \frac{2N}{\pi n^2} \right)^s \,\Gamma(s)\\
        &= \left( \frac{2N}{\pi} \right)^s \,\Gamma(s) \, \sum_{n=1}^\infty n^{1-\nu-2s}\,\omega^{(\nu)}_k(n;N) \\
        &=\left( \frac{2N}{\pi} \right)^s \,\Gamma(s)\,\CCL^{(\nu)}_{k}\left(\nu-\tfrac{1}{2}+s\right)
        \,.
    \end{aligned}
    \end{equation}
    For $\nu=0$, the Mellin inversion of Eq.~\eqref{eq:mellin_eichler} picks up, besides the poles of $\Gamma(s)$, a simple pole of its integrand at $s=1$: the Dirichlet series $\CCL^{(0)}_k(s)$ has a simple pole at $s=\tfrac12$ with residue $\tfrac{1}{2N}$, hence the residue of $\big(\tfrac{2N}{\pi}\big)^{s}\Gamma(s)\,\CCL^{(0)}_k\big(s-\tfrac12\big)\,t^{-s}$ at $s=1$ equals $\tfrac{1}{\pi t}=\tfrac{\ri}{\pi\tau}$. For $\nu=1$, no such pole occurs.
    
    Expressing the coefficients of the asymptotic expansion of the generating $q$-series ${\sf EI}[\theta^{(\nu)}_{N,k}]$ in terms of the values of the corresponding Dirichlet series (see~\cite[Prop.~3.4]{FR24-2}), we obtain that
    \begin{equation} \label{eq: Phi-step1}
            \tbEI^{(\nu)}_N(\tau) \=\sum_{\ell=0}^\infty\frac{1}{\ell!} \pmb{\mathscr{L}}^{(\nu)}_N \left(\nu-\tfrac{1}{2}-\ell\right) \left(\frac{2\pi\ri \tau }{4 N} \right)^\ell \,.
    \end{equation}
    From the functional equation in Eq.~\eqref{eq:func-L}, we find that
    \begin{equation} \label{eq: Phi-step2}
            \pmb{\mathscr{L}}^{(\nu)}_N\left(\nu-\tfrac{1}{2}-\ell\right) \=\ri^{\omega} \left(\frac{2N}{\pi} \right)^{3/2-\nu+2\ell}\frac{\Gamma(\ell+1)}{\Gamma\left(\nu-\tfrac{1}{2}-\ell\right) } \, \Om^{(\nu)}(S)\, \pmb{\mathscr{L}}^{(\nu)}_N(1+\ell) \, ,
    \end{equation}
    where $\omega=1/2+\nu$. Substituting Eq.~\eqref{eq: Phi-step2} into Eq.~\eqref{eq: Phi-step1} and applying the reflection identity for the gamma function gives the desired expression for $\tbEI^{(\nu)}_N$.
\end{proof}

The coefficients of the asymptotic series $\tEI[\theta_{N,k}^{(\nu)}] \in \CC [\![\tau]\!]$ can be written in terms of trigonometric functions using the expression for the $S$-matrix $\Om^{(\nu)}(S)$ in Eq.~\eqref{eq: Om}. More precisely, 
\begin{subequations}
\begin{align}
    c_{k,\ell}^{(0)}&\=\frac{2}{\pi\sqrt{2\ri N}}\left(\frac{2N}{\pi \ri}\right)^{3/2+\ell} \Gamma\big(\tfrac{3}{2}+\ell\big)\sum_{m=1}^\infty\frac{\cos(\frac{\pi k m}{N})}{m^{2(\ell+1)}}\, , \label{eq:c0-LZ} \\ 
    c_{k,\ell}^{(1)}&\=\frac{2\ri}{\pi\sqrt{2\ri N}}\left(\frac{2N}{\pi \ri}\right)^{1/2+\ell} \Gamma\big(\tfrac{1}{2}+\ell\big)\sum_{m=1}^\infty\frac{\sin(\frac{\pi k m}{N})}{m^{2\ell+1}}\, . \label{eq:c1-LZ}
    \end{align}
\end{subequations}

\paragraph{Alternative derivation via Lawrence\texorpdfstring{--}{-}Zagier: asymptotics.}

For $\nu=1$, the asymptotic series in Eq.~\eqref{eq: tilde-Phi} can equivalently be derived via the $L$-series approach of Lawrence and Zagier~\cite{LawrenceZagier99}, which we now review. 

\begin{lemma}[Prop.~on~page~98~in~\cite{LawrenceZagier99}]\label{lem:LZ}
Let $C:\mathbb{Z}\to \mathbb{C}$ be a periodic function with mean value zero. The associated $L$-series
\begin{equation}
    L(s,C) := \sum_{n=1}^\infty \frac{C(n)}{n^s} \, , \quad \Re(s)>1 \, ,
\end{equation}
extends holomorphically to $\mathbb{C}$, and the function $\sum_{n=1}^\infty C(n) \re^{-n^2t}$, defined for $t >0$, has the asymptotic expansion
\begin{equation}
    \sum_{n=1}^\infty C(n) \re^{-n^2 t}  \sim \sum_{r=0}^\infty L(-2r,C) \frac{(-t)^r}{r!}  \, ,
\end{equation}
as $t\to 0^+$. The number $L(-r,C)$ is given explicitly by
\begin{equation}\label{Lmr-def}
    L(-r,C) = - \frac{M^r}{r+1} \sum_{n=1}^M C(n) B_{r+1}\left( \tfrac{n}{M} \right) \, , \quad r\in\mathbb{Z}_{\geq 0} \, ,
\end{equation}
where $B_k(x)$ is the $k$-th Bernoulli polynomial and $M$ is any period of $C$.
\end{lemma}

\begin{prop} \label{prop: EI1-asympt-new}
    For $k \in \IZ/2N$, the Eichler integral ${\sf EI}[\theta^{(1)}_{N,k}](\tau)$, defined in Eq.~\eqref{eq:tilde-th}, has the asymptotic expansion
    \begin{equation}\label{A:perturbativeexpansionB}
       {\sf EI}[\theta^{(1)}_{N,k}](\tau)\sim \tEI[\theta_{N,k}^{(1)}](\tau)=\frac{2}{\pi\sqrt{\pi }}\sum_{\ell=0}^\infty
       \left(\frac{2N}{\pi \ri}\right)^{\ell}\Gamma\big(\tfrac{1}{2}+\ell\big)\sum_{m=1}^\infty\frac{\sin(\frac{\pi k m}{N})}{m^{2\ell+1}} \tau^\ell \, ,
    \end{equation}
    as $\tau\to0^+$.
\end{prop}
\begin{proof}
    Recall that the Bernoulli polynomials satisfy the identities $B_\ell(1-x) = (-1)^\ell B_\ell(x)$ and
    \begin{equation}\label{Bernoulli-B2np1}
        B_{2\ell+1}(x) = (-1)^{\ell+1} \frac{2(2\ell+1)!}{(2\pi)^{2\ell+1}} \sum_{m=1}^\infty \frac{\sin(2\pi m x)}{m^{2\ell+1}}  \, , \quad \ell \in \IZ_{\ge 0} \,.
    \end{equation}
    Rewriting Eq.~\eqref{eq:tilde-th} for $\nu = 1$ by splitting the sum over $n \in \ZZ$ into positive and negative parts gives
    \be
        {\sf EI}[\theta^{(1)}_{N,k}](\tau) \= \sum_{n=1}^\infty \omega^{(1)}_k(n;N)\,\re^{\pi \ri n^2 \tau/(2N)} \, , 
    \ee
    with $\omega^{(1)}_k(n;N)$ defined in Eq.~\eqref{eq:omega-def}. The arithmetic function $\omega^{(1)}_k(\,\cdot\,;N) : \ZZ \to \CC$ is $2N$-periodic with mean value zero, so we can apply Lemma~\ref{lem:LZ} to ${\sf EI}[\theta^{(1)}_{N,k}]$. Its asymptotic expansion is
\begin{equation}\label{A:perturbativeexpansion}
    \tEI[\theta_{N,k}^{(1)}](\tau)=\sum_{\ell=0}^\infty a_{k,\ell} \, \tau^{\ell} \, , \quad a_{k,\ell}= \frac{L\big(-2\ell,\, \omega^{(1)}_k(\,\cdot\,;N)\big)}{\ell!} \left( \frac{\pi \ri}{2N} \right)^\ell \,,
\end{equation}
where the number $L(-2\ell,\,\omega^{(1)}_k(\,\cdot\,;N))$ is computed explicitly as
\begin{equation}
    L\big(-2\ell,\,\omega^{(1)}_k(\,\cdot\,;N)\big) = -\frac{(2N)^{2\ell}}{2\ell+1}
    \left( B_{2\ell+1} \left( \tfrac{k}{2N} \right) -
    B_{2\ell+1} \left(1- \tfrac{k}{2N} \right) \right) \, , \quad \ell \in \mathbb{Z}_{\ge 0} \, .
\end{equation}
Substituting the identity in Eq.~\eqref{Bernoulli-B2np1} yields
\begin{equation}
    L\big(-2\ell,\,\omega^{(1)}_k(\,\cdot\,;N)\big) = (-1)^\ell \frac{2 (2\ell)!}{\pi} \left(\frac{N}{\pi}\right)^{2\ell} \sum_{m=1}^\infty \frac{\sin\left( \tfrac{\pi k m}{N}\right)}{m^{2\ell+1}} \, ,
\end{equation}
and Eq.~\eqref{A:perturbativeexpansionB} then follows.
\end{proof}

As shown in the computation leading to Eq.~\eqref{eq:c1-LZ}, the asymptotic series $\tEI[\theta_{N,k}^{(1)}]$ admits two equivalent expressions---the Mellin\,/\,$\CCL^{(\nu)}$ form of Eq.~\eqref{eq: tilde-Phi} and the Lawrence--Zagier form of Eq.~\eqref{A:perturbativeexpansionB}---which coincide as formal power series.
Note that we cannot apply the strategy of Lemma~\ref{lem:LZ} to the other family of Eichler integrals ${\sf EI}[\theta_{N,k}^{(0)}]$ with $k \in \ZZ/2N$, since the coefficients in their $q$-series expansion in Eq.~\eqref{eq:tilde-th} are not mean-zero. Explicitly, their $q$-series expansion in Eq.~\eqref{eq:tilde-th} for $\nu=0$ can be rewritten as
\be
    {\sf EI}[\theta^{(0)}_{N,k}](\tau) \= \sum_{n=1}^\infty n \, \omega^{(0)}_k(n;N)\,\re^{\pi \ri n^2 \tau/(2N)}\,,
\ee
with $\omega^{(0)}_k(n;N)$ defined in Eq.~\eqref{eq:omega-def}. The arithmetic function $\omega^{(0)}_k(\,\cdot\,;N) : \ZZ \to \CC$ is $2N$-periodic with mean value $\tfrac{1}{N}$, and hence Lemma~\ref{lem:LZ} does not apply. However, following the steps leading to the formula in Eq.~\eqref{eq:c0-LZ} shows that the Eichler integral ${\sf EI}[\theta^{(0)}_{N,k}](\tau)$ has the asymptotic expansion
\begin{equation}\label{A:perturbativeexpansionB-0}
       {\sf EI}[\theta^{(0)}_{N,k}](\tau)-\frac{\ri}{\pi\tau} \sim \tEI[\theta_{N,k}^{(0)}](\tau)=-\frac{4N}{\pi^2\sqrt{\pi}}\sum_{\ell=0}^\infty\left(\frac{2N}{\pi \ri}\right)^{\ell} \Gamma\big(\tfrac{3}{2}+\ell\big)\sum_{m=1}^\infty\frac{\cos(\frac{\pi k m}{N})}{m^{2(\ell+1)}}\, \tau^\ell \, ,
    \end{equation}
as $\tau\to0^+$. The pole term is unavoidable here: the Eichler integral ${\sf EI}[\theta^{(0)}_{N,k}](\ri t)$ grows like $\tfrac{1}{\pi t}$ as $t\to0^+$, whereas the right-hand side above is bounded. This is the same subtraction performed in Lemma~\ref{lem:asympfalsetheta} and later in Proposition~\ref{prop:median}, and it has no counterpart for $\nu=1$.

\medskip

In the rest of this section, we will resume our (modular) resurgent analysis of the asymptotic expansions of ${\sf EI}[\theta_{N,k}^{(\nu)}]$ for $\nu =0,1$ using the expressions computed in Lemma~\ref{lem:asympfalsetheta}.

\subsubsection{Borel transform and Stokes constants}\label{sec:MRS-EI}

The resurgent structures of the formal power series $\tEI[\theta_{N,k}^{(\nu)}]$ in Eq.~\eqref{eq: tilde-Phi} can be computed exactly and shown to satisfy the properties in parts~1 and~2 of Definition~\ref{def:vv-MRS}. Together with Lemma~\ref{lem: decomp-ei}, we prove that they are components of two vector-valued MRSs.

\begin{prop}\label{prop:falsetheta-vvmrs}
    For $\nu=0,1$, the vector of asymptotic series $\tbEI^{(\nu)}_N$, defined in Eq.~\eqref{eq: tilde-Phi}, is a vector-valued modular resurgent series.
\end{prop}
\begin{proof}
    Let us introduce the auxiliary vector $\boldsymbol\omega^{(\nu)}(m;N) := \big(\omega^{(\nu)}_k(m;N)\big)_{k\in\widetilde{G}^{(\nu)}_N}$, $m \in \IZ_{>0}$, so that Eq.~\eqref{eq:L} becomes
    \be
    \pmb{\mathscr{L}}^{(\nu)}_N(s) = \sum_{m=1}^\infty \frac{\boldsymbol\omega^{(\nu)}(m;N)}{m^{2s-\nu}} \, ,
    \ee
    using the folding of Eq.~\eqref{eq:Dirichlet-vec}. 
    The Borel transform of the Gevrey-1 asymptotic vector $\tau^{3/2-\nu}\tbEI^{(\nu)}_N(\tau)$ can be computed from Eq.~\eqref{eq: cj-coeffs} as
    \begin{equation} \label{eq: borel-ei-steps}
        \begin{aligned}
            \mathcal{B}[\tau^{3/2-\nu}\tbEI^{(\nu)}_N](\xi)
            &\= \frac{(-1)^\nu}{\pi} \left(\frac{2N}{\pi \ri}\right)^{3/2-\nu} \xi^{1/2-\nu}\,\Om^{(\nu)}(S)\sum_{\ell=0}^\infty \pmb{\mathscr{L}}^{(\nu)}_N(1+\ell)\left(\frac{2N \xi}{\pi \ri}\right)^{\ell}\\
            &\= \frac{(-1)^\nu}{\pi} \left(\frac{2N}{\pi \ri}\right)^{3/2-\nu} \xi^{1/2-\nu}\,\Om^{(\nu)}(S)\sum_{\ell=0}^\infty \sum_{m=1}^\infty \frac{\boldsymbol\omega^{(\nu)}(m;N)}{m^{2-\nu}}\left(\frac{2N \xi}{\pi \ri m^2}\right)^{\ell} \, ,
        \end{aligned}
    \end{equation}
    where $\xi$ is the Borel variable conjugate to $\tau$, corresponding to $\xi = \zeta^2$ in Definition~\ref{def:vv-MRS}.

    Anticipating the Stokes data emerging from the pole structure below, let us define the \emph{Stokes vectors}
    \begin{equation}\label{eq:StokesS-tot}
        {\bf St}^{(\nu)}_m := 2 \ri (-1)^\nu m^{1-\nu}\,\Om^{(\nu)}(S)\,\boldsymbol\omega^{(\nu)}(m;N)\,, \quad m \in \ZZ_{>0}\,.
    \end{equation}
    Exploiting the convergence of the series in the right-hand side of Eq.~\eqref{eq: borel-ei-steps}, we permute the order of summation, and the Borel transform becomes
    \begin{equation}\label{eq:Borel-nu}
        \mathcal{B}[\tau^{3/2-\nu}\tbEI^{(\nu)}_N](\xi) \= -\frac{1}{2 \pi\ri } \, \left(\frac{2N \xi}{\pi\ri}\right)^{1/2-\nu} \, \sum_{m=1}^\infty \frac{m^{2\nu-1} \,{\bf St}^{(\nu)}_m}{\xi-\tfrac{\pi\ri m^2}{2N}}\,,
    \end{equation}
    which makes the simple-pole structure with Stokes constants ${\bf St}^{(\nu)}_m$ immediately manifest. In the Borel variable $\zeta$ of Definition~\ref{def:vv-MRS} with $\xi=\zeta^2$, the Borel transform displays a tower of simple poles at\footnote{Under the folding $\xi=\zeta^2$, the two singularities at $\rho_{\pm m}$ collapse onto the single point $\eta_m=\CA^2 m^2=\tfrac{\pi\ri}{2N}m^2$, $m\in\ZZ_{>0}$, in the $\xi$-plane.}
    \begin{equation}\label{eq:poles-1}
        \rho_m\=\CA\,m\,,\quad \CA\=\sqrt{\frac{\pi\ri}{2N}}\,,\quad m\in\ZZ_{\neq 0} \, ,
    \end{equation}
    in agreement with part~1 of Definition~\ref{def:vv-MRS}.
    
    The Stokes vectors in Eq.~\eqref{eq:StokesS-tot} assemble into the vector of Dirichlet series
    \begin{equation}\label{eq:L-funct-Stheta}
        \boldsymbol\CL^{(\nu)}_N(s):=\sum_{m=1}^\infty \frac{{\bf St}^{(\nu)}_m}{m^s} \= 2 \ri (-1)^\nu\Om^{(\nu)}(S)\,\boldsymbol{D}^{(\nu)}_N\big(s-1+\nu\big)\,,
    \end{equation}
    where $\boldsymbol{D}^{(\nu)}_N$ is the canonical Dirichlet vector in Eq.~\eqref{eq:Dirichlet-vec} and the argument shift by $-(1-\nu)$ originates from the factor $m^{1-\nu}$ carried by ${\bf St}^{(\nu)}_m$ in Eq.~\eqref{eq:StokesS-tot}. 
    We stress that Eq.~\eqref{eq:L-funct-Stheta} is the Dirichlet series of the \emph{rescaled} residues ${\bf St}^{(\nu)}_m$, \emph{i.e.}, the Stokes constants in the sense of Eq.~\eqref{eq: Stokes0}, which follows the normalisation we use throughout. 
    By Lemma~\ref{lem: decomp-ei}, the \emph{Stokes--Dirichlet vector} $\boldsymbol\CL^{(\nu)}_N$ can be expressed as a linear combination of a set of Dirichlet $L$-functions, thus satisfying in part~3 of Definition~\ref{def:vv-MRS}. Explicitly, let us introduce the notation $J^{(\nu)} := J^{\rm even}_{{\rm all},\,2N}$ for $\nu = 0$ and $J^{(\nu)} := J^{\rm odd}_{{\rm all},\,2N}$ for $\nu = 1$ for simplicity. Applying Eq.~\eqref{eq:decomp_ei2-bis} in vector form yields
    \be
        \boldsymbol\CL^{(\nu)}_N(s) \= {\bf M}^{(\nu)}\,({\bf C}^ {(\nu)})^{1-\nu-s}\,\boldsymbol{L}^{(\nu)}\big(s-1+\nu\big)\, .
    \ee
    Here, using the notation introduced in the proof of Lemma~\ref{lem: decomp-ei}, $\boldsymbol{L}^{(\nu)}(s) := (L_j(s))_{j\in J^{(\nu)}}$ is the vector of Dirichlet $L$-functions $L_j$ of any modulus $m_j$ diving $2N$ with fixed parity, ${\bf C}^{(\nu)}:= \mathrm{diag}_{j\in J^{(\nu)}}(c_j)$ is the diagonal matrix of constants $c_j=2N/m_j$, and
    \be\label{eq:M-mat-Stheta}
        {\bf M}^{(\nu)} := 2 \ri (-1)^\nu \Om^{(\nu)}(S)\,{\bf O}^{(\nu)}
    \ee
    is the decomposition matrix of size $|\widetilde{G}^{(\nu)}_N|\times|J^{(\nu)}|$ analogous to the matrix ${\bf M}$ of Section~\ref{sec: qPochh-vv}, where
    \be
        {\bf O}^{(\nu)} := ({\bf O}_{kj})_{k \in \widetilde{G}^{(\nu)}_N\,,\; j \in J^{(\nu)}}\,.
    \ee

    Finally, a direct count shows that the matrix ${\bf M}^{(\nu)}$ in Eq.~\eqref{eq:M-mat-Stheta} is square, as is the corresponding matrix in Section~\ref{sec: qPochh-vv}.
    Recall that there are $\phi(m)$ Dirichlet characters modulo $m$. The identity $\sum_{m\vert2N} \phi(m) = 2N$ counts all Dirichlet characters whose moduli divide $2N$, partitioned by modulus. For each $m \geq 3$, the even and odd characters split equally, while the unique characters of moduli $1$ and $2$ are both even. 
    As a consequence, 
    \be
    |J^{\rm even}_{{\rm all},\,2N}| = 2 + (2N - 2)/2 = N + 1 \quad \text{ and } \quad |J^{\rm odd}_{{\rm all},\,2N}| = (2N - 2)/2 = N - 1  \, ,
    \ee
    matching the cardinalities $r_0=|\widetilde{G}^{(0)}_N|$ and $r_1=|\widetilde{G}^{(1)}_N|$ from Eq.~\eqref{eq:G-tilde-folded}, respectively. 
\end{proof}

\paragraph{Alternative derivation via Lawrence\texorpdfstring{--}{-}Zagier: Borel transform.}

Applying the well-known property $\Gamma\big(\ell+\frac{1}{2}\big) = \frac{\sqrt{\pi}}{2^{2\ell}} \frac{(2\ell)!}{\ell!}$, $\ell \in \ZZ_{\ge 0}$, the Borel transform of the asymptotic series in Eq.~\eqref{A:perturbativeexpansion}, expressed in the Borel variable $\zeta$ of Definition~\ref{def:vv-MRS} via $\xi = \zeta^2$, becomes
\begin{equation} \label{eq: borel-alternative}
    \mathcal{B}\big[\sqrt{\tau}\, \tEI[\theta_{N,k}^{(1)}](\tau)\big](\zeta)
     \= \sum_{\ell=0}^\infty \frac{a_{k,\ell}}{\Gamma(\ell+\tfrac{1}{2})}\, \zeta^{2\ell-1}
     \=\frac{1}{\zeta\sqrt{\pi}} \sum_{\ell=0}^\infty \frac{L\big(-2\ell,\,\omega^{(1)}_k(\,\cdot\,;N)\big)}{(2\ell)!}\, \big(2\CA\zeta\big)^{2\ell} \, ,
\end{equation}
with $\CA=\sqrt{\pi\ri/(2N)}$ defined in Eq.~\eqref{eq:poles-1}.
This power series admits a closed form, controlled by the meromorphic kernel
\begin{equation}\label{dfn:kernel}
    {\bm K}_{N,k}(x) := \frac{ \sinh\big( (N-k)x \big) }{ \sinh\big( Nx \big) }\,, \quad 0<k<2N\,,
\end{equation}
whose analytic structure we record first. The kernel is regular at $x=0$, decays as $|\Re(x)|\to\infty$, and is holomorphic away from the zeros at $x=\tfrac{\pi\ri m}{N}$, $m\in\ZZ_{\neq0}$, of the denominator, where it has \emph{at most} simple poles, with residues
\be\label{eq:kernel-residues}
    \frac{\sinh\big((N-k)\pi\ri m/N\big)}{N\cosh(\pi\ri m)} \= -\frac{\ri}{N}\,\sin\!\Big(\frac{\pi k m}{N}\Big) \, .
\ee
The pole at a given $m$ is present precisely when this residue is non-zero, that is, when $N \nmid km$. For all remaining values of $m$, the numerator of ${\bm K}_{N,k}$ vanishes as well, and the singularity is removable. In particular, ${\bm K}_{N,N}\equiv0$ has no poles, while for $k$ such that $g=(k,N)>1$ the poles of ${\bm K}_{N,k}$ are absent at $m\in\tfrac{N}{g}\ZZ$. Since $\sin\big(\tfrac{\pi k m}{N}\big)$ corresponds to the Stokes datum of Eq.~\eqref{eq:StokesS-tot} at index $m$ (up to a non-vanishing prefactor), this is the same condition that makes the corresponding Stokes constants zero, and it is harmless in that all sums below are over the residues in Eq.~\eqref{eq:kernel-residues}, which we allow to vanish.
The kernel in Eq.~\eqref{dfn:kernel} therefore equals its symmetric Mittag--Leffler pole expansion
\begin{equation}\label{eq:mittag-leffler}
    {\bm K}_{N,k}(\pi y)
     \=  -\frac{\ri}{N\pi} \lim_{m_\ast\to\infty} \sum_{m=-m_\ast}^{m_\ast} \frac{\sin\left( \tfrac{\pi km}{N} \right)}{y-\tfrac{\ri m}{N}} \, ,
\end{equation}
or its manifestly even form
\begin{equation}\label{eq:ML-paired}
    {\bm K}_{N,k}(x) \= \frac{2\pi}{N^2}\sum_{m=1}^{\infty} \frac{m\,\sin\!\big(\tfrac{\pi k m}{N}\big)}{x^2+\tfrac{\pi^2 m^2}{N^2}}\,, 
\end{equation}
where we have paired the terms for $m$ and $-m$ using the oddness of the sine function.
We claim that the Borel transform in Eq.~\eqref{eq: borel-alternative} admits the closed form~\cite{Gukov:2016njj, Cheng:2018vpl}
\begin{equation}\label{eq:sihn_xi}
     \mathcal{B}\big[\sqrt{\tau}\, \tEI[\theta_{N,k}^{(1)}](\tau)\big](\zeta) \= \frac{1}{\zeta\sqrt{\pi}}\, {\bm K}_{N,k}(2\CA\zeta)\,, \quad 0<k<2N \, ,
\end{equation}
which we prove in the following two equivalent ways.
 
\emph{1. From the poles (resurgent structure).} Under the change of variable $x=2\CA\zeta$, the poles at $x=\tfrac{\pi\ri}{N} m$ with $m \in \ZZ_{\ne 0}$ are sent to $\zeta=\CA m$ via $\CA^2=\tfrac{\pi\ri}{2N}$, and the kernel pole expansion in Eq.~\eqref{eq:ML-paired} reproduces term by term the Stokes series in Eq.~\eqref{eq:Borel-nu}: the residues of ${\bm K}_{N,k}$ coincide with the Stokes constants of Eq.~\eqref{eq:StokesS-tot}. Since both sides of Eq.~\eqref{eq:sihn_xi} are given by the same convergent pole sum, this proves our claim. As expected, the Borel transform has a two-sided tower of equally spaced simple poles at the points in Eq.~\eqref{eq:poles-1}.

\emph{2. From the coefficients (functional equation).} Expanding each summand of Eq.~\eqref{eq:ML-paired} geometrically for $|\zeta|<\CA$---the expected radius of convergence---and swapping the absolutely convergent sums yields
\be \label{eq: bracket-eq}
    {\bm K}_{N,k}(x) \= \sum_{\ell=0}^\infty\left[\frac{2(-1)^\ell N^{2\ell}}{\pi^{2\ell+1}}\sum_{m=1}^\infty\frac{\sin\big(\tfrac{\pi k m}{N}\big)}{m^{2\ell+1}}\right]x^{2\ell}\,.
\ee
At the same time, comparing Eqs.~\eqref{eq: borel-alternative} and~\eqref{eq:sihn_xi}, the claimed closed form of the Borel transform is equivalent to the generating-function identity
\be
    \sum_{\ell=0}^\infty \frac{L\big(-2\ell,\,\omega^{(1)}_k(\,\cdot\,;N)\big)}{(2\ell)!}\, x^{2\ell} \= {\bm K}_{N,k}(x)\,, \quad x = 2\CA\zeta\,,
\ee
which by Eq.~\eqref{eq: bracket-eq} equals the statement
\be
L\big(-2\ell,\,\omega^{(1)}_k(\,\cdot\,;N)\big) =\frac{2(-1)^\ell N^{2\ell} (2\ell)!}{\pi^{2\ell+1}}\sum_{m=1}^\infty\frac{\sin\big(\tfrac{\pi k m}{N}\big)}{m^{2\ell+1}}\, , \quad \ell \in \ZZ_{>0} \, .
\ee
This is precisely the functional equation of the $L$-series of the odd, period-$2N$, mean-zero function $\omega^{(1)}_k(\,\cdot\,;N)$, evaluated at the negative even integers $-2 \ell$, and follows independently by comparing the two coefficient formulae established for the same asymptotic expansion in Eqs.~\eqref{A:perturbativeexpansion} and~\eqref{eq:c1-LZ} (see proof of Proposition~\ref{prop: EI1-asympt-new}). In this sense, the closed form in Eq.~\eqref{eq:sihn_xi} is the Borel-plane incarnation of the underlying functional equation.

We note that the kernel ${\bm K}_{N,k}$ is the key object governing all structures that follow: its poles carry the Stokes data, its Laplace transforms along rotated rays are the lateral Borel--Laplace sums entering the median resummation of Proposition~\ref{prop:median}, and its Gaussian-integral avatar is the Eichler-type identity in Eq.~\eqref{Eichler-type-identity-1-new} of Section~\ref{sec:bimodular-new}, where the rotation of the integration contour across the poles of ${\bm K}_{N,k}$ gives the quantum modular behaviour of the $q$-series vector $\bEI^{(1)}_N$ a geometric origin.

\subsubsection{The paradigm}\label{subsec:paradigm_theta}

Here, we show that a variant of the modular resurgence paradigm applies to Eichler integrals of unary theta series. To do so, we compute the discontinuities of the formal power series $\tEI[\theta_{N,k}^{(\nu)}]$ in Eq.~\eqref{eq: tilde-Phi}.\footnote{A result similar to Lemma~\ref{lem:disc} is stated in Proposition~3.6 in~\cite{2025arXiv250500799M}.}

\begin{lemma}\label{lem:disc}
    For $\nu=0,1$, the discontinuity of the vector of asymptotic series $\tbEI^{(\nu)}_N$, defined in Eq.~\eqref{eq: tilde-Phi}, is given by
    \begin{equation}\label{eq:disc_Phi}
        {\rm disc}_{\frac{\pi}{2}}\big[\tau^{3/2-\nu}\,\tbEI^{(\nu)}_N\big](\tau) \= (-1)^\nu 2 \ri \, \Om^{(\nu)}(S) \, \bEI^{(\nu)}_N\Big(-\tfrac{1}{\tau}\Big)\,,
    \end{equation}
    with $\Om^{(\nu)}(S)$ the $S$-matrix in Lemma~\ref{lem:mod_theta} and $\bEI^{(\nu)}_N$ the vector of Eichler integrals in Eq.~\eqref{eq:bEI-folded}.
    Moreover, 
    \be \label{eq:disc_Phi-slash}
        \tfrac{1}{2} {\rm disc}_{\frac{\pi}{2}}\big[\tau^{\eta}\,\tbEI^{(\nu)}_N\big](\tau) \= \tau^{\eta}\,\big(\bEI^{(\nu)}_N|_\eta S\big)(\tau) \, , \quad \eta = 3/2 - \nu \, ,
    \ee
    where $|_\eta$ is the weight-$\eta$ slash operator in Eq.~\eqref{eq: slash} with multiplier system $\Om^{(\nu)}$.
\end{lemma}
\begin{proof}
    As shown in the proof of Proposition~\ref{prop:falsetheta-vvmrs}, the Borel transform $\mathcal{B}[\tau^{3/2-\nu}\tbEI^{(\nu)}_N](\xi)$ in Eq.~\eqref{eq:Borel-nu} has only simple poles at $\xi=\tfrac{\pi\ri}{2N}m^2$, $m \in \IZ_{>0}$, with Stokes constants collected in the vectors ${\bf St}^{(\nu)}_m$ of Eq.~\eqref{eq:StokesS-tot}. Summing the corresponding exponentially small contributions, as in Eq.~\eqref{eq: Stokes1-poles}, gives the discontinuity directly in the folded vector form
    \begin{equation} \label{eq:disc_Phi-steps}
       \begin{aligned}
            {\rm disc}_{\frac{\pi}{2}}\big[\tau^{3/2-\nu}\,\tbEI^{(\nu)}_N\big](\tau)
            &\= \sum_{m=1}^\infty {\bf St}^{(\nu)}_m\, \re^{-\frac{\pi\ri m^2}{2N\tau}} \\
            &\= (-1)^\nu 2 \ri\,\Om^{(\nu)}(S)\sum_{m=1}^\infty m^{1-\nu}\,\boldsymbol\omega^{(\nu)}(m;N)\, \re^{-\frac{\pi\ri m^2}{2N\tau}}\\
            &\=(-1)^\nu 2 \ri\,\Om^{(\nu)}(S)\,\bEI^{(\nu)}_N\Big(-\tfrac{1}{\tau} \Big)\,,
       \end{aligned}
    \end{equation}
    where the second and third lines use the definitions of ${\bf St}^{(\nu)}_m$ and $\bEI^{(\nu)}_N$ in Eqs.~\eqref{eq:StokesS-tot} and~\eqref{eq:bEI-folded}, respectively.
    
    Finally, note that the Fricke-type transformation of Eq.~\eqref{eq: disc-f} here combines with the $S$-matrix to produce the modular $S$-transformation in Eq.~\eqref{eq: slash} for the Eichler integral functions. As a result, using $(\Om^{(\nu)}(S))^2 = (-1)^{\nu+1}\ri\,\boldsymbol 1$, Eq.~\eqref{eq:disc_Phi} can be equivalently expressed in the form of Eq.~\eqref{eq:disc_Phi-slash}.
\end{proof}

Therefore, for $\nu=0,1$, the $q$-series vector $\bEI_{N}^{(\nu)}$ in Eq.~\eqref{eq:bEI-folded} satisfies the following variation of the modular resurgence paradigm of Section~\ref{sec:paradigm}. 
In the diagram below, we omit the arguments of the functions to improve readability and adopt the vector notation. In particular, we use the vector of Dirichlet series $\pmb{\mathscr{L}}^{(\nu)}_N$ in Eq.~\eqref{eq:Dirichlet-vec} and the vectors of Stokes constants ${\bf St}^{(\nu)}_{m}$, $m \in \ZZ_{>0}$, in Eq.~\eqref{eq:StokesS-tot}.
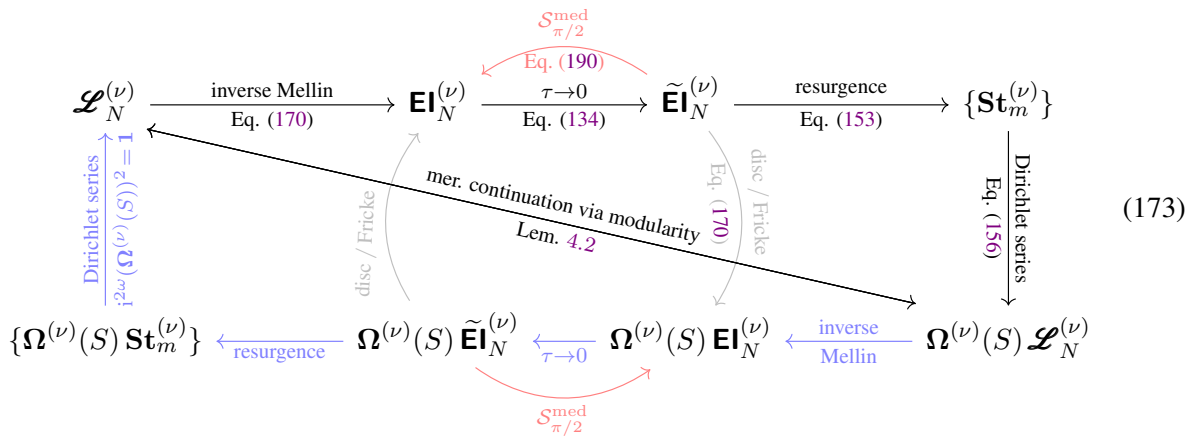
\begin{equation}\label{diag:theta}
\begin{tikzcd}[column sep=2.2em, row sep=2.9em]
    \pmb{\mathscr{L}}^{(\nu)}_N  
    \arrow[ddrrrrr,sloped,"\quad \quad \text{mer. continuation via modularity}", start anchor={[yshift=-1ex]}, end anchor={[yshift=1ex]}]
    \arrow[rr, "\text{inverse Mellin}","\text{Eq.}~\eqref{eq:disc_Phi}" swap]
    &  & \bEI_{N}^{(\nu)} \arrow[r,"\tau \rightarrow 0","\text{Eq}.~\eqref{eq: tilde-Phi}" swap] &  \tbEI_{N}^{(\nu)} \arrow[dd, gray!50, bend left=35, sloped, "\text{disc}\, /\, \text{Fricke} \quad", "\text{Eq.}~\eqref{eq:disc_Phi} \quad"'{yshift=-0.5ex}]
    \arrow[rr,"\text{resurgence}","\text{Eq.}~\eqref{eq:StokesS-tot}" swap]\arrow[l,bend right=35,red!50,"\text{Eq.~\eqref{eq:med-theta}}","\CS^{\rm med}_{\pi/2}" swap]  & & \{{\bf St}_{m}^{(\nu)}\} \arrow[dd,sloped,"\text{Dirichlet series}","\text{Eq.}~\eqref{eq:L-funct-Stheta}" swap] \\  \\
    \{ {\Om}^{(\nu)}(S) \, 
    {\bf St}_{m}^{(\nu)} \} \arrow[uu,blue!50, "\text{Dirichlet series}", "{\ri^{2\omega}(\Om^{(\nu)}(S))^{2}\,=\,\boldsymbol{1}}"',sloped] & &  {\Om}^{(\nu)}(S) \, 
    \tbEI_{N}^{(\nu)} \arrow[uu,gray!50, bend left=35, sloped,
    "\text{disc} \, /\, \text{Fricke} \quad \quad"] \arrow[r,bend right=35,red!50,"\CS^{\rm med}_{\pi/2}", swap] \arrow[ll,blue!50,"\text{resurgence}"] & {\Om}^{(\nu)}(S) \, 
    \bEI_{N}^{(\nu)}  \arrow[l,blue!50,"\tau \rightarrow 0"] & & \Om^{(\nu)}(S)\,\pmb{\mathscr{L}}^{(\nu)}_N
    \arrow[uulllll,sloped, start anchor={[yshift=1ex]}, end anchor={[yshift=-1ex]}]\arrow[ll,blue!50, "\text{inverse}"', "\text{Mellin}"] 
    \arrow[uulllll,sloped,swap,"\quad \quad  \text{Lem.}~\ref{lem:functionalRel-Lfn-theta}", start anchor={[yshift=1ex]}, end anchor={[yshift=-1ex]}]
\end{tikzcd}
\end{equation}
It is worth emphasising that the appearance of the $S$-matrix $\Om^{(\nu)}(S)$ in the diagonal arrow of the diagram above is not an accident: it is the resurgent imprint of the classical modularity of the underlying unary theta series. The chain of implications runs as follows. The vector of theta series $\boldsymbol\theta^{(\nu)}_N$ is a vector-valued modular form for the full modular group $\mathsf{SL}_2(\IZ)$ with $S$-multiplier $\Om^{(\nu)}(S)$ (Lemma~\ref{lem:mod_theta}).
By the classical Hecke--Mellin correspondence, its modular transformation is equivalent to the functional equation of the associated completed Dirichlet vector $\boldsymbol\Lambda^{(\nu)}_N$, namely $\boldsymbol\Lambda^{(\nu)}_N(s) = \ri^\omega\,\Om^{(\nu)}(S)\,\boldsymbol\Lambda^{(\nu)}_N(\omega-s)$ with $\omega=1/2+\nu$ (Lemma~\ref{lem:functionalRel-Lfn-theta}). This can be re-expressed as the functional equation for the theta-Mellin Dirichlet series $\pmb{\mathscr{L}}^{(\nu)}_N$ in Eq.~\eqref{eq: Phi-step2}, that is, 
\begin{equation}
        \pmb{\mathscr{L}}^{(\nu)}_N\left(1-s\right) \=\ri^{\omega} \left(\frac{2N}{\pi} \right)^{-3/2+\nu+2s}\frac{\Gamma\left(s-\tfrac{1}{2}+\nu\right)}{\Gamma\left(1-s\right) } \, \Om^{(\nu)}(S)\, \pmb{\mathscr{L}}^{(\nu)}_N\left(s-\tfrac{1}{2}+\nu\right) \, .
\end{equation}
At the same time, up to the pre-factor $(-1)^\nu 2\ri$, the vector of Dirichlet series of the Stokes constants ${\boldsymbol \CL}_N^{(\nu)}(2s)$},
computed explicitly in Eq.~\eqref{eq:L-funct-Stheta}, is precisely the image $\Om^{(\nu)}(S)\,\pmb{\mathscr{L}}^{(\nu)}_N\left(s-\tfrac{1}{2}+\nu\right)$ appearing on the right-hand side of this functional equation.
The bottom row of the diagram is therefore the $S$-transform of the top row by structural necessity: it does not introduce a genuinely new vector of Dirichlet series, but reproduces the same data viewed through the modular action of $S$.
The fact that the diagram closes is then a consequence of the involution property of the $S$-matrix together with its weight phase: a direct computation shows that 
\be \label{eq: Om-square-id}
\big(\Om^{(\nu)}(S)\big)^2 = (-1)^{\nu+1}\,\ri\,\boldsymbol{1} \, , 
\ee
so that restoring the weight factor $\ri^{2\omega}=(-1)^\nu\ri$ with $\omega=\tfrac12+\nu$, as in Eq.~\eqref{eq:func-L}, returns the loop to the Dirichlet vector $\pmb{\mathscr{L}}^{(\nu)}_N$. Explicitly,
\be
\ri^{2\omega}\big(\Om^{(\nu)}(S)\big)^2\,\pmb{\mathscr{L}}^{(\nu)}_N = \pmb{\mathscr{L}}^{(\nu)}_N \, . 
\ee
Note that $\Om^{(\nu)}(S)$ being of order $8$, equivalently $\big(\Om^{(\nu)}(S)\big)^4 = -\boldsymbol{1}$, is a metaplectic signature: the multiplier represents the double cover $\mathsf{Mp}_2(\IZ)$ (Remark~\ref{rmk: unary-identities}), and restoring the weight-phase $\ri^{2\omega}$ realizes the order-two central element $S^2 = -I$ acting trivially on $\pmb{\mathscr{L}}^{(\nu)}_N$ and closes the loop.

\begin{rmk}
Let $\chi_j$ be a Dirichlet character of modulus~$2N$.
Taking the inverse Mellin transform of its Dirichlet $L$-function $L_j(s)$, up to the elementary kernel $\Gamma(s)\,\big(\tfrac{2N}{\pi}\big)^{s}$ and the substitution $s\mapsto 2s-1+\nu$ as in Eq.~\eqref{eq:mellin_eichler}, produces the $q$-series
\be \label{eq: rmk-ei-dir}
    {\sf EI}[\Theta^{(\nu)}_{N,j}](\tau)= \sum_{n\in \IZ} \mathrm{sgn}(n)\, n^{1-\nu}\,\chi_{j}(n)\, q^{\frac{n^2}{4N}} \,,
\ee
where $\Theta^{(\nu)}_{N,j}(\tau)$ is the Dirichlet theta series of $\chi_j$ and ${\sf EI}[\Theta^{(\nu)}_{N,j}](\tau)$ its Eichler integral.
Under $n \mapsto -n$, the summand picks up a factor $(-1)^{\nu}\chi_j(-1)$, so the series in Eq.~\eqref{eq: rmk-ei-dir} vanishes identically unless $j \in \Jeven_{2N}$ for $\nu=0$ or $j \in \Jodd_{2N}$ for $\nu=1$.
In the nontrivial cases, it reduces to the one-sided series
\be \label{eq: rmk-ei-dir2}
    {\sf EI}[\Theta^{(\nu)}_{N,j}](\tau)= \begin{cases}
        2\sum\limits_{n=1}^\infty n\,\chi_{j}(n)\, q^{\frac{n^2}{4N}} &\mbox{if} \;\; \nu=0 \;\; \mbox{and} \;\; j \in \Jeven_{2N}
        \\ \\
        2\sum\limits_{n=1}^\infty \chi_{j}(n)\, q^{\frac{n^2}{4N}} &\mbox{if} \;\; \nu=1 \;\; \mbox{and} \;\; j \in \Jodd_{2N}
    \end{cases} \, .
\ee
The modular behaviour of these $q$-series is clearest through their finite Fourier dual. Since $\chi_j(n)$ depends only on the congruence class of $n$ modulo $2N$, the Dirichlet theta series $\Theta^{(\nu)}_{N,j}$ is the character-basis counterpart of the unary theta series $\theta^{(\nu)}_{N,k}$ in Eq.~\eqref{eq:theta-un}. Concretely,
\begin{equation}
\Theta^{(\nu)}_{N,j}=\sum_{k\in\ZZ/2N}\chi_j(k)\,\theta^{(\nu)}_{N,k} \, ,
\end{equation} 
whose modular $S$-transformation goes via the Fourier matrix $\widehat\Om^{(\nu)}_{kk'}(S)$ in Eq.~\eqref{eq: Om-tilde}, so that
\be
    \Theta^{(\nu)}_{N,j}\big|_\omega S = \frac{(-1)^\nu}{\sqrt{2\ri N}} \sum_{k'\in\ZZ/2N}\Big(\sum_{k\in\ZZ/2N}\chi_j(k)\,\re^{2\pi\ri \frac{k k'}{2N}}\Big)\theta^{(\nu)}_{N,k'}\,, \quad \omega=1/2+\nu \, .
\ee

Assume now that $\chi_j$ is \emph{primitive}. Then, its Gauss sum is separable at every $k'\in\ZZ/2N$, that is,
\be \label{eq: gauss-id}
    \sum_{k\in\ZZ/2N}\chi_j(k)\,\re^{2\pi\ri \frac{k k'}{2N}}\=\overline{\chi_j}(k')\,\mathscr{G}(\chi_j)\,,
\ee
a property\footnote{For $k'$ coprime to the modulus, the identity in Eq.~\eqref{eq: gauss-id} holds for any $\chi_j$ and is the one used in Eq.~\eqref{eq: FT-omega2}.} that in fact characterizes primitive characters~\cite[Thm.~8.19]{Apostol}. 
The $S$-image under the slash operator is therefore 
\be
    \Theta^{(\nu)}_{N,j}\big|_\omega S = \mathscr{G}(\chi_j) \frac{(-1)^\nu}{\sqrt{2\ri N}} \sum_{k'\in\ZZ/2N}\overline{\chi_j}(k')\, \theta^{(\nu)}_{N,k'} = \mathscr{G}(\chi_j)\,\Theta^{(\nu)}_{N,\bar\jmath} \,,
\ee
where $\bar\jmath$ labels the complex conjugate of the character labelled by $j$. Hence, the $S$-action closes on the family of primitive characters, and each such $\Theta^{(\nu)}_{N,j}$ is a modular form of weight $\omega=1/2+\nu$~\cite{Shimura73, SerreStark}.

When $\chi_j$ is \emph{imprimitive}, separability of the Gauss sum fails precisely on the classes $k' \in G_{2N}$ with $\gcd(k',2N)>1$, and the $S$-image acquires components along the unary theta series $\theta^{(\nu)}_{N,k'}$ supported on the non-invertible residue classes $k'$, which no linear combination of the Dirichlet theta series $\Theta^{(\nu)}_{N,j}$ can reach: the family does not close under $S$ at level $2N$. The $q$-series $\Theta^{(\nu)}_{N,j}$ remains, nevertheless, a weight-$\omega$ modular form on a congruence subgroup of $\mathsf{SL}_2(\ZZ)$ by virtue of the decomposition
\be
    \Theta^{(\nu)}_{N,j}(\tau)\=\sum_{\ell\,\mid\,h_j} \ell^\nu \,\mu(\ell)\,\chi_{j,\ast}(\ell)\,\sum_{n=1}^\infty n^\nu \, \chi_{j,\ast}(n)\,q^{\frac{\ell^2 n^2}{4N}}
\ee
into rescaled theta series of the primitive inducing character $\chi_{j,\ast}$ of modulus $D_j \vert 2N$, with $h_j = 2N/D_j$ and $\mu$ the M\"obius function. What is lost in the imprimitive case is not modularity itself but membership in a vector-valued family at level $2N$.
Restricting to the primitive even (resp.\ odd) characters modulo $2N$, the $q$-series in Eq.~\eqref{eq: rmk-ei-dir2} package into two holomorphic vector-valued quantum modular forms of weight $\eta=3/2-\nu$, whose asymptotic expansions are vector-valued MRSs by the arguments of Section~\ref{sec:MRS-EI}.
\end{rmk}

\subsection{Verification of the conjectures}

We show that Conjectures~\ref{conj:vvMR-conj1} and~\ref{conj:vvMR-conj2} of Section~\ref{sec:conjectures} hold for the $q$-series vectors $\bEI^{(\nu)}_N$ in Eq.~\eqref{eq:bEI-folded} for $\nu=0,1$. Conjecture~\ref{conj:vvMR-conj2} is established via the cocycle computation of Proposition~\ref{prop:qm-ei}, while Conjecture~\ref{conj:vvMR-conj1} follows from the Dirichlet series decomposition in Lemma~\ref{lem: decomp-ei}, which supplies its hypothesis, together with the resummation computation in Proposition~\ref{prop:median}, which supplies its conclusion.

\subsubsection{Quantum modularity}\label{sec:quantum-modularity-EI}

Taking Eichler integrals of modular forms breaks modular invariance, resulting in the construction of quantum modular forms.\footnote{Bringmann and Rolen proved the same statement of Proposition~\ref{prop:qm-ei} for non-holomorphic Eichler integrals and $\tau\in\QQ$~\cite{BR}; Goswami and Osburn also proved the quantum modularity of $\theta^{(1)}_{N,k}$ under the action of congruence subgroups~\cite{GO}.} Here, we prove the (vector-valued) quantum modular behaviour of Eichler integrals of unary theta series.

\begin{prop}\label{prop:qm-ei}
     For $\nu=0,1$, the vector of $q$-series $\bEI^{(\nu)}_N(\tau)$, $\tau \in \IH$, defined in Eq.~\eqref{eq:bEI-folded}, is a vector-valued quantum modular form for $\mathsf{SL}_2(\ZZ)$ of weight $\eta=3/2-\nu$. More precisely, for every $\gamma=\left(\begin{smallmatrix}
        a & b\\
        c & d
    \end{smallmatrix}\right)\in\mathsf{SL}_2(\ZZ)$, the slashed cochain
    \begin{equation}\label{eq:slash-cochain-ei}
        h_\gamma[\bEI^{(\nu)}_N]:=\bEI^{(\nu)}_N|_\eta\gamma-\bEI^{(\nu)}_N \, ,
    \end{equation}
    formed with the weight-$\eta$ slash operator of Eq.~\eqref{eq: slash} and the same multiplier system $\Om^{(\nu)}$ of Eq.~\eqref{eq:mod_trans}, extends holomorphically to the cut plane $\CC_\gamma$ in Eq.~\eqref{eq: C_gamma}.
\end{prop}
\begin{proof}
We discuss the two cases of $\nu=0$ and $\nu=1$ separately.
\begin{itemize}
\item[(1)] The $q$-series ${\sf EI}[\theta^{(0)}_{N,k}](\tau)$ has the following integral representation. Since the kernel $(t-\tau)^{-3/2}$ is not integrable at $t=\tau$, the integral cannot be based at $\tau$ itself; we introduce instead an auxiliary regularization point $x=\tau+\ri\epsilon$, $\epsilon>0$, on the vertical ray above $\tau$, and set
\begin{equation}\label{eq:tilde_int_0}
    {\sf EI}[\theta^{(0)}_{N,k}](\tau)\=\frac{\sqrt{2N\ri}}{2\pi}\,\lim_{x\to\tau}\left[\,2\,\theta^{(0)}_{N,k}(x)\,(x-\tau)^{-1/2}-\int_x^{\ri\infty}\theta^{(0)}_{N,k}(t)\,(t-\tau)^{-3/2}\, dt\,\right],
\end{equation}
where the boundary term subtracts the divergence of the integral as $x \to \tau$ vertically, and the bracket approaches its limit at the rate $O\big((x-\tau)^{1/2}\big)$. Equivalently, integrating by parts, we recover the expression in Eq.~\eqref{eq:tilde_int_0-setting}, that is,
\begin{equation}\label{eq:tilde_int_0-ibp}
    {\sf EI}[\theta^{(0)}_{N,k}](\tau)\=-\frac{\sqrt{2N\ri}}{\pi}\int_\tau^{\ri\infty}\frac{d\theta^{(0)}_{N,k}}{dt}(t)\,(t-\tau)^{-1/2}\, dt\,,
\end{equation}
whose right-hand side converges absolutely and reproduces the $q$-series term-wise.\footnote{Parametrize the vertical ray by $t=\tau+\ri s$ and fix $z_n = \pi n^2/(2N)$ with $n\in \ZZ_{\ne 0}$, so that $\re^{2 \pi \ri n^2 t/(4N)}=q^{n^2/(4N)} \, \re^{-z_n s}$. The right-hand side of Eq.~\eqref{eq:tilde_int_0-ibp} is evaluated using the gamma integral $\int_0^\infty \re^{-z_n s}\,s^{-1/2}\,ds = \Gamma(1/2)\,z_n^{-1/2} = \sqrt{\pi/ z_n}$.
}

By the definition of the slash operator and Lemma~\ref{lem:mod_theta}, the $k$-th component of the cochain for $\gamma=\left(\begin{smallmatrix}
    a & b\\
    c & d
\end{smallmatrix}\right)\in\mathsf{SL}_2(\ZZ)$ is
\begin{equation}\label{eq:cocycle_vvqm-0}
    h_\gamma[\bEI^{(0)}_N]_k(\tau) \= (c\tau+d)^{-3/2}\sum_{k'\in\widetilde{G}^{(0)}_N} (\Om^{(0)}(\gamma))^{-1}_{kk'}\,{\sf EI}[\theta^{(0)}_{N,k'}](\gamma\tau) - {\sf EI}[\theta^{(0)}_{N,k}](\tau)\,.
\end{equation}
Substituting Eq.~\eqref{eq:tilde_int_0} evaluated at $\gamma\tau$ (with regularization point $\gamma(x)$) and at $\tau$ (with regularization point $x$), and applying the modular transformation for $\theta^{(0)}_{N,k}$ in Eq.~\eqref{eq:mod_trans-2}, the multiplier sum collapses via $(\Om^{(0)}(\gamma))^{-1}\Om^{(0)}(\gamma)={\bf 1}$.  The boundary contributions collect into the term $\tfrac{2c}{c\tau+d}\,\theta^{(0)}_{N,k}(x)\,(x-\tau)^{1/2}$, which vanishes in the limit $x\to\tau$, while the integral terms, after the change of variable $t=\gamma(s)$ and using $\gamma^{-1}(\ri\infty)=-d/c$ for $c>0$, combine into the right-hand side of Eq.~\eqref{eq:slash-cocycle-ei-0}.

\item[(2)] Using the integral representation of the $q$-series ${\sf EI}[\theta_{N,k}^{(1)}](\tau)$ in Eq.~\eqref{eq:tilde_int_1}, the cochain
\begin{equation}\label{eq:cocycle_vvqm-1}
    h_\gamma[\bEI^{(1)}_N]_k(\tau) \= (c\tau+d)^{-1/2}\sum_{k'\in\widetilde{G}^{(1)}_N}(\Om^{(1)}(\gamma))^{-1}_{kk'}\,{\sf EI}[\theta^{(1)}_{N,k'}](\gamma\tau) - {\sf EI}[\theta^{(1)}_{N,k}](\tau)
\end{equation}
can be computed analogously to the previous case. Substituting Eq.~\eqref{eq:tilde_int_1} at $\gamma\tau$ and at $\tau$, applying the modular transformation for $\theta^{(1)}_{N,k}$ in Eq.~\eqref{eq:mod_trans-2}, and performing the same change of variable $t=\gamma(s)$ yields the right-hand side of Eq.~\eqref{eq:slash-cocycle-ei-1}.
\end{itemize}

In both cases, the multiplier matrix cancels against its inverse and the resulting expression is manifestly free of $\Om^{(\nu)}(\gamma)$. For $k \in \widetilde{G}^{(\nu)}_N$,
\begin{subequations}\label{eq:slash-cocycle-ei}
\begin{align}
    h_\gamma[\bEI^{(0)}_N]_k(\tau)&=\frac{\sqrt{2N\ri}}{2\pi}\int_{-d/c}^{\ri\infty}\theta^{(0)}_{N,k}(t)\,(t-\tau)^{-3/2}\, dt\,, \label{eq:slash-cocycle-ei-0}\\ 
    h_\gamma[\bEI^{(1)}_N]_k(\tau)&=-\frac{1}{\sqrt{2N\ri}}\int_{-d/c}^{\ri\infty}\theta^{(1)}_{N,k}(t)\,(t-\tau)^{-1/2}\, dt\, \label{eq:slash-cocycle-ei-1}
\end{align}
\end{subequations}
are precisely the components of the cochain $h_\gamma[\bEI^{(\nu)}_N]$ in  Eq.~\eqref{eq:slash-cochain-ei}.

To justify a holomorphic extension to the domain $\CC_\gamma$ of Definition~\ref{def: vv-QMF}, choose $\gamma=\bigl(\begin{smallmatrix}a&b\\c&d\end{smallmatrix}\bigr)\in\mathsf{SL}_2(\IZ)$ with $c>0$. For fixed $t$ on the contour of integration, the kernel $(t-\tau)^{-\eta}$ is defined by continuity along the contour, normalized at the upper end by $\arg(t-\tau)\to\tfrac{\pi}{2}$ as $t\to\ri\infty$, corresponding to the branch $\arg(t-\tau)\in(-\tfrac{\pi}{2},\tfrac{3\pi}{2})$, so that the cut in the variable $t$ lies along the downward ray $\tau-\ri\,\IR_{\geq0}$, and the kernel is then continuous on the contour for every $\tau$ off it. As a function of $\tau$, the branch cut of $(t-\tau)^{-\eta}$ in the convention above is instead the upward ray $t+\ri\,\IR_{\geq0}$, so the straight-contour integral in Eq.~\eqref{eq:slash-cocycle-ei} defines a holomorphic function in $\tau$ on the complement of the vertical ray $L_\gamma:=-d/c+\ri\,\IR_{\geq0}$ (the union of the $\tau$-cuts $t+\ri\,\IR_{\geq0}$ over $t$ on the contour), and it represents the cochain for $\tau$ to the right of $L_\gamma$. 
To continue the integral across the ray, we deform the contour for a fixed $\tau$: since $\boldsymbol\theta^{(\nu)}_N$ is holomorphic on $\IH$ and the integral converges at both endpoints (see below), Cauchy's theorem allows deforming the vertical contour to any path $C\subset\overline\IH$ from $-d/c$ to $\ri\infty$ avoiding $\tau$, with the branch of $(t-\tau)^{-\eta}$ transported along $C$ from the principal value at the upper end; the integral depends only on the position of $\tau$ relative to $C$. Defining $h_\gamma[\bEI^{(\nu)}_N]_k(\tau)$ by the integral over any admissible $C$ leaving $\tau$ to its right yields a single-valued function on $\CC_\gamma=\CC\smallsetminus\RR_{\leq-d/c}$: for $\tau$ in a compact subset thereof, an admissible $C$ can be chosen locally uniformly, and holomorphy follows from the majorants below by Morera's theorem. 
The continuation is obstructed precisely when $\tau$ approaches $\RR_{\le-d/c}$, where every admissible contour is pinched between $\tau$ and the fixed endpoint $-d/c$.

It remains to verify convergence. Near the upper endpoint, as $t \to \ri\infty$, all components of $\boldsymbol\theta^{(\nu)}_N$ with nonzero exponents have Gaussian decay. The only possible constant term occurs in $\theta^{(0)}_{N,0}$; its contribution to the integrand decays as $|t-\tau|^{-3/2}$ and is therefore integrable. Near the lower endpoint, as $t\to-d/c$, Lemma~\ref{lem:mod_theta} gives $\boldsymbol\theta^{(\nu)}_N(t)=O\big(|t+d/c|^{-\omega}\big)$ with $\omega=1/2+\nu$, improved to exponential decay for $\nu=1$ by cuspidality. In either case, the integrand is integrable against the bounded kernel, uniformly for $\tau$ in compacts of $\CC_\gamma$ away from $C$. As admissible contours agree with the vertical contour near the two endpoints, this proves convergence for every admissible contour.

Finally, the case of $c<0$ follows by the same argument, with contour deformations in the opposite direction, yielding a holomorphic extension to $\CC_\gamma=\CC\smallsetminus\RR_{\ge-d/c}$.
\end{proof}

\subsubsection{Median resummation} \label{sed:EffectivenessMedianResum}

By Lemma~\ref{lem: decomp-ei}, the Dirichlet series $\CCL^{(\nu)}_k$ in Eq.~\eqref{eq:L}, that is, the Mellin transforms of the Eichler integrals ${\sf EI}[\theta_{N,k}^{(\nu)}]$, are appropriate linear combinations of a collection of $L$-functions, so the hypothesis of Conjecture~\ref{conj:vvMR-conj1} is satisfied. Here, we prove the effectiveness of the median resummation in reconstructing the $q$-series vectors $\bEI^{(\nu)}_N$ for $\nu=0,1$. Notice that the statement of Conjecture~\ref{conj:vvMR-conj1} applies component-wise.

\begin{prop}\label{prop:median}
For $\nu=0,1$, the vector of Eichler integrals $\bEI^{(\nu)}_N(\tau)$, defined in Eq.~\eqref{eq:bEI-folded}, is reconstructed via median resummation of its asymptotic vector $\tbEI^{(\nu)}_N(\tau)$ as $\tau\to 0$ with $\Im(\tau)>0$. Namely,
\begin{equation}\label{eq:med-theta}
    \mathcal{S}^{\rm med}_{\frac{\pi}{2}}\big[\tbEI^{(\nu)}_N\big](\tau)\= \bEI^{(\nu)}_N(\tau)-\delta_{\nu,0}\,\frac{\ri}{\pi\tau}\,\boldsymbol{1} \, , \quad \tau \in \HH \, ,
\end{equation}
which is intended to hold component-wise for each ${\sf EI}[\theta_{N,k}^{(\nu)}]$ with $k \in \widetilde{G}^{(\nu)}_N$.
\end{prop}

\begin{proof}
Let us begin with $\nu=1$ and take $\theta=\arg(\tau)\in(0,\pi/2)$. The case of $\arg(\tau)\in(\pi/2,\pi)$ is treated analogously at the end.
From the cochain formula in Eq.~\eqref{eq:cocycle_vvqm-1} with $\gamma=S$, evaluated at $-1/\tau$, together with its integral expression in Eq.~\eqref{eq:slash-cocycle-ei-1} and the metaplectic relation $\ri\,(\Om^{(1)}(S))^{2}=-\boldsymbol 1$, we obtain the vector identity\footnote{Note the branch subtlety in relating the period integral to the cocycle: for $\arg(\tau)\in(0,\pi/2)$, the point $-1/\tau$ lies in the second quadrant, where the principal branch of the kernel is continuous along the contour, and the identification acquires a sign relative to the na\"ive substitution $\tau\mapsto-1/\tau$ in Eq.~\eqref{eq:slash-cocycle-ei-1}. This sign change across the ray $\arg(\tau)=\pi/2$ is itself a manifestation of the Stokes phenomenon established below.}
\begin{equation}\label{eq:median-proof1-slash}
    \ri\,\tau^{1/2}\,\bEI^{(1)}_N(\tau)\=\Om^{(1)}(S)\Big[\bEI^{(1)}_N(-1/\tau)-\frac{1}{\sqrt{2N\ri}}\int_{0}^{\ri\infty}\boldsymbol\theta_{N}^{(1)}(t)\,(t+1/\tau)^{-1/2}\, dt\Big]\,,
\end{equation}
where the integral of the vector $\boldsymbol\theta^{(1)}_N$ is understood component-wise. 
For the first term, Lemma~\ref{lem:disc} with weight $\eta=1/2$, written in the explicit form of Eq.~\eqref{eq:disc_Phi}, gives
\begin{equation}\label{eq:median-proof2-slash}
    \Om^{(1)}(S)\,\bEI^{(1)}_N(-1/\tau)\=\frac{\ri}{2}\,{\rm disc}_{\frac{\pi}{2}}\big[\tau^{1/2}\,\tbEI^{(1)}_N\big](\tau)\,.
\end{equation}
For the second term, we use the Gamma-function representation of the kernel
\begin{equation}
    (t+1/\tau)^{-1/2}\=\frac{1}{\sqrt{\pi}}\int_0^\infty v^{-1/2}\,\re^{-(t+1/\tau)v}\,dv\,,
\end{equation}
which converges absolutely along the contour since $\Re(t+1/\tau)=\Re(1/\tau)>0$ for $t\in[0,\ri\infty)$. Exchanging the two integrations\footnote{The resulting sum over $m$ in Eq.~\eqref{eq:median-gamma-kernel} converges conditionally, by the periodicity and mean-zero property of the function $\boldsymbol\omega^{(1)}(\,\cdot\,;N)$. The interchange of integrations is justified by grouping the terms over periods.} and evaluating the inner integral term-wise in the $q$-series expansion of the vector $\boldsymbol\theta^{(1)}_{N}$ via
\begin{equation}
    \int_0^{\ri\infty}\re^{\frac{\pi\ri m^2 t}{2N}}\,\re^{-vt}\,dt\=\frac{1}{v-\eta_m}\,,\quad \eta_m=\frac{\pi\ri m^2}{2N}\,, 
\end{equation}
we obtain that
\begin{equation}\label{eq:median-gamma-kernel}
    \int_0^{\ri\infty}\boldsymbol\theta_{N}^{(1)}(t)\,(t+1/\tau)^{-1/2}\,dt\=\frac{1}{\sqrt{\pi}}\int_0^{\infty}\frac{\re^{-v/\tau}}{\sqrt{v}}\,\sum_{m=1}^\infty\frac{m\,\boldsymbol\omega^{(1)}(m;N)}{v-\eta_m}\,dv\,,
\end{equation}
whose poles sit precisely at the Borel singularities in Eq.~\eqref{eq:poles-1}. The action of the multiplier matrix $\Om^{(1)}(S)$ is now immediate. Recall from Eqs.~\eqref{eq:StokesS-tot} and~\eqref{eq:Borel-nu} that 
\begin{equation}\label{eq:median-borel-id}
    \frac{1}{\sqrt{2N\ri}}\,\frac{1}{\sqrt{\pi v}}\sum_{m=1}^\infty\frac{m \, \Om^{(1)}(S)\,\boldsymbol\omega^{(1)}(m;N)}{v-\eta_m}\=-\ri\;\mathcal{B}\big[\tau^{1/2}\,\tbEI^{(1)}_N\big](v)
\end{equation}
is precisely the partial-fraction expansion of the Borel transform. 
Since the poles at $v=\eta_m$ lie on the positive imaginary axis, the contour $\IR_{\ge0}$ in Eq.~\eqref{eq:median-gamma-kernel} can be rotated to $\re^{\ri\theta}\IR_{\ge0}$ without crossing any singularity, identifying the $v$-integral with the Laplace integral of Eq.~\eqref{eq: Laplace}:
\begin{equation}\label{eq:median-laplace-id}
    \frac{1}{\sqrt{2N\ri}}\,\Om^{(1)}(S)\int_0^{\ri\infty}\boldsymbol\theta_{N}^{(1)}(t)\,(t+1/\tau)^{-1/2}\,dt\=-\ri\,s_\theta\big[\tau^{1/2}\,\tbEI^{(1)}_N\big](\tau)\,,
\end{equation}
where $s_\theta$ acts component-wise. Substituting Eqs.~\eqref{eq:median-proof2-slash} and~\eqref{eq:median-laplace-id} into Eq.~\eqref{eq:median-proof1-slash} and dividing by $\ri\,\tau^{1/2}$, we conclude
\begin{equation}\label{eq:median-proof3-slash}
   \bEI^{(1)}_N(\tau)\=s_\theta\big[\tbEI^{(1)}_N\big](\tau)+\frac{1}{2}\,{\rm disc}_{\frac{\pi}{2}}\big[\tbEI^{(1)}_N\big](\tau)\=\mathcal{S}^{\rm med}_{\frac{\pi}{2}}\big[\tbEI^{(1)}_N\big](\tau)\,,
\end{equation}
where the last equality holds because $s_\theta\big[\tbEI^{(1)}_N\big](\tau)$ for $\arg(\tau)\in(0,\pi/2)$ is the analytic continuation of the Borel--Laplace sum from below the Stokes ray at, and the median resummation is given by Eq.~\eqref{eq: median}.
In the symmetric case of $\arg(\tau)\in(\pi/2,\pi)$, the same computation yields $s_\theta-\tfrac12\,{\rm disc}_{\pi/2}$, which equals $\mathcal{S}^{\rm med}_{\pi/2}$ from above the ray. 

Let us now consider $\nu=0$. The same procedure applies with kernel 
\be
(t+1/\tau)^{-3/2}=\frac{2}{\sqrt\pi}\int_0^\infty v^{1/2}\re^{-(t+1/\tau)v} dv \, , 
\ee
using the vector $\boldsymbol\omega^{(0)}(m;N)$ and the partial-fraction expansion of the Borel trasform from Eqs.~\eqref{eq:StokesS-tot} and~\eqref{eq:Borel-nu}, establishing Eq.~\eqref{eq:med-theta} in vector form for $\nu=0$. 
The only crucial difference lies in the constant term of the theta vector $\boldsymbol\theta^{(0)}_N(t)={\bf e}_0+\big(\boldsymbol\theta^{(0)}_N(t)-{\bf e}_0\big)$ with ${\bf e}_0$ the indicator of the component $k=0$, reflecting the non-cuspidality of $\boldsymbol\theta^{(0)}_N$. Its inner Laplace transform $\int_0^{\ri\infty}\re^{-vt}\,dt=v^{-1}$ contributes at the \emph{origin} of the Borel plane, rather than at the simple poles at $\eta_m$. Explicitly, 
\be
\frac{2}{\sqrt\pi}\int_0^\infty v^{1/2}\,\re^{-v/\tau}\,v^{-1}\,dv=2\tau^{1/2} \, , 
\ee
and, after acting with the $S$-matrix, using $\Om^{(0)}(S)\,{\bf e}_0=(2\ri N)^{-1/2}\,\boldsymbol{1}$, we obtain that
\begin{equation}\label{eq:median2-origin}
    \frac{\sqrt{2N\ri}}{2\pi}\,\Om^{(0)}(S)\int_0^{\ri\infty}{\bf e}_0\,(t+1/\tau)^{-3/2}\,dt\=\frac{\tau^{1/2}}{\pi}\,\boldsymbol{1}\,,
\end{equation}
which is precisely $-\ri\,\tau^{3/2}$ times the singular elementary term $\tfrac{\ri}{\pi\tau}\boldsymbol 1$ of Lemma~\ref{lem:asympfalsetheta}. The cuspidal part is evaluated as in the $\nu=1$ case.
\end{proof}

\subsection{Bimodular completions of Eichler integrals}\label{sec:bimodular-new}

The quantum modularity of the Eichler integrals ${\sf EI}[\theta_{N,k}^{(\nu)}](\tau)$ in Eq.~\eqref{eq:tilde-th}, as established in Proposition~\ref{prop:qm-ei} via the computation of the cocycle, is due to the invariance of the upper integration limit at $\ri\infty$ in the integral representations in Eqs.~\eqref{eq:tilde_int_0-setting} and~\eqref{eq:tilde_int_1}. The same failure of modularity admits a complementary, geometric interpretation: it can be traced to the motion of an integration contour across the poles of a meromorphic kernel under the action of the modular group. 
To exhibit this mechanism in the case $\nu=1$, we recast the relevant Borel--Laplace sum as a Gaussian integral against the $\sinh$-ratio kernel of Eq.~\eqref{eq:sihn_xi} and follow how the contour must be deformed under the modular $S$-transformation.

More precisely, consider the Eichler integral identity
\begin{equation}\label{Eichler-type-identity-1-new}
     \int_0^\infty \frac{\re^{-y^2/\tau}}{\sqrt{\tau}} {\bm K}_{N,k}(\pi y) \, dy
     =
     \sqrt{\frac{\pi}{8N}} \int_{0}^\infty \frac{\theta^{(1)}_{N,k}(\ri u)}{\sqrt{u+\frac{\pi N \tau}{2}}}\,du
     \, , \quad N\in\mathbb{Z}_{>0}\, , \; 0<k<2N\,,
\end{equation}
where $\theta^{(1)}_{N,k}(\ri u) = \sum_{m\equiv k\,(2N)} m\,\re^{-\pi u m^2/(2N)}$ denotes the restriction of the unary theta series to the positive imaginary axis. The identity holds for $\tau\in\IR_{>0}$ and extends by analytic continuation to $\Re(\tau)>0$, both sides being analytic there. It is derived from the Zwegers identity~\cite{zwegers}
\begin{equation}\label{Eichler-type-identity-2-new}
    \int_{-\infty}^\infty \frac{\re^{-\pi t y^2}}{y-\mathrm{i} \kappa} dy = \pi \mathrm{i}  \int_0^\infty \frac{\kappa\, \re^{-\pi \kappa^2 u}}{\sqrt{u+t}} du ~, \quad \kappa\in\mathbb{R}_{\ne 0}\,,
\end{equation}
together with the Mittag--Leffler expansion in Eq.~\eqref{eq:mittag-leffler} and, crucially, the modular transformation of the theta series (Lemma~\ref{lem:mod_theta}), as follows.

To see this, recall the even pole expansion of the kernel ${\bm K}_{N,k}$ in Eq.~\eqref{eq:ML-paired} at the end of Section~\ref{sec:MRS-EI}.
Note that the symmetric truncation $\lim_{m_\ast\to\infty}\sum_{m=-m_\ast}^{m_\ast}$ in Eq.~\eqref{eq:mittag-leffler} is what renders this pairing legitimate, the unpaired series being only conditionally convergent. Similarly, since the Gaussian factor in Eq.~\eqref{Eichler-type-identity-2-new} is even in $y$, only the even part of the pole $\tfrac{1}{y-\ri\kappa}$ contributes, and the Zwegers identity is equivalent to
\begin{equation}\label{eq:zwegers-even}
    \int_{-\infty}^\infty \frac{\re^{-\pi t y^2}}{y^2+\kappa^2}\, dy \= \pi \int_0^\infty \frac{\re^{-\pi \kappa^2 u}}{\sqrt{u+t}}\, du \, , \quad \kappa\in\IR_{>0} \, .
\end{equation}
We now extend the integral on the left-hand side of Eq.~\eqref{Eichler-type-identity-1-new} to the full real line by evenness, substitute the expansion in Eq.~\eqref{eq:ML-paired}, and apply Eq.~\eqref{eq:zwegers-even} term-wise to each pole pair, with $t=\tfrac{1}{\pi\tau}$ and $\kappa=\tfrac{m}{N}$; after the rescaling $u \mapsto \tfrac{N}{2}u$, which normalises the Gaussians to the theta convention, this leads to
\begin{equation}\label{eq:mordell-intermediate}
    \int_0^\infty \frac{\re^{-y^2/\tau}}{\sqrt{\tau}} {\bm K}_{N,k}(\pi y) \, dy 
    \=
    \frac{1}{\sqrt{2N^3\tau}} \int_0^\infty \frac{\vartheta_{N,k}^{(1)}(u)}{\sqrt{u+\frac{2}{\pi N \tau}}}\, du \, ,
\end{equation}
where the interchange of summation and integration is justified by the absolute convergence, for every $u>0$, of the series
\begin{equation}\label{eq:sine-theta}
    \vartheta_{N,k}^{(1)}(u) := \sum_{m=1}^{\infty} m \, \sin\!\big(\tfrac{\pi k m}{N}\big) \, \re^{-\frac{\pi u m^2}{2N}} \= \frac{1}{2} \sum_{k'\in\ZZ/2N} \sin\!\big(\tfrac{\pi k k'}{N}\big)\, \theta^{(1)}_{N,k'}(\ri u) \, .
\end{equation}
The second equality in Eq.~\eqref{eq:sine-theta} exposes the structure of the intermediate result: the pole expansion does \emph{not} directly produce the single component $\theta^{(1)}_{N,k}$, but rather the sine-weighted combination of all components, that is, up to normalisation, the $k$-th component of the action of the $S$-matrix $\Om^{(1)}(S)$ of Eq.~\eqref{eq: Om} on the theta vector. It is at this point that classical modularity enters: by Lemma~\ref{lem:mod_theta}, equivalently by Poisson summation, the sine-weighted combination folds back to a single component at the inverted argument, so that
\begin{equation}\label{eq:sine-theta-poisson}
    \vartheta_{N,k}^{(1)}(u) \= \sqrt{\frac{N}{2}}\; u^{-3/2}\, \theta^{(1)}_{N,k}(\ri/u) \, .
\end{equation}
Substituting Eq.~\eqref{eq:sine-theta-poisson} into Eq.~\eqref{eq:mordell-intermediate} and changing variables $u\mapsto 1/u$ then leads directly to Eq.~\eqref{Eichler-type-identity-1-new}, the remaining prefactors combining into $\sqrt{\pi/8N}$.

This derivation makes the geometric mechanism behind quantum modularity transparent. The identity in Eq.~\eqref{Eichler-type-identity-1-new} closes only because of the modular transformation in Eq.~\eqref{eq:sine-theta-poisson} of the unary theta series: the $S$-matrix intertwining the two sides is the same one that closes the modular resurgence diagram in Eq.~\eqref{diag:theta}. Moreover, under the modular $S$-transformation the Gaussian factor $\re^{-y^2/\tau}$ acquires the opposite decay behaviour along the real axis, and the $y$-contour on the left-hand side of Eq.~\eqref{Eichler-type-identity-1-new} must be rotated in the complex plane to preserve convergence. In doing so, it sweeps across the simple poles $y = \ri m/N$ of the kernel, picking up residues proportional to $\sin\big(\tfrac{\pi k m}{N}\big)$---these residues are precisely the obstruction to classical modularity and determine the quantum modular behaviour of the Eichler integrals. However, this failure of modularity can be overcome by considering a bimodular completion, as explained below.

\paragraph{Bimodular completion.}

Following~\cite{Bringmann2019} (see also~\cite{CCR-ongoing, CCR-ongoing2}), the weight-$k$ holomorphic Eichler integral of a weight-$(2-k)$ cusp form $g(\tau)$ admits a bimodular completion of weight $(k,0)$ defined as
\begin{equation} \label{eq:completed-Eichler-integral-new}
    \int_\tau^{\tilde\tau} g(\tau')(\tau'-\tau)^{-k}d\tau'\,,
\end{equation}
where the variable $\tilde\tau\in\mathbb{H}$ representing the upper limit of integration transforms under modular transformations like $\tau$. In other words, Eq.~\eqref{eq:completed-Eichler-integral-new} is a modular form of weight $(k,0)$ with respect to both parameters $\tau$ and $\tilde\tau$, which gives a {\it completion} of the Eichler integral of $g(\tau')$. 

The analogue of the identity in Eq.~\eqref{Eichler-type-identity-2-new} for the bimodular Eichler integral completion is given by the following proposition. 

\begin{prop}\label{Bimod-prop-new}
    Let $t,k\in\mathbb{R}$ with $t>0$, $k\neq 0$, and $\tau,\,\tilde\tau\in\mathbb{H}$ such that the path of integration from $\tau$ to $\tilde\tau$ is well-defined and convergent. Then the identity 
    \begin{equation}
        \label{Eichler-type-identity-3-new}
    \int_{-\infty}^\infty \frac{\re^{-\pi k^2 \tau -\pi (t+\tau) y^2} -  \re^{-\pi k^2 \tilde\tau -\pi (t+\tilde\tau) y^2}}{y-\mathrm{i} k} dy 
    = 
    \pi \mathrm{i}  \int_\tau^{\tilde\tau} \frac{k \, \re^{-\pi k^2 u}}{\sqrt{u+t}} du
    \end{equation}
    generalises the identity in Eq.~\eqref{Eichler-type-identity-2-new}.
\end{prop}

\begin{proof}
Using the standard Gaussian integral representation for the inverse square root 
\begin{equation}
    \frac{1}{\sqrt{u+t}} = \int_{-\infty}^\infty \re^{-\pi (u+t)y^2} dy ~,
\end{equation}
the right-hand side of Eq.~\eqref{Eichler-type-identity-3-new} is 
\begin{equation}
    \pi \mathrm{i}  \int_\tau^{\tilde\tau} \frac{k \, \re^{-\pi k^2 u}}{\sqrt{u+t}} du 
    =
    \pi \mathrm{i} k \int_\tau^{\tilde\tau}  \, \re^{-\pi k^2 u} \, \left( \int_{-\infty}^\infty \re^{-\pi (u+t)y^2} \, dy \right) \, du ~.
\end{equation}
The integrand is Lebesgue integrable and, by Fubini’s theorem, we can swap the order of integration 
\begin{equation}
    \pi \mathrm{i}  \int_\tau^{\tilde\tau} \frac{k \, \re^{-\pi k^2 u}}{\sqrt{u+t}} du 
    =
    \pi \mathrm{i} k \int_{-\infty}^\infty \, \re^{-\pi t y^2} \, \left( \int_\tau^{\tilde\tau}  \re^{-\pi (y^2+k^2)u} \, du \right) \, dy ~,
\end{equation}
evaluate the inner integral with respect to $u$ analytically
\begin{equation}
    \int_\tau^{\tilde\tau}  \re^{-\pi (y^2+k^2)u} \, du = \frac{ \re^{-\pi(y^2+k^2)\tau} - \re^{-\pi(y^2+k^2)\tilde\tau} }{ \pi (y^2+k^2) } ~,
\end{equation}
and find 
\begin{equation}
    \pi \mathrm{i}  \int_\tau^{\tilde\tau} \frac{k \, \re^{-\pi k^2 u}}{\sqrt{u+t}} du 
    =
    \int_{-\infty}^\infty \, \frac{ \mathrm{i} k }{ y^2+k^2 } \, \left( \re^{-\pi (t+\tau) y^2-\pi k^2\tau} - \re^{-\pi (t+\tilde\tau) y^2-\pi k^2 \tilde\tau} \right) \, dy ~.
\end{equation}
Now by substituting $\tfrac{ \mathrm{i} k }{y^2+k^2}= \tfrac{1}{y-\mathrm{i}k} - \tfrac{y}{y^2+k^2}$ on the right-hand side, and noting that the integral  
\begin{equation}
    \int_{-\infty}^\infty \, \frac{ y }{ y^2+k^2 } \, \left( \re^{-\pi (t+\tau) y^2-\pi k^2\tau} - \re^{-\pi (t+\tilde\tau) y^2-\pi k^2 \tilde\tau} \right) \, dy = 0
\end{equation}
vanishes because the expression inside the parentheses is an even function of $y$ and the full integrand is an odd function on a symmetric interval, we arrive directly at the identity in Eq.~\eqref{Eichler-type-identity-3-new}.
\end{proof}

Now remark that on the left-hand side of Eq.~\eqref{Eichler-type-identity-3-new}, the point $y=\mathrm{i}k$ is a simple pole with \emph{zero} residue 
\begin{equation}
    \mathrm{Res}_{y=\mathrm{i}k} \frac{\re^{-\pi k^2 \tau -\pi (t+\tau) y^2} -  \re^{-\pi k^2 \tilde\tau -\pi (t+\tilde\tau) y^2}}{y-\mathrm{i} k} = \re^{\pi t k^2} - \re^{\pi t k^2} = 0
\end{equation}
and therefore a removable singularity. Under the $S$-transformation, the Gaussian factor $\re^{-\pi(y^2+k^2)\tau}$ becomes $\re^{\pi(y^2+k^2)/\tau}$ with opposite decay behaviour along the real axis, and similarly for $\re^{-\pi(y^2+k^2)\tilde\tau}$. To maintain convergence, the $y$-contour must be rotated into a region of the complex plane where the new exponent decays. For the Eichler integral of $\theta^{(1)}_{N,k}(\tau)$, where $\tilde\tau$ is pinned at infinity, working term-wise in the $q$-series expansion of the theta function, this rotation sweeps across the imaginary-axis pole $y=\mathrm{i}k$ and picks up the residue $\re^{\pi t k^2}$ of Eq.~\eqref{Eichler-type-identity-3-new}, thereby determining the modular anomaly of the Eichler integral. For the bimodular Eichler integral, on the contrary, this residue is cancelled by the $\tilde\tau$-dependent contribution, ensuring modularity. The distinction between a genuine simple pole and a removable singularity is therefore not a deformation of a single object, but a consequence of removing one of two cancelling pieces which exist for any $\tilde\tau$ under the conditions of Eq.~\eqref{eq:completed-Eichler-integral-new}.

\subsection{Beyond unary theta series}\label{sec:beyond-unary}

The proofs of Sections~\ref{subsec:modularresugence}--\ref{sed:EffectivenessMedianResum} use remarkably little about unary theta series beyond their classical modularity. In this subsection, we outline how the paradigm extends to the Eichler integral of an arbitrary vector-valued cusp form of half-integral weight $\kappa\in\{1/2,3/2\}$---hence, in particular, to the shadow of any vector-valued mock modular form of weight $2-\kappa\in\{3/2,1/2\}$. Throughout, we indicate for each statement the relevant proof from Sections~\ref{sec:mod-res-vect} and~\ref{sec:eichler}, together with the substitutions required.
Finally, we highlight in Remark~\ref{rmk:beyond-caveats-v2} the one structural feature of Sections~\ref{subsec:modularresugence}--\ref{sed:EffectivenessMedianResum} that does \emph{not} generalize here.

\paragraph{Setup.} Let $\boldsymbol{\varrho}\colon\mathsf{Mp}_2(\IZ)\to\mathsf{GL}_r(\CC)$ be a unitary multiplier system with $\boldsymbol{\varrho}(T)=\mathrm{diag}\big(\re^{2\pi\ri\alpha_1},\ldots,\re^{2\pi\ri\alpha_r}\big)$ of finite order, so that $\alpha_k\in\QQ\cap[0,1)$, and let ${\bf g}=(g_1,\ldots,g_r)\colon\IH\to\CC^r$ be a vector-valued cusp form of weight $\kappa\in\{1/2,3/2\}$ and multiplier $\boldsymbol{\varrho}$, with Fourier expansions
\be\label{eq:g-fourier-general-v2}
    g_k(\tau)\=\sum_{\substack{\lambda\in\alpha_k+\IZ_{\geq0} \\ \lambda>0}} a_k(\lambda)\,q^{\lambda}\,,\quad a_k(\lambda)=O\big(\lambda^{\kappa/2}\big)\,,
\ee
the coefficient bound being the trivial Hecke estimate for cusp forms. As before, we take $q=\re^{2 \pi \ri \tau}$.
We introduce ${\bf a}:=(a_1,\ldots,a_r)$, whose polynomial growth in $\lambda$ will be exploited below, and denote by $D\in\IZ_{>0}$ the common denominator of the set of numbers $\alpha_k$, so that every exponent in Eq.~\eqref{eq:g-fourier-general-v2} satisfies $\lambda\in\frac{1}{D}\IZ_{>0}$. The vector of Eichler integrals $\mathsf{EI}[{\bf g}](\tau):=\big({\sf EI}[g_k](\tau)\big)_{k}$ and the vector of Dirichlet series $\pmb{\CCL}(s):=\big(\CCL_k(s)\big)_k$ are defined component-wise by
\be\label{eq:EI-general-v2}
    {\sf EI}[g_k](\tau):=\sum_{\lambda} \lambda^{1-\kappa}\,a_k(\lambda)\,q^{\lambda}\,,\quad
    \CCL_k(s):=\sum_{\lambda}\frac{a_k(\lambda)}{\lambda^{s}}\,,
\ee
generalizing Eqs.~\eqref{eq:EI-intro} and~\eqref{eq:L}. For $\kappa=3/2$, one has the convergent integral representation 
\be \label{eq: int-rep-g}
{\sf EI}[g_k](\tau)\propto\int_\tau^{\ri\infty}g_k(t)\,(t-\tau)^{-1/2}dt \, , 
\ee
analogous to Eq.~\eqref{eq:tilde_int_1}; for $\kappa=1/2$, the same representation holds with kernel $(t-\tau)^{-3/2}$ in a regularized form, analogous to Eq.~\eqref{eq:tilde_int_0}.

We observe that under the dictionary
\be\label{eq:dictionary-general-v2}
\ba
    {\bf g}=\boldsymbol\theta^{(\nu)}_N\,,\quad \boldsymbol\varrho =\Om^{(\nu)} \,, \quad\kappa=\omega=\tfrac12+\nu\,,\\
    D=4N \, , \quad \lambda=\tfrac{m^2}{4N}\,,\quad a_k(\lambda)=m^{\nu}\,\omega^{(\nu)}_k(m;N)\,,
\ea
\ee
with $m \in \ZZ_{>0}$, one has 
\begin{equation} 
{\sf EI}[g_k]=(4N)^{\nu-\frac12}\,{\sf EI}[\theta^{(\nu)}_{N,k}] \quad {\rm{and}} \quad \CCL_k(s)=(4N)^{s}\,\CCL^{(\nu)}_k(s)~,
\end{equation}
whence the completed series of Step~1 below coincides \emph{on the nose} with ${\boldsymbol \Lambda}^{(\nu)}_N$ of Eq.~\eqref{eq:Lambda-def}. All identities below therefore reduce to those of Sections~\ref{subsec:modularresugence}--\ref{sed:EffectivenessMedianResum} (up to the recorded powers of $4N$), for $\nu=1$ as a special case of the present setup and for $\nu=0$ in the sense of the following remark.

\begin{rmk} \label{rmk: theta-0-rmk}
Restricting to cusp forms keeps the discussion clean, at the cost of excluding one of the two families of Section~\ref{sec:eichler}. The case $\nu=1$, $\kappa=3/2$, is a genuine instance of the present setup, since $\theta^{(1)}_{N,k}$ has no constant term. The case $\nu=0$, $\kappa=1/2$, strictly speaking, is not: the component $\theta^{(0)}_{N,0}$ has constant term $1$, so the vector $\boldsymbol\theta^{(0)}_N$ is not cuspidal, and it has been treated separately in Sections~\ref{subsec:modularresugence}--\ref{sed:EffectivenessMedianResum}. The two steps at which the difference is realized are already visible there. Firstly, $\boldsymbol\Lambda^{(0)}_N(s)$ acquires simple poles at $s=0$ and $s=\tfrac12$, interchanged by the functional equation, so that $\int_0^\infty\theta^{(0)}_{N,0}(\ri t)\,t^{s-1}dt$ diverges for every $s$ and the constant term has to be split off in the proof of Lemma~\ref{lem:functionalRel-Lfn-theta}. Secondly, the Mellin inversion in the proof of Lemma~\ref{lem:asympfalsetheta} picks up an additional pole at $s=1$, producing the elementary term $\tfrac{\ri}{\pi\tau}\boldsymbol1$ that is subtracted there, in Eq.~\eqref{A:perturbativeexpansionB-0} and in Eq.~\eqref{eq:med-theta}. Apart from having to carry that term along, Steps 1--7 here apply to $\boldsymbol\theta^{(0)}_N$ unchanged.
\end{rmk}

\paragraph{Step 1: functional equation via the vector-valued Hecke correspondence.} The completed Dirichlet vector $\boldsymbol{\Lambda}(s):=(2\pi)^{-s}\,\Gamma(s)\,\pmb{\CCL}(s)$ extends to an entire function of finite order, bounded on vertical strips, and satisfies
\be\label{eq:func-general-v2}
    \boldsymbol{\Lambda}(s)\=\ri^{\kappa}\,\boldsymbol\varrho(S)\,\boldsymbol{\Lambda}(\kappa-s)\,, \quad s \in \CC \, .
\ee
The proof of Lemma~\ref{lem:functionalRel-Lfn-theta} applies: one splits the Mellin transform of ${\bf g}$ along the imaginary axis at the $S$-fixed point and applies the modular transformation ${\bf g}|_\kappa S={\bf g}$ to the lower segment; cuspidality at the unique cusp guarantees entireness. Since $\boldsymbol\Lambda$ is entire and $\Gamma$ has no zeros, $\pmb{\CCL}$ is entire as well. This is the meromorphic continuation of the Dirichlet vector ``induced by modularity''.

\paragraph{Step 2: asymptotic expansion.} The Mellin transform of ${\sf EI}[g_k](\ri t)$ equals $(2\pi)^{-s}\Gamma(s)\,\CCL_k(s+\kappa-1)$, as in Eq.~\eqref{eq:mellin_eichler}. By Step~1, its only poles are those of $\Gamma(s)$, at $s=-\ell$ with $\ell\in\IZ_{\geq0}$.
Exactly as in Lemma~\ref{lem:asympfalsetheta}, Mellin inversion together with the functional equation in Eq.~\eqref{eq:func-general-v2} then give the vector of Gevrey-$1$ asymptotic expansions $\tEI[{\bf g}](\tau)=\sum_{\ell=0}^\infty {\bf c}_\ell\,\tau^\ell$ as $\tau\to0$ with $\Im(\tau)>0$, where
\be\label{eq:cl-general-v2}
    {\bf c}_\ell\=\mathfrak{c}_\kappa\;\Gamma(2-\kappa+\ell)\,(2\pi\ri)^{-\ell}\, \boldsymbol{\varrho}(S)\,\pmb{\CCL}(\ell+1)\,,
    \quad \mathfrak{c}_\kappa:=\frac{\ri^{\kappa}\sin\!\big(\pi(\kappa-1)\big)}{\pi}\,(2\pi)^{\kappa-2}\,,
\ee
the analogue of Eq.~\eqref{eq: cj-coeffs}. 

\paragraph{Step 3: resurgent structure.} As a consequence of Eq.~\eqref{eq:cl-general-v2}, the gamma factors cancel in the Borel transform $\mathcal{B}\big[\tau^{2-\kappa}\,\tEI[{\bf g}]\big]$, and the partial-fraction identity
\be\label{eq:partial-frac-general-v2}
    \sum_{\ell=0}^\infty \pmb{\CCL}(\ell+1)\,x^\ell\=\sum_{\lambda}\frac{{\bf a}(\lambda)\,\lambda^{-1}}{1-x/\lambda}
\ee
yields the closed form
\be\label{eq:Borel-general-v2}
    \mathcal{B}\big[\tau^{2-\kappa}\,\tEI[{\bf g}]\big](\xi)
    \= \mathfrak{c}_\kappa\;\xi^{1-\kappa}\,\boldsymbol\varrho(S)\sum_{\lambda}\frac{{\bf a}(\lambda)\,\lambda^{-1}}{1-\xi/(2\pi\ri\lambda)} 
    \= -\,2\pi\ri\,\mathfrak{c}_\kappa\;\xi^{1-\kappa}\,\boldsymbol\varrho(S)\sum_{\lambda}\frac{{\bf a}(\lambda)}{\xi-2\pi\ri\lambda}\,,
\ee
where $\xi$ is the Borel variable conjugate to $\tau$. 
Away from the origin, the function above is meromorphic with only \emph{simple poles}, located at
\be\label{eq:poles-general-v2}
    \xi_\lambda\=2\pi\ri\lambda\;\in\;\CA\,\IZ_{>0}\,,\quad \CA:=\frac{2\pi\ri}{D}\,,
\ee
and with residues the \emph{Stokes vectors}
\be\label{eq:Stokes-general-v2}
    {\bf St}_\lambda\;\propto\;\lambda^{1-\kappa}\,\boldsymbol\varrho(S)\,{\bf a}(\lambda)\,,
\ee
the factor $\lambda^{1-\kappa}$ coming from the prefactor $\xi^{1-\kappa}$ evaluated at $\xi_\lambda$. 
Parts~1 and~2 of Definition~\ref{def:vv-MRS} are therefore verified for the vector of asymptotic series $\tau^{2-\kappa}\tEI[{\bf g}]$, which has trivial secondary resurgent series, and whose Stokes vectors have polynomial growth by Eq.~\eqref{eq:g-fourier-general-v2}. Writing $\lambda=m/D$ with $m\in\IZ_{>0}$, the associated Stokes--Dirichlet vector is
\be\label{eq:Stokes-dirichlet-general-v2}
    \boldsymbol\CL(s):=\sum_{m=1}^\infty\frac{{\bf St}_{m/D}}{m^s}\;\propto\;D^{-s}\,\boldsymbol\varrho(S)\,\pmb{\CCL}(s+\kappa-1)\,,
\ee
whose meromorphic continuation is supplied by Step~1. 
We stress the shift by $\kappa-1$ in the argument of the Dirichlet series, which is the analogue of the one by $\nu-1$ in Eq.~\eqref{eq:L-funct-Stheta}.

Two features of Eq.~\eqref{eq:poles-general-v2} deserve further emphasis, since they are weaker than their counterparts in Section~\ref{sec:MRS-EI}. 
First, the tower of simple poles is \emph{one-sided}: all values of $\lambda$ are positive, so all singularities of the Borel transform lie on the ray with $\arg(\xi)=\pi/2$.
Second, the inclusion $\{ \xi_\lambda \}_{\lambda} \subseteq \CA \ZZ_{>0}$ is in general strict, and different components of the vector $\mathcal{B}\big[\tau^{2-\kappa}\,\tEI[{\bf g}]\big]$ have generally different singular sets, since the exponents in the $q$-series expansion of $g_k(\tau)$ in Eq.~\eqref{eq:g-fourier-general-v2} lie in $\alpha_k+\IZ_{\geq0}$. Thus, one has a tower of simple poles at $\xi=\rho_m=\CA m$, $m\in\IZ_{>0}$, only in the weak sense that the Stokes vectors at $m$ vanish if the simple pole is missing there. Note that both discrepancies are resolved in Section~\ref{sec:MRS-EI}, where $\lambda=m^2/(4N)$, by the substitution $\xi=\zeta^2$, which converts the sparse one-sided tower $\{\CA m^2\}_{m>0}$ of the $\xi$-plane into the equally spaced two-sided tower $\{\pm\sqrt{\CA}\,m\}_{m>0}$ of the $\zeta$-plane in Eq.~\eqref{eq:poles-1}. No such folding exists for general ${\bf g}$. 
The formula in Eq.~\eqref{eq:partial-frac-general-v2} is the generic avatar of the Mittag--Leffler expansion in Eq.~\eqref{eq:mittag-leffler}, although the closed trigonometric form in Eq.~\eqref{eq:sihn_xi} is special to the example of unary theta series.

\paragraph{Step 4: discontinuity and Fricke-type symmetry.} Summing the exponentially small contributions from the poles, as in Lemma~\ref{lem:disc}, we obtain 
\be\label{eq:disc-general-v2}
    {\rm disc}_{\frac{\pi}{2}}\big[\tau^{2-\kappa}\,\tEI[{\bf g}]\big](\tau)\;\propto\;\boldsymbol\varrho(S)\;{\sf EI}[{\bf g}]\Big(-\tfrac{1}{\tau}\Big)\,,
\ee
the identification of the right-hand side following Eqs.\eqref{eq:EI-general-v2} and~\eqref{eq:Stokes-general-v2}.
Using
\be\label{eq:metaplectic-general-v2}
    \ri^{2\kappa}(\boldsymbol\varrho(S))^2\,{\bf g}\={\bf g}\,,
\ee
Eq.~\eqref{eq:disc-general-v2} can be recast in the slash form 
\be \label{eq:disc-general-v2-slash}
\tfrac12{\rm disc}_{\frac{\pi}{2}}\big[\tau^{\eta}\tEI[{\bf g}]\big](\tau) =\tau^{\eta}\,\big({\sf EI}[{\bf g}]|_\eta S \big)(\tau) \, , \quad \eta=2-\kappa \, , 
\ee
analogous to Eq.~\eqref{eq:disc_Phi-slash}, once the proportionality constant in Eq.~\eqref{eq:disc-general-v2} is normalised as in Lemma~\ref{lem:disc}.

\paragraph{Step 5: closure of the diagram.} Because the vector of Dirichlet series of the Stokes constants $\boldsymbol\CL(s)$ is proportional to $D^{-s}\, \boldsymbol\varrho(S)\,\pmb{\CCL}(s+\kappa-1)$ by means of Eq.~\eqref{eq:Stokes-dirichlet-general-v2}, the bottom row of the analogue of the modular resurgence diagram in Eq.~\eqref{diag:theta} is simply the $\boldsymbol\varrho(S)$-image of the top row, and the loop closes through Eq.~\eqref{eq:metaplectic-general-v2}, \emph{i.e.}, \be
\ri^{2\kappa}\, (\boldsymbol\varrho(S))^2\,\pmb{\CCL}=\pmb{\CCL} \, , 
\ee
exactly as in Section~\ref{subsec:paradigm_theta}: the diagram closes upon itself, and no genuinely new function appears.

\paragraph{Step 6: quantum modularity.} The $q$-series vector ${\sf EI}[{\bf g}]$ is a vector-valued holomorphic quantum modular form for $\mathsf{SL}_2(\ZZ)$ of weight $\eta=2-\kappa$ and multiplier system $\boldsymbol\varrho$: the proof of Proposition~\ref{prop:qm-ei} carries over with $\boldsymbol\theta^{(\nu)}_N$ replaced by ${\bf g}$, since it uses only the integral representations above, the modular invariance ${\bf g}|_\kappa\gamma={\bf g}$, and the exponential decay $O\big(\re^{-2\pi\lambda_{\min}\Im(t)}\big)$ of ${\bf g}(t)$ along vertical rays $\Im(t)\to \infty$, which follows from cuspidality. The cocycle is again the period integral
\be\label{eq:cocycle-general-v2}
    h_\gamma[{\sf EI}[{\bf g}]](\tau)\;\propto\;\int_{-d/c}^{\ri\infty}{\bf g}(t)\,(t-\tau)^{\kappa-2}\,dt\,, \quad \gamma=\left(\begin{smallmatrix}
        a & b\\
        c & d
    \end{smallmatrix}\right)\in\mathsf{SL}_2(\ZZ) \, .
\ee
As in Proposition~\ref{prop:qm-ei}, the branch cut of $(t-\tau)^{\kappa-2}$ requires care. Assuming that $c>0$, with the same branch-cut convention and choice of vertical contour, the right-hand side of Eq.~\eqref{eq:cocycle-general-v2} computes $h_\gamma$ only for $\Re(\tau)>-d/c$, and its holomorphic extension to $\CC_\gamma=\CC\smallsetminus\IR_{\leq-d/c}$, required Definition~\ref{def: vv-QMF}, is obtained by deforming the contour to any admissible path from $-d/c$ to $\ri\infty$ avoiding $\tau$ from the same side, which is permitted by holomorphy and cuspidal decay at the endpoints. The case of $c<0$ is treated symmetrically, giving a holomorphic extension to $\CC_\gamma=\CC\smallsetminus\IR_{\geq-d/c}$. 
Hence, Conjecture~\ref{conj:vvMR-conj2} holds.

\paragraph{Step 7: median resummation.} The proof of Proposition~\ref{prop:median} combines three ingredients, each of which generalizes to the current setup, while the branch subtleties discussed there remain unchanged. (i) The $S$-cocycle in Eq.~\eqref{eq:cocycle-general-v2} evaluated at $-1/\tau$,  as in Eq.~\eqref{eq:median-proof1-slash}, together with the metaplectic identity in Eq.~\eqref{eq:metaplectic-general-v2}, analogous to Eq.~\eqref{eq: Om-square-id}. (ii) The discontinuity formula in the slash form of Eq.~\eqref{eq:disc-general-v2-slash}, as in Lemma~\ref{lem:disc}. (iii) The identification of the lateral Borel--Laplace sum with the period integral, which follows from the Gamma-function representation of the kernel $(t+1/\tau)^{\kappa-2}$ and the elementary Laplace transform 
\be
\int_0^{\ri\infty}\re^{2\pi\ri\lambda t}\re^{-vt}dt=\frac{1}{v-2\pi\ri\lambda} \, .
\ee
The resulting integrand is precisely the pole expansion in Eq.~\eqref{eq:Borel-general-v2}. Cuspidality is what makes the third step clean, no contribution arising at the origin of the Borel plane. Hence, 
\be
\CS^{\rm med}_{\tfrac{\pi}{2}}\big[\tEI[{\bf g}]\big](\tau)={\sf EI}[{\bf g}](\tau) \, , 
\ee
which is the conclusion of Conjecture~\ref{conj:vvMR-conj1}. We stress that what Steps~1--6 establish is the \emph{conclusion} of Conjecture~\ref{conj:vvMR-conj1}; its hypothesis, that the Mellin transforms of the $q$-series ${\sf EI}[g_k]$ be linear combinations of $L$-functions in the sense of Definition~\ref{def: Lfunct}, is not available at this level of generality, as discussed in the following remark.

\begin{rmk}\label{rmk:beyond-caveats-v2}
Four caveats delimit the scope of this generalization. First, part~3 of Definition~\ref{def:vv-MRS} is not established by Steps~1--7. What Eq.~\eqref{eq:Stokes-dirichlet-general-v2} provides is a decomposition for the vector of Stokes--Dirichlet series in the shape of Eq.~\eqref{eq: def-Lfunct} with the single scale $c=D$ and with $\pmb{\CCL}$ in place of a vector of $L$-functions. An Euler product, which Definition~\ref{def: Lfunct} requires, would need ${\bf g}$ to be related to Hecke eigenforms. Accordingly, the statement of Step~3 is that the vector of asymptotic series $\tau^{2-\kappa}\tEI[{\bf g}]$ satisfies parts~1 and~2 of Definition~\ref{def:vv-MRS} and that its Stokes--Dirichlet vector $\pmb{\CL}$ admits meromorphic continuation through a functional equation; the arithmetic input of part~3 of Definition~\ref{def:vv-MRS}, that is the existence of an Euler product, is genuinely extra. Second, the passage to the (unfolded) Borel variable $\zeta$ via $\xi=\zeta^2$, which is what makes the tower of Section~\ref{sec:MRS-EI} equally spaced and two-sided, has no analogue for general ${\bf g}$, as explained in Step~3. Third, for $\kappa=1/2$ the apparent generality is illusory: by the Serre--Stark theorem~\cite{SerreStark}, every component of ${\bf g}$ is a linear combination of unary theta series, and one is back, up to rescalings $\tau\mapsto\ell^2\tau$, to the original setting of Section~\ref{sec:eichler}; the cuspidality hypothesis singles out the combinations with vanishing constant term, so that $\boldsymbol\theta^{(0)}_N$ itself is excluded, as recalled in Remark~\ref{rmk: theta-0-rmk}. Therefore, the genuinely new examples occur at $\kappa=3/2$, for which no analogue of the Serre--Stark theorem holds and the space of cusp forms is not spanned by unary theta series.
\end{rmk}

\section{Conclusions}\label{sec:conclusions}

In this paper, we have extended the framework of modular resurgence to the vector-valued setting (Definition~\ref{def:vv-MRS}). 
We have formulated the corresponding paradigm relating the resurgent and arithmetic structures of canonical pairs of vector-valued MRSs (Section~\ref{sec:paradigm}). 
The median resummations of such pairs produce two vectors of $q$-series related by an $S$-type, or Fricke-type more generally, transformation, and the Mellin transform of one such vector reproduces (up to simple prefactors) the Dirichlet series of the Stokes constants of the asymptotic expansion of the other. These Dirichlet series are, in turn, directly related through their meromorphic continuation. 
We have also proposed two conjectures on the summability and quantum modularity of certain vectors of $q$-series with vector-valued modular resurgent asymptotics (Conjectures~\ref{conj:vvMR-conj1} and~\ref{conj:vvMR-conj2}). Namely, such $q$-series are expected to be vector-valued holomorphic quantum modular forms and to be effectively reconstructed by median resummation.
Finally, we have applied the framework and verified both the paradigm and the conjectures explicitly in two large families of examples: vectors of $q$-Pochhammer symbols (Section~\ref{sec: qPochh-vv}) and vectors of Eichler integrals of unary theta series (Section~\ref{sec:eichler}). The latter is the fully worked instance of a general result: as outlined in Section~\ref{sec:beyond-unary}, the same modular resurgent structure---single tower of simple Borel singularities, Stokes constants rotated by the $S$-matrix, meromorphic continuation via the Hecke correspondence, quantum modularity, and median summability---holds for the Eichler integral of any vector-valued cusp form of weight $1/2$ or $3/2$.

The two examples illustrate a structural distinction: in the $q$-Pochhammer case, the modular representation underlying the associated quantum modular form is trivial and the vector-valued framework amounts to a change-of-basis repackaging of the scalar resurgent structure of~\cite{FR25}; in the Eichler-integral case, the modular representation is non-trivial and the vector-valued framework introduced in this paper is strictly necessary. This structural difference can be read off the rows of the diagram in Eq.~\eqref{diag:vvMR-diagram}.
For the vectors of $q$-Pochhammer symbols, the two rows are genuinely independent: they contain two distinct vectors of $q$-series $({\bf g}_N$ and ${\bf f}_N)$, each capturing the non-perturbative content of the other (diagram in Eq.~\eqref{diag:vvMR-diagram-poch}). 
For the vectors of Eichler integrals of unary theta series, by contrast, the second row is forced by the classical modularity of the underlying theta series, via the Hecke theorem, to be a copy of the first row after acting by the $S$-multiplier $(\bEI^{(\nu)}_N$ and $\Om^{(\nu)}(S)\, \bEI^{(\nu)}_N)$. Thus, no genuinely new function appears, and the diagram closes upon itself (diagram in Eq.~\eqref{diag:theta}). The matrix that closes the diagram corresponding to the modular $S$-matrix in this latter case is therefore not a coincidental match but a structural consequence of the underlying modularity. This mechanism is not special to theta series: it operates verbatim for every half-integral weight vector-valued cusp form, which is precisely why the diagram closes for the whole family of Section~\ref{sec:beyond-unary}.
A closely related distinction appears in the quantum modular behaviour of the two examples.  While the vectors of $q$-Pochhammer symbols of Section~\ref{sec: qPochh-vv} are holomorphic quantum modular functions component-wise, the vectors of Eichler integrals of Section~\ref{sec:eichler} are genuine vector-valued holomorphic quantum modular forms with a non-trivial multiplier system. 

A natural next direction of investigation would be to extend the framework beyond the depth-one quantum modular setting explored in this paper. Higher-depth quantum modular forms, and in particular higher-rank false theta functions, arise as characters of higher-rank logarithmic vertex algebras and as quantum invariants of plumbed three-manifolds in the higher-rank case, playing a role structurally analogous to that of the unary false theta functions considered here. 

\section*{Acknowledgements}

We would like to thank Daniele Dorigoni, Eleanor McSpirit, and Mart\'i Rossell\'o for insightful discussions.
The research of M.C. is supported by the Academia Sinica Investigator grant (AS-IA-111-M03) and NWO Vici grant (number VI.C.232.117). 
The research of I.C. is supported by the Horizon Europe MSCA grant No.101204790, \emph{QFT2VOA-DuStRel}.
The research of C.R. is supported by the Huawei Young Talents Program at Institut des Hautes \'Etudes Scientifiques (IHES).

\bibliography{Bibliography.bib}
\bibliographystyle{ytphys.bst}

\end{document}